\pdfoutput=1
\newif\ifpersonal
\RequirePackage[l2tabu,orthodox]{nag} 
\documentclass[12pt,a4paper,reqno]{amsart} 
\usepackage{amsmath,amsthm,amssymb,mathrsfs,mathtools,bm,eucal,tensor} 
\usepackage{microtype,fixltx2e,lmodern} 
\usepackage[utf8]{inputenc} 
\usepackage[T1]{fontenc} 
\usepackage{enumerate,comment,braket,xspace,tikz-cd,csquotes} 
\usepackage[all,cmtip]{xy} 
\usepackage[centering,vscale=0.7,hscale=0.7]{geometry}
\usepackage[hidelinks]{hyperref}
\usepackage[capitalize]{cleveref}

\theoremstyle{plain}
\newtheorem{thm-intro}{Theorem}
\newtheorem{thm}{Theorem}[section]
\newtheorem*{thm*}{Theorem}
\newtheorem{lem}[thm]{Lemma}
\newtheorem{prop}[thm]{Proposition}

\newtheorem{cor}[thm]{Corollary}

\theoremstyle{definition}
\newtheorem{defin}[thm]{Definition}

\theoremstyle{remark}

\newtheorem{rem}[thm]{Remark}
\numberwithin{equation}{section}
\newtheorem{construction}[thm]{Construction}

\ifpersonal
\newcommand*{\personal}[1]{\textcolor[rgb]{0.6,0.6,1}{(Personal: #1)}}
\newcommand*{\todo}[1]{\textcolor{red}{(Todo: #1)}}
\else
\newcommand*{\personal}[1]{\ignorespaces}
\newcommand*{\todo}[1]{\ignorespaces}
\fi

\newcommand{\C}{\mathbb C}

\newcommand{\rL}{\mathrm L}
\newcommand{\rR}{\mathrm R}

\newcommand{\cA}{\mathcal A}

\newcommand{\cC}{\mathcal C}
\newcommand{\cD}{\mathcal D}

\newcommand{\cF}{\mathcal F}

\newcommand{\cG}{\mathcal G}

\newcommand{\cO}{\mathcal O}

\newcommand{\cS}{\mathcal S}
\newcommand{\cT}{\mathcal T}

\newcommand{\cX}{\mathcal X}
\newcommand{\cY}{\mathcal Y}
\newcommand{\cZ}{\mathcal Z}
\DeclareFontFamily{U}{BOONDOX-calo}{\skewchar\font=45 }
\DeclareFontShape{U}{BOONDOX-calo}{m}{n}{<-> s*[1.05] BOONDOX-r-calo}{}
\DeclareFontShape{U}{BOONDOX-calo}{b}{n}{<-> s*[1.05] BOONDOX-b-calo}{}
\DeclareMathAlphabet{\mathcalboondox}{U}{BOONDOX-calo}{m}{n}

\newcommand{\bbA}{\mathbb A}

\newcommand{\bA}{\mathbf A}
\newcommand{\bD}{\mathbf D}
\newcommand{\bP}{\mathbf P}

\newcommand{\bDelta}{\bm{\Delta}}

\makeatletter
\let\save@mathaccent\mathaccent
\newcommand*\if@single[3]{%
	\setbox0\hbox{${\mathaccent"0362{#1}}^H$}%
	\setbox2\hbox{${\mathaccent"0362{\kern0pt#1}}^H$}%
	\ifdim\ht0=\ht2 #3\else #2\fi
}
\newcommand*\rel@kern[1]{\kern#1\dimexpr\macc@kerna}
\newcommand*\widebar[1]{\@ifnextchar^{{\wide@bar{#1}{0}}}{\wide@bar{#1}{1}}}
\newcommand*\wide@bar[2]{\if@single{#1}{\wide@bar@{#1}{#2}{1}}{\wide@bar@{#1}{#2}{2}}}
\newcommand*\wide@bar@[3]{%
	\begingroup
	\def\mathaccent##1##2{%
		\let\mathaccent\save@mathaccent
		\if#32 \let\macc@nucleus\first@char \fi
		\setbox\z@\hbox{$\macc@style{\macc@nucleus}_{}$}%
		\setbox\tw@\hbox{$\macc@style{\macc@nucleus}{}_{}$}%
		\dimen@\wd\tw@
		\advance\dimen@-\wd\z@
		\divide\dimen@ 3
		\@tempdima\wd\tw@
		\advance\@tempdima-\scriptspace
		\divide\@tempdima 10
		\advance\dimen@-\@tempdima
		\ifdim\dimen@>\z@ \dimen@0pt\fi
		\rel@kern{0.6}\kern-\dimen@
		\if#31
		\overline{\rel@kern{-0.6}\kern\dimen@\macc@nucleus\rel@kern{0.4}\kern\dimen@}%
		\advance\dimen@0.4\dimexpr\macc@kerna
		\let\final@kern#2%
		\ifdim\dimen@<\z@ \let\final@kern1\fi
		\if\final@kern1 \kern-\dimen@\fi
		\else
		\overline{\rel@kern{-0.6}\kern\dimen@#1}%
		\fi
	}%
	\macc@depth\@ne
	\let\math@bgroup\@empty \let\math@egroup\macc@set@skewchar
	\mathsurround\z@ \frozen@everymath{\mathgroup\macc@group\relax}%
	\macc@set@skewchar\relax
	\let\mathaccentV\macc@nested@a
	\if#31
	\macc@nested@a\relax111{#1}%
	\else
	\def\gobble@till@marker##1\endmarker{}%
	\futurelet\first@char\gobble@till@marker#1\endmarker
	\ifcat\noexpand\first@char A\else
	\def\first@char{}%
	\fi
	\macc@nested@a\relax111{\first@char}%
	\fi
	\endgroup
}
\makeatother

\newcommand{\tphi}{\widetilde{\phi}}

\newcommand{\PSh}{\mathrm{PSh}}
\newcommand{\Sh}{\mathrm{Sh}}

\newcommand{\Geom}{\mathrm{Geom}}

\newcommand{\infcat}{$\infty$-category\xspace}
\newcommand{\infcats}{$\infty$-categories\xspace}
\newcommand{\infsite}{$\infty$-site\xspace}

\newcommand{\inftopos}{$\infty$-topos\xspace}
\newcommand{\inftopoi}{$\infty$-topoi\xspace}

\newcommand{\rSet}{\mathrm{Set}}

\newcommand{\tauet}{\tau_\mathrm{\acute{e}t}}

\newcommand{\An}{\mathrm{An}}
\newcommand{\Afd}{\mathrm{Afd}}
\newcommand{\Top}{\mathcal T\mathrm{op}}

\newcommand{\dAnk}{\mathrm{dAn}_k}
\newcommand{\Ank}{\mathrm{An}_k}
\newcommand{\cTan}{\cT_{\mathrm{an}}}
\newcommand{\cTank}{\cT_{\mathrm{an}}(k)}

\newcommand{\cTdisck}{\cT_{\mathrm{disc}}(k)}

\newcommand{\cTetk}{\cT_{\mathrm{\acute{e}t}}(k)}

\newcommand{\Strloc}{\mathrm{Str}^\mathrm{loc}}
\newcommand{\RTop}{\tensor*[^\rR]{\Top}{}}
\newcommand{\LTop}{\tensor*[^\rL]{\Top}{}}
\newcommand{\RHTop}{\tensor*[^\rR]{\mathcal{H}\Top}{}}
\newcommand{\LRT}{\mathrm{LRT}}
\newcommand{\Tor}{\mathrm{Tor}}
\newcommand{\dAfd}{\mathrm{dAfd}}
\newcommand{\dAfdk}{\mathrm{dAfd}_k}

\newcommand{\biget}{\mathrm{big,\acute{e}t}}
\newcommand{\trunc}{\mathrm{t}_0}
\newcommand{\Hyp}{\mathrm{Hyp}}
\newcommand{\HSpec}{\mathrm{HSpec}}
\newcommand{\CAlg}{\mathrm{CAlg}}
\newcommand{\trunctopoi}{\Spec^{\cG_{\mathrm{an}}^{\le 0}(k)}_{\cG_{\mathrm{an}(k)}}}

\newcommand{\llp}{(\!(}
\newcommand{\rrp}{)\!)}
\newcommand{\an}{^\mathrm{an}}
\newcommand{\alg}{^\mathrm{alg}}

\newcommand{\et}{_\mathrm{\acute{e}t}}

\newcommand{\inv}{^{-1}}

\newcommand{\nanal}{non-archimedean analytic\xspace}
\newcommand{\kanal}{$k$-analytic\xspace}

\newcommand{\op}{^\mathrm{op}}

\providecommand{\abs}[1]{\lvert#1\rvert}

\usetikzlibrary{decorations.markings} 
\tikzset{
  closed/.style = {decoration = {markings, mark = at position 0.5 with { \node[transform shape, xscale = .8, yscale=.4] {/}; } }, postaction = {decorate} },
  open/.style = {decoration = {markings, mark = at position 0.5 with { \node[transform shape, scale = .7] {$\circ$}; } }, postaction = {decorate} }
}

\DeclareMathOperator{\Fun}{Fun}

\DeclareMathOperator{\Hom}{Hom}

\DeclareMathOperator{\Map}{Map}

\DeclareMathOperator{\Sp}{Sp}

\DeclareMathOperator{\Spec}{Spec}

\DeclareMathOperator*{\colim}{colim}

\DeclareMathOperator*{\cotimes}{\widehat{\otimes}}

\begin{document}
\title{Derived non-archimedean analytic spaces}

\author{Mauro PORTA}
\address{Mauro PORTA, University of Pennsylvania, David Rittenhouse Laboratory, 209 South 33rd Street, Philadelphia, PA 19104, United States}
\email{maurop@math.upenn.edu}

\author{Tony Yue YU}
\address{Tony Yue YU, Laboratoire de Mathématiques d'Orsay, Université Paris-Sud, CNRS, Université Paris-Saclay, 91405 Orsay, France}
\email{yuyuetony@gmail.com}
\date{January 5, 2016 (Revised on December 30, 2016)}

\begin{abstract}
	We propose a derived version of non-archimedean analytic geometry.
	Intuitively, a derived non-archimedean analytic space consists of an ordinary non-archimedean analytic space equipped with a sheaf of derived rings.
	Such a naive definition turns out to be insufficient.
	In this paper, we resort to the theory of pregeometries and structured topoi introduced by Jacob Lurie.
	We prove the following three fundamental properties of derived non-archimedean analytic spaces:
	
	(1) The category of ordinary non-archimedean analytic spaces embeds fully faithfully into the $\infty$-category of derived non-archimedean analytic spaces.
	
	(2) The $\infty$-category of derived non-archimedean analytic spaces admits fiber products.
	
	(3) The $\infty$-category of higher non-archimedean analytic Deligne-Mumford stacks embeds fully faithfully into the $\infty$-category of derived non-archimedean analytic spaces.
	The essential image of this embedding is spanned by $n$-localic discrete derived non-archimedean analytic spaces.
	
	We will further develop the theory of derived non-archimedean analytic geometry in our subsequent works.
	Our motivations mainly come from intersection theory, enumerative geometry and mirror symmetry.
\end{abstract}

\maketitle

\tableofcontents

\section{Introduction}

\paragraph{\textbf{Motivations}}
Derived algebraic geometry is a far-reaching enhancement of classical algebraic geometry.
We refer to Toën-Vezzosi \cite{HAG-I,HAG-II} and Lurie \cite{Lurie_Thesis,DAG-V} for foundational works.
The prototypical idea of derived algebraic geometry originated from intersection theory:
Let $X$ be a smooth complex projective variety.
Let $Y, Z$ be two smooth closed subvarieties of complementary dimension.
We want to compute their intersection number.
When $Y$ and $Z$ intersect transversally, it suffices to count the number of points in the set-theoretic intersection $Y\cap Z$.
When $Y$ and $Z$ intersect non-transversally, we have two solutions:
The first solution is to perturb $Y$ and $Z$ into transverse intersection;
the second solution is to compute the Euler characteristic of the derived tensor product $\cO_Y\otimes^\rL_{\cO_X}\cO_Z$ of the structure sheaves.
The second solution can be seen as doing perturbation in a more conceptual and algebraic way.
It suggests us to consider spaces with a structure sheaf of derived rings instead of ordinary rings.
This is one main idea of derived algebraic geometry.

Besides intersection theory, motivations for derived algebraic geometry also come from deformation theory, cotangent complexes, moduli problems, virtual fundamental classes, homotopy theory, etc.\ (see Toën \cite{Toen_Derived_2014} for an excellent introduction).
All these motivations apply not only to algebraic geometry, but also to analytic geometry.
Therefore, a theory of derived analytic geometry is as desirable as derived algebraic geometry.

The \emph{purpose} of this paper is to define a notion of derived space in non-archimedean analytic geometry and then study their basic properties.
A non-archimedean field is a field with a complete nontrivial non-archimedean absolute value.
By non-archimedean analytic geometry, we mean the theory of analytic geometry over a non-archimedean field $k$, initiated by Tate \cite{Tate_Rigid_1971}, then systematically developed by Raynaud \cite{Raynaud_Geometrie_analytique_rigide_1974}, Berkovich \cite{Berkovich_Spectral_1990,Berkovich_Etale_1993}, Huber \cite{Huber_Generalization_1994,Huber_Etale_1996} and other mathematicians with different levels of generalizations.
The survey \cite{Conrad_Several_approaches_2008} by Conrad gives a friendly overview of the subject.
We will restrict to the category of quasi-paracompact\footnote{A rigid \kanal space is called quasi-paracompact if it has an admissible affinoid covering of finite type.} quasi-separated rigid $k$-analytic spaces, which is the common intersection of the various approaches to non-archimedean analytic geometry mentioned above.
For readers more familiar with Berkovich spaces, we remark that this category is equivalent to the category of paracompact strictly $k$-analytic spaces in the sense of Berkovich (cf.\ \cite[\S 1.6]{Berkovich_Etale_1993}).

A more direct motivation of our study comes from mirror symmetry.
Mirror symmetry is a conjectural duality between Calabi-Yau manifolds (cf.\ \cite{Yau_Essays_1992,Voisin_Symetrie_1996,Cox_Mirror_symmetry_1999,Hori_Mirror_symmetry_2003}).
More precisely, mirror symmetry concerns degenerating families of Calabi-Yau manifolds instead of individual manifolds.
An algebraic family of Calabi-Yau manifolds over a punctured disc gives rise naturally to a non-archimedean analytic space over the field $\C\llp t\rrp$ of formal Laurent series.
In \cite[\S 3.3]{Kontsevich_Homological_2001}, Kontsevich and Soibelman suggested that the theory of Berkovich spaces may shed new light on the study of mirror symmetry.
Progresses along this direction are made by Kontsevich-Soibelman \cite{Kontsevich_Affine_2006} and by Tony Yue Yu \cite{Yu_Balancing_2013,Yu_Gromov_2014,Yu_Tropicalization_2014,Yu_Enumeration_cylinders_2015,Yu_Enumeration_cylinders_II_2016}.
The works by Gross, Hacking, Keel, Siebert \cite{Gross_Real_Affine_2011,Gross_Mirror_Log_published,Gross_Tropical_2011} are in the same spirit.

More specifically, in \cite{Yu_Enumeration_cylinders_2015}, a new geometric invariant is constructed for log Calabi-Yau surfaces, via the enumeration of holomorphic cylinders in non-archimedean geometry.
These invariants are essential to the reconstruction problem in mirror symmetry.
In order to go beyond the case of log Calabi-Yau surfaces, a general theory of virtual fundamental classes in non-archimedean geometry must be developed.
The situation here resembles very much the intersection theory discussed above, because moduli spaces in enumerative geometry can often be represented locally as intersections of smooth subvarieties in smooth ambient spaces.
The virtual fundamental class is then supposed to be the set-theoretic intersection after perturbation into transverse situations.
However, perturbations do not necessarily exist in analytic geometry.
Consequently, we need a more general and more robust way of constructing the virtual fundamental class, whence the need for derived non-archimedean geometry.

\bigskip

\paragraph{\textbf{Basic ideas and main results}}

Our previous discussion on intersection numbers suggests the following definition of a derived scheme:

\begin{defin}[cf.\ \cite{Toen_Derived_2014}]\label{def:derived_scheme}
A \emph{derived scheme} is a pair $(X,\cO_X)$ consisting of a topological space $X$ and a sheaf $\cO_X$ of commutative simplicial rings on $X$, satisfying the following conditions:
\begin{enumerate}[(i)]
\item The ringed space $(X,\pi_0(\cO_X))$ is a scheme.
\item For each $j\ge 0$, the sheaf $\pi_j(\cO_X)$ is a quasi-coherent sheaf of $\pi_0(\cO_X)$-modules.
\end{enumerate}
\end{defin}

In order to adapt \cref{def:derived_scheme} to analytic geometry, we need to impose certain analytic structures on the sheaf $\cO_X$.
For example, we would like to have a notion of norm on the sections of $\cO_X$;
moreover, we would like to be able to compose the sections of $\cO_X$ with convergent power series.
A practical way to organize such analytic structures is to use the notions of pregeometry and structured topos introduced by Lurie \cite{DAG-V}.
We will review these notions in \cref{sec:definitions} (see also the introduction of \cite{Porta_DCAGI} for an expository account of these ideas).

We will define a pregeometry $\cTank$ which will help us encode the theory of non-archimedean geometry responsible for our purposes.

After that, we are able to introduce our main object of study: derived \kanal spaces.
It is a pair $(\cX,\cO_\cX)$ consisting of an \inftopos $\cX$ and a $\cTank$-structure $\cO_\cX$, satisfying analogs of \cref{def:derived_scheme} Conditions (i)-(ii).
We will explain more intuitions in \cref{rem:definition_intuition}.

The goal of this paper is to study the basic properties of derived \kanal spaces and to compare them with ordinary \kanal spaces.
Here are our main results:

\begin{thm}[cf.\ \cref{thm:fully_faithfulness}]
The category of quasi-paracompact quasi-separated rigid \kanal spaces embeds fully faithfully into the \infcat of derived \kanal spaces.
\end{thm}

\begin{thm}[cf.\ \cref{thm:fiber_products}]
The \infcat of derived \kanal spaces admits fiber products.
\end{thm}

Let $(\Ank,\tauet)$ denote the étale site of rigid \kanal spaces (cf.\ \cite[\S 8.2]{Fresnel_Rigid_2004}) and let $\bP\et$ denote the class of étale morphisms.
The triple $(\Ank,\tauet,\bP\et)$ constitutes a geometric context in the sense of \cite{Porta_Yu_Higher_analytic_stacks_2014}.
The associated geometric stacks are called \emph{higher \kanal Deligne-Mumford stacks}.

\begin{thm}[cf.\ \cref{cor:underived_higher_kanal_stacks}]
	The \infcat of higher \kanal Deligne-Mumford stacks embeds fully faithfully into the \infcat of derived \kanal spaces.
	The essential image of this embedding is spanned by $n$-localic discrete derived \kanal spaces.
\end{thm}

\bigskip
\paragraph{\textbf{Outline of the paper}}

In \cref{sec:definitions}, we introduce the pregeometry $\cTank$ and the notion of derived \kanal space.

In \cref{sec:pregeometry}, we study the properties of the pregeometry $\cTank$.
We prove the unramifiedness conditions as well as the compatibility with truncations.

In \cref{sec:fullyfaithfulness}, we construct a functor $\Phi\colon\Ank\to\dAnk$ from the category of \kanal spaces to the \infcat of derived \kanal spaces.
We prove that $\Phi$ is a fully faithful embedding.

In \cref{sec:closed_etale}, we study closed immersions and étale morphisms under the embedding $\Phi$.

In \cref{sec:fiber_products}, we prove the existence of fiber products between derived \kanal spaces.

In \cref{sec:essential_image}, we characterize the essential image of the embedding $\Phi$.
Moreover, we compare derived \kanal spaces with higher \kanal stacks in the sense of \cite{Porta_Yu_Higher_analytic_stacks_2014}.

\bigskip
\paragraph{\bfseries Notations and terminology}

We refer to Bosch-Güntzer-Remmert \cite{Bosch_Non-archimedean_1984} and Fresnel-van der Put \cite{Fresnel_Rigid_2004} for the classical theory of non-archimedean analytic geometry, to Lurie \cite{HTT,Lurie_Higher_algebra} for the theory of \infcats, and to Lurie \cite{DAG-V} for the theory of structured spaces.

Throughout the paper, by \kanal spaces, we mean quasi-paracompact quasi-separated rigid \kanal spaces.

We denote by $\rSet$ the category of sets and by $\cS$ the \infcat of spaces.
For any small \infcat $\cC$ equipped with a Grothendieck topology $\tau$ and any presentable \infcat $\cD$, we denote by $\PSh_\cD(\cC)$ the \infcat of $\cD$-valued presheaves on $\cC$ and by $\Sh_\cD(\cC,\tau)$ the \infcat of $\cD$-valued sheaves on the \infsite $(\cC,\tau)$.
We will refer to $\cS$-valued presheaves (resp.\ sheaves) simply as presheaves (resp.\ sheaves), and denote $\PSh(\cC) \coloneqq \PSh_\cS(\cC)$, $\Sh(\cC,\tau) \coloneqq \Sh_\cS(\cC,\tau)$.
We denote the Yoneda embedding by
\[ h\colon\cC\to\PSh(\cC),\qquad X\mapsto h_X.\]

\bigskip
\paragraph{\textbf{Related works and further developments}}

Our approach is very much based on the foundational works of Lurie \cite{DAG-V,DAG-VII,DAG-VIII,DAG-IX} on derived algebraic geometry and derived complex analytic geometry.

In \cite{Porta_DCAGI,Porta_DCAGII}, Mauro Porta studied the theories of analytification and deformation in derived complex analytic geometry, more specifically, the analytification functor, relative flatness, derived GAGA theorems, square-zero extensions, analytic modules and cotangent complexes.

The papers by Ben-Bassat and Kremnitzer \cite{Ben-Bassat_Non-archimedean_2013}, by Bambozzi and Ben-Bassat \cite{Bambozzi_Dagger_2015}, and by Paugam \cite{Paugam_Overconvergent_2014} suggest other approaches to derived analytic geometry.

In order to apply derived non-archimedean analytic geometry to enumerative geometry, mirror symmetry as well as other domains of mathematics, we must show that derived non-archimedean analytic spaces arise naturally in these contexts.
The key to the construction of derived structures is to prove a representability theorem in derived non-archimedean geometry.
This will be the main goal of our subsequent work \cite{Porta_Yu_Representability}.

\bigskip
\paragraph{\textbf{Acknowledgements}}
We are grateful to Vladimir Berkovich, Antoine Chambert-Loir, Brian Conrad, Antoine Ducros, Bruno Klingler, Maxim Kontsevich, Jacob Lurie, Marco Robalo, Matthieu Romagny, Pierre Schapira, Carlos Simpson, Michael Temkin, Bertrand Toën and Gabriele Vezzosi for valuable discussions.
The authors would also like to thank each other for the joint effort.
This research was partially conducted during the period Mauro Porta was supported by Simons Foundation grant number 347070 and the group GNSAGA, and Tony Yue Yu served as a Clay Research Fellow.

\section{Basic definitions} \label{sec:definitions}

Intuitively, a derived \nanal space is a ``topological space'' $\cX$ equipped with a structure sheaf $\cO_\cX$ of ``derived non-archimedean analytic rings''.
In order to give the precise definition, we introduce the notions of pregeometry and structured topos following \cite{DAG-V}.

\begin{defin}[{\cite[3.1.1]{DAG-V}}]
A \emph{pregeometry} is an \infcat $\cT$ equipped with a class of \emph{admissible} morphisms and a Grothendieck topology generated by admissible morphisms, satisfying the following conditions:
\begin{enumerate}[(i)]
\item The \infcat $\cT$ admits finite products.
\item The pullback of an admissible morphism along any morphism exists, and is again admissible.
\item For morphisms $f,g$, if $g$ and $g\circ f$ are admissible, then $f$ is admissible.
\item Every retract of an admissible morphism is admissible.
\end{enumerate}
\end{defin}

We now define two pregeometries responsible for derived non-archimedean geometry.

\begin{construction}
	We define a pregeometry $\cTank$ as follows:
	\begin{enumerate}[(i)]
		\item the underlying category of $\cTank$ is the category of smooth $k$-analytic spaces;
		\item a morphism in $\cTank$ is admissible if and only if it is étale;
		\item the topology on $\cTank$ is the étale topology (cf.\ \cite[\S 8.2]{Fresnel_Rigid_2004}).
	\end{enumerate}
\end{construction}

\begin{construction}
	We define a pregeometry $\cTdisck$ as follows:
	\begin{enumerate}[(i)]
		\item the underlying category of $\cTdisck$ is the full subcategory of the category of $k$-schemes spanned by affine spaces $\Spec(k[x_1, \ldots, x_n])$;
		\item a morphism in $\cTdisck$ is admissible if and only if it is an isomorphism;
		\item the topology on $\cTdisck$ is the trivial topology, i.e.\ a collection of admissible morphisms is a covering if and only if it is nonempty.
	\end{enumerate}
\end{construction}

\begin{defin}[{\cite[3.1.4]{DAG-V}}] \label{def:structure}
Let $\cT$ be a pregeometry, and let $\cX$ be an \inftopos.
A \emph{$\cT$-structure} on $\cX$ is a functor $\cO\colon\cT\to\cX$ with the following properties:
\begin{enumerate}[(i)]
\item The functor $\cO$ preserves finite products.
\item Suppose given a pullback diagram
\[
\begin{tikzcd}
U' \arrow{r} \arrow{d} & U \arrow{d}{f} \\
X' \arrow{r} & X
\end{tikzcd}
\]
in $\cT$, where $f$ is admissible.
Then the induced diagram
\[
\begin{tikzcd}
\cO(U') \arrow{r} \arrow{d} & \cO(U) \arrow{d} \\
\cO(X') \arrow{r} & \cO(X)
\end{tikzcd}
\]
is a pullback square in $\cX$.
\item Let $\{U_\alpha\to X\}$ be a covering in $\cT$ consisting of admissible morphisms.
Then the induced map
\[\coprod_\alpha\cO(U_\alpha)\to\cO(X)\]
is an effective epimorphism in $\cX$.
\end{enumerate}
A morphism of $\cT$-structures $\cO\to\cO'$ on $\cX$ is \emph{local} if for every admissible morphism $U\to X$ in $\cT$, the resulting diagram
	\[ \begin{tikzcd}
	\cO(U) \arrow{r} \arrow{d} & \cO'(U) \arrow{d} \\
	\cO(X) \arrow{r} & \cO'(X)
	\end{tikzcd} \]
is a pullback square in $\cX$.
We denote by $\Strloc_\cT(\cX)$ the \infcat of $\cT$-structures on $\cX$ with local morphisms.

A \emph{$\cT$-structured \inftopos} is a pair $(\cX,\cO_\cX)$ consisting of an \inftopos $\cX$ and a $\cT$-structure $\cO_\cX$ on $\cX$.
We denote by $\RTop(\cT)$ the \infcat of $\cT$-structured \inftopoi (cf.\ \cite[Definition 1.4.8]{DAG-V}).
Note that a 1-morphism $f\colon (\cX, \cO_\cX) \to (\cY, \cO_\cY)$ in $\RTop(\cT)$ consists of a geometric morphism of \inftopoi $f_*\colon\cX\rightleftarrows\cY\colon f\inv$ and a local morphism of $\cT$-structures $f^\sharp \colon f\inv \cO_\cY \to \cO_\cX$.
\end{defin}

We have a natural functor $\cTdisck \to \cTank$ induced by analytification.
Composing with this functor, we obtain an ``algebraization'' functor
\[
(-)^\mathrm{alg} \colon \Strloc_{\cTank}(\cX) \to \Strloc_{\cTdisck}(\cX).
\]
In virtue of \cite[Example 3.1.6, Remark 4.1.2]{DAG-V}, we have an equivalence induced by the evaluation on the affine line
\[\Strloc_{\cTdisck}(\cX) \xrightarrow{\ \sim\ } \Sh_{\CAlg_k}(\cX), \]
where $\CAlg_k$ denotes the \infcat of simplicial commutative algebras over $k$.

We are now ready to introduce our main object of study: derived \kanal spaces.

\begin{defin}\label{def:derived_space}
A $\cTank$-structured \inftopos $(\cX,\cO_\cX)$ is called a \emph{derived \kanal space} if $\cX$ is hypercomplete and there exists an effective epimorphism from $\coprod_i U_i$ to the final object of $\cX$ satisfying the following conditions, for every index $i$:
\begin{enumerate}[(i)]
\item The pair $(\cX_{/U_i}, \pi_0(\cO\alg_\cX | U_i))$ is equivalent to the ringed \inftopos associated to the étale site on a \kanal space $X_i$.
\item For each $j\ge 0$, $\pi_j(\cO\alg_\cX | U_i)$ is a coherent sheaf of $\pi_0(\cO\alg_\cX | U_i)$-modules on $X_i$.
\end{enumerate}
We denote by $\dAnk$ the full subcategory of $\RTop(\cTank)$ spanned by derived \kanal spaces.
\end{defin}

\begin{rem}\label{rem:definition_intuition}
	Let us explain the heuristic relation between \cref{def:derived_space} and \cref{def:derived_scheme} in the introduction.
	Let $(\cX,\cO_\cX)$ be a derived \kanal space as in \cref{def:derived_space}.
	Let $\bA^1_k$ be the \kanal affine line and let $O\coloneqq \cO_\cX (\bA^1_k)\in\cX$.
	We have the sum operation $+\colon \bA^1_k \times \bA^1_k \to \bA^1_k$ and the multiplication operation $\bullet\colon \bA^1_k\times \bA^1_k \to \bA^1_k$.
	By \cref{def:structure}(i), they induce respectively a sum operation $+\colon O\times O\to O$ and a multiplication operation $\bullet\colon O\times O\to O$ on $O$.
	Therefore, intuitively, we can think of $O$ as a sheaf of commutative simplicial rings as in \cref{def:derived_scheme}.
	Moreover, the sheaf $O$ is also equipped with analytic structures.
	For example, let $\mathbf D^1_k\subset \bA^1_k$ denote the closed unit disc.
	By \cref{def:structure}(ii), we obtain a monomorphism $\cO_\cX(\mathbf D^1_k)\hookrightarrow O$.
	We can think of $\cO_\cX(\mathbf D^1_k)$ as the subsheaf of $O$ consisting of functions of norm less than or equal to one.
	Furthermore, any holomorphic function $f$ on $\mathbf D^1_k$ induces a morphism $f_O\colon \cO_\cX(\mathbf D^1_k)\to O$, which we think of as the composition with $f$.
	(See also the discussion after \cref{def:derived_scheme}.)
\end{rem}

\begin{rem}
	The hypercompleteness assumption in \cref{def:derived_space} will ensure that the underlying \inftopos of a derived \kanal space has enough points (cf.\ \cref{rem:points_of_cX_X}).
\end{rem}

The goal of this paper is to study the basic properties of derived \kanal spaces and to compare them with ordinary \kanal spaces as well as with the higher \kanal stacks introduced in \cite{Porta_Yu_Higher_analytic_stacks_2014}.

Before moving on, we stress that the underlying \inftopos of a derived \kanal space is, by definition, hypercomplete.
Therefore, using the notations of \cite[\S 2.2]{DAG-V}, for $X \in \cTank$, the $\cTank$-structured \inftopos $\Spec^{\cTank}(X)$ is \emph{not} a derived \kanal space.
We remedy this problem by introducing the hypercomplete spectrum as follows:

Let $\RTop$ (resp.\ $\LTop$) denote the \infcat of \inftopoi where morphisms are right (resp.\ left) adjoint geometric morphisms.
Denote by $\RHTop$ the full subcategory of $\RTop$ spanned by hypercomplete \inftopoi.
Denote by $\RHTop(\cTank)$ the full subcategory of $\RTop(\cTank)$ spanned by $\cTank$-structured \inftopoi $(\cX, \cO_\cX)$ such that $\cX$ is a hypercomplete. It follows from \cite[6.5.2.13]{HTT} that the inclusion $\RHTop \to \RTop$ admits a right adjoint, given by hypercompletion.
This induces a right adjoint to the inclusion $\RHTop(\cTank) \hookrightarrow \RTop(\cTank)$, as the next lemma shows:

\begin{lem} \label{lem:hyp_right_adjoint}
	The inclusion $\RHTop(\cTank) \hookrightarrow \RTop(\cTank)$ admits a right adjoint, which we denote by $\Hyp \colon \RTop(\cTank) \to \RHTop(\cTank)$.
\end{lem}

\begin{proof}
	Fix $X\coloneqq(\cX, \cO_\cX) \in \RTop(\cTank)$.
	Since the hypercompletion $L \colon \cX \to \cX^\wedge$ is left exact,  we obtain a well defined functor $\Strloc_{\cTank}(\cX) \to \Strloc_{\cTank}(\cX^\wedge)$ induced by composition with $L$.
	Let $X^\wedge \coloneqq (\cX^\wedge, L(\cO_\cX))$ be the resulting hypercomplete $\cTank$-structured \inftopos.
	In $\RTop(\cTank)$ there is a natural morphism $\varphi \colon X^\wedge \to X$.
	
	Using the dual of \cite[5.2.7.8]{HTT} it suffices to show that $\varphi$ exhibits $X^\wedge$ as a colocalization of $X$ relative to $\RHTop(\cTank)$.  In order to prove this, let $Y \coloneqq (\cY, \cO_\cY)$ be any hypercomplete $\cTank$-structured \inftopos.
	Using \cite[6.5.2.13]{HTT} we obtain an equivalence
	\[ \Map_{\RTop}(\cY, \cX^\wedge) \to \Map_{\RTop}(\cY, \cX). \]
	Fix a geometric morphism $g_*\colon\cY\rightleftarrows\cX^\wedge\colon g\inv$ and let $(f\inv, f_*)$ denote the induced geometric morphism $\cY \rightleftarrows \cX$.
	We remark that $f\inv \simeq g\inv \circ L$.
	Using \cite[2.4.4.2]{HTT} and \cite[Remark 1.4.10]{DAG-V} we obtain a morphism of fiber sequences:
	\[ \begin{tikzcd}[column sep=small]
	\Map_{\Strloc_{\cTank}(\cY)}(g\inv L(\cO_X), \cO_\cY) \arrow{r} \arrow{d} & \Map_{\RTop(\cTank)}(Y, X^\wedge) \arrow{r} \arrow{d} & \Map_{\RTop}(\cY, \cX^\wedge) \arrow{d} \\
	\Map_{\Strloc_{\cTank}(\cY)}(f\inv \cO_X, \cO_\cY) \arrow{r} & \Map_{\RTop(\cTank)}(Y, X) \arrow{r} & \Map_{\RTop}(\cY, \cX).
	\end{tikzcd} \]
	Since $f\inv \simeq g\inv \circ L$, we see that the left vertical morphism is an equivalence. Since this holds for all base points in $\Map_{\RTop}(\cY, \cX)$, we conclude that the middle vertical morphism is an equivalence as well, completing the proof.
\end{proof}

\begin{defin}
	Given $X \in \cTank$, we define its \emph{hypercomplete (absolute) spectrum} $\HSpec^{\cTank}(X)$ to be $\Hyp(\Spec^{\cTank}(X))$.
\end{defin}

\begin{lem} \label{lem:universal_property_HSpec}
	Let $Y \coloneqq (\cY, \cO_\cY)$ be a derived \kanal space and let $X \in \cTank$.
	The natural morphism $\HSpec^{\cTank}(X) \to \Spec^{\cTank}(X)$ induces an equivalence
	\[ \Map_{\RHTop(\cTank)}(Y, \HSpec^{\cTank}(X)) \xrightarrow{\ \sim\ } \Map_{\RTop(\cTank)}(Y, \Spec^{\cTank}(X)). \]
\end{lem}

\begin{proof}
	Since $Y$ belongs to $\RHTop(\cTank)$, the statement is an immediate consequence of \cref{lem:hyp_right_adjoint}.
\end{proof}

\section{Properties of the pregeometry} \label{sec:pregeometry}

In this section, we study the properties of the pregeometry $\cTank$ introduced in \cref{sec:definitions}.
More specifically, we will prove the unramifiedness of $\cTank$, the unramifiedness of the algebraization and the compatibility of $\cTank$ with $n$-truncations.

\subsection{Unramifiedness}

In order that the collection of closed immersions behaves well with respect to fiber products, our pregeometry $\cTank$ has to verify a condition of unramifiedness.

\begin{defin}[{\cite[1.3]{DAG-IX}}]\label{def:unramified_pregeometry}
	A pregeometry $\cT$ is said to be \emph{unramified} if for every morphism $f\colon X\to Y$ in $\cT$ and every object $Z\in\cT$, the diagram
	\[ \begin{tikzcd}
	X\times Z \arrow{r} \arrow{d} & X\times Y\times Z \arrow{d} \\
	X \arrow{r} & X\times Y
	\end{tikzcd} \]
	induces a pullback square
	\[ \begin{tikzcd}
	\cX_{X\times Z} \arrow{r} \arrow{d} & \cX_{X\times Y\times Z} \arrow{d} \\
	\cX_X \arrow{r} & \cX_{X\times Y}
	\end{tikzcd} \]
	in $\RTop$, where the symbol $\cX_{(-)}$ denotes the associated \inftopos.
\end{defin}

Our first goal is to prove that the pregeometry $\cTank$ is unramified (cf.\ \cref{cor:Tkan_unramified}). In order to do this, we need to describe explicitly the \inftopos $\cX_X$ associated to a \kanal space $X$ and prove that the assignment $X \mapsto \cX_X$ is well behaved with respect to closed immersions (cf.\ \cref{prop:preserve_closed_immersion}).

Let $\Ank$ denote the category of \kanal spaces and let $\Afd_k$ denote the category of $k$-affinoid spaces.
For $X\in\Ank$, let $(\An_X)\et$ (resp.\ $(\Afd_X)\et$) denote the category of étale morphisms from \kanal spaces (resp.\ $k$-affinoid spaces) to $X$.
We equip the categories $(\An_X)\et$ and $(\Afd_X)\et$ with the étale topology.
By \cite[Proposition 2.24]{Porta_Yu_Higher_analytic_stacks_2014}, the inclusion $(\Afd_X)\et\hookrightarrow(\An_X)\et$ induces an equivalence of \inftopoi
\begin{equation}\label{eq:afd_in_an}
\Sh((\Afd_X)\et)\xrightarrow{\ \sim\ }\Sh((\An_X)\et).
\end{equation}
We call the two equivalent \inftopoi above the \emph{étale \inftopos associated to $X$}, and denote it by $\cX_X$.
We will denote the site $(\Afd_X)\et$ by $X\et$ for simplicity.

\begin{rem}
	The \inftopos $\cX_X$ is not hypercomplete in general.
	In the subsequent sections we will also consider its hypercompletion $\cX_X^\wedge$.
	\end{rem}

\begin{rem}\label{rem:points_of_cX_X}
Since the site $X\et$ is a 1-category, the \inftopos $\cX_X$ is $1$-localic.
It follows that for any \inftopos $\cY$ one has an equivalence of $\infty$-categories \[
\Fun_*(\cY, \cX_X) \simeq \Fun_*(\tau_{\le 0} \cY, \tau_{\le 0} \cX_X),
\]
where $\Fun_*$ denotes the $\infty$-category of geometric morphisms (taken in $\RTop$).
Put $\cY = \cS$ and observe that $\tau_{\le 0} \cX_X = \Sh_\rSet(X\et)$ and $\tau_{\le 0}(\cS) \simeq \rSet$.
We conclude that the points of $\cX_X$ correspond bijectively to the points of the classical 1-topos associated to the site $X\et$.
The latter is classified by the geometric points of the adic space associated to $X$ in the sense of Huber (cf.\ \cite[Proposition 2.5.17]{Huber_Etale_1996}).

Since the site $X\et$ is finitary, it follows from \cite[Corollary 3.22]{DAG-VII} that the hypercompletion $\cX_X^\wedge$ is locally coherent.
Therefore, by Theorem 4.1 in loc.\ cit., the \inftopos $\cX_X^\wedge$ has enough points.
\end{rem}

\begin{rem}\label{rem:spectrum}
As we already discussed in \cref{sec:definitions}, \cite[\S 2.2]{DAG-V} assigns to every $X \in \cTank$ a $\cTank$-structured \inftopos $\Spec^{\cTank}(X)$, called the \emph{spectrum} of $X$.
It is characterized by the following universal property: for any $\cTank$-structured \inftopos $(\cY, \cO_\cY)$ there is a natural equivalence
\[ \Map_{\RTop(\cTank)}( (\cY, \cO_\cY), \Spec^{\cTank}(X) ) \simeq \Map_{\mathrm{Ind}(\cG_\mathrm{an}(k)^{\mathrm{op}})}( X, \Gamma_{\cG}(\cY, \cO_\cY) ), \]
where $\cG_{\mathrm{an}}(k)$ denotes a geometric envelope of $\cTank$ (cf.\ \cite[Theorem 2.2.12]{DAG-V}).
We note that the underlying \inftopos of $\Spec^{\cTank}(X)$ can be identified with $\cX_X$.
\end{rem}

We refer to \cite[7.3.2]{HTT} for the notion of closed immersion of \inftopoi.

\begin{prop}\label{prop:preserve_closed_immersion}
The functor
\begin{align*}
\Ank&\longrightarrow \mathrm{h}(\RTop) \\
X&\longmapsto\cX_X
\end{align*}
preserves closed immersions, where $\mathrm h(\RTop)$ denotes the homotopy category of $\RTop$.
\end{prop}

\begin{rem}
	It will follow from the results of \cref{sec:fullyfaithfulness} (see in particular \cref{lem:rigidity} and the construction of $\Phi$) that the functor above can be promoted to an $\infty$-functor $\Ank \to \RTop$.
	\end{rem}

\begin{lem} \label{lem:descent_for_closed_subtopoi}
Let $\cX, \cY$ be \inftopoi and let $U \in \cX$.
Let $f\inv \colon \cX /U \rightleftarrows \cY \colon f_*$ be a geometric morphism.
Then $(f\inv, f_*)$ is an equivalence if and only if there exists an effective epimorphism $V \to 1_\cX$ such that $\cX_{/V} / (U \times V) \rightleftarrows \cY_{/ f\inv(V)}$ is an equivalence.
\end{lem}
\begin{proof}
To see that the condition is necessary it is enough to take $V \to 1_\cX$ to be the identity of $1_\cX$.
We now prove the sufficiency.
Let us denote by $j\inv \colon \cX \leftrightarrows \cX / U \colon j_*$ (resp.\ $i\inv \colon \cX \leftrightarrows \cX_{/V} \colon i_*$) the given closed (resp.\ \'etale) morphism of \inftopoi.
We claim that
\[\cX_{/V} / (U \times V) \simeq (\cX/U)_{/j\inv(V)}.\]
Indeed, the left hand side can be identified with the pullback $\cX_{/V} \times_{\cX} \cX/U$ in virtue of \cite[6.3.5.8]{HTT}.
The right hand side can be identified with the same pullback in virtue of \cite[7.3.2.13]{HTT}.
At this point, we obtain a commutative square of geometric morphisms in $\RTop$
\[
\begin{tikzcd}
\cY_{/f\inv(V)} \arrow{r} \arrow{d} & (\cX/U)_{/j\inv(V)} \arrow{d} \\
\cY \arrow{r}{f_*} & \cX / U.
\end{tikzcd}
\]
So the lemma follows from the descent property of \inftopoi \cite[6.1.3.9(3)]{HTT}.
\end{proof}

\begin{lem} \label{lem:qet_structure}
Let $A \to B$ be a surjective morphism of $k$-affinoid algebras.
Let $B \to B'$ be an \'etale morphism of $k$-affinoid $A$-algebras.
Then there exists an \'etale $A$-algebra $A'$ and a pushout square:
\[
\begin{tikzcd}
A \arrow{r} \arrow{d} & B \arrow{d} \\
A' \arrow{r} & B'.
\end{tikzcd}
\]
\end{lem}

\begin{proof}
Since $B \to B'$ is \'etale, by \cite[Proposition 1.7.1]{Huber_Etale_1996},  we can write
\[
B' = B \langle y_1, \ldots, y_m \rangle / (f_1, \ldots, f_m),
\]
such that the Jacobian $J\coloneqq\mathrm{Jac}(f_1, \ldots, f_m)$ is invertible in $B'$.
So
\[\rho\coloneqq\min_{x\in\Sp B'}\abs{J(x)}\]
is positive.
Since $A \to B$ is surjective, the induced morphism
\[
A \langle y_1, \ldots, y_m \rangle \to B \langle y_1, \ldots, y_m \rangle
\]
is surjective as well.
Therefore we can find elements $\overline{f_1}, \ldots, \overline{f_m} \in A \langle y_1, \ldots, y_m \rangle$ lifting $f_1, \ldots, f_m$.
Set
\[
A_0 \coloneqq A \langle y_1, \ldots, y_m \rangle / (\overline{f_1}, \ldots, \overline{f_m}).
\]
Let $\widebar J\coloneqq\mathrm{Jac}(\overline{f_1}, \ldots, \overline{f_m})$.
Let $n$ be a positive integer such that $\rho^n\in\abs{k}$ and let $a$ be an element in $k$ such that $\abs{a}=\rho^n$.
Set $A' \coloneqq A_0 \langle w \rangle / (w \widebar J^n - a)$.
We see that the natural morphism $A_0\to B'$ factors as
\[
A_0 \to A' \to B'.
\]
It follows from the construction that $A \to A'$ is \'etale, and moreover $B' \simeq A' \cotimes_A B$, completing the proof.
\end{proof}

\begin{proof}[Proof of \cref{prop:preserve_closed_immersion}]
Let $f \colon Y \to X$ be a closed immersion in $\An_k$.
Let $U \colon X\et \to \cS$ be the functor defined by the formula
\[
U(Z) = \begin{cases}
\{*\} & \text{if } Z\times_X Y =\emptyset, \\
\emptyset & \text{otherwise.}
\end{cases}
\]
This is a sheaf and therefore determines a closed subtopos $\cX_X / U$.
The morphism $f$ induces a geometric morphism
\[f\inv \colon \cX_X \rightleftarrows \cX_Y \colon f_* .\]
We claim that $f_*$ factors through the closed subtopos $\cX_X / U$.
Indeed, it suffices to check that for every sheaf $G \in \cX_Y$ and every representable sheaf $h_{Z}$ in $\cX_X$ such that $\Map_{\cX_X}(h_{Z}, U) \ne \emptyset$, the space $\Map_{\cX_X}(h_{Z}, f_*(G))$ is contractible.
This is true, because we have
\[
\Map_{\cX_X}(h_{Z}, f_*(G)) \simeq G(Z\times_X Y) = G(\emptyset) \simeq \{*\}.
\]
We denote by $(f\inv,f_*)$ again the induced adjunction
\begin{equation} \label{eq:induced_adjunction}
f\inv \colon \cX_X / U \rightleftarrows \cX_Y \colon f_* .
\end{equation}
We conclude our proof by the following lemma.
\end{proof}

\begin{lem} \label{lem:closed_immersion_closed_subtopos}
The adjunction in \cref{eq:induced_adjunction} is an equivalence.
\end{lem}
\begin{proof}
By \cref{lem:descent_for_closed_subtopoi}, we can assume that both $X$ and $Y$ are affinoid.
Note that $\cX_X / U$ and $\cX_Y$ are $1$-localic \inftopoi in virtue of \cite[7.5.4.2]{HTT} and \cite[Lemma 1.2.6]{DAG-VIII}.
Therefore it suffices to show that the adjunction $(f\inv, f_*)$ induces an equivalence when restricted to $1$-truncated objects of $\cX_X / U$ and $\cX_Y$.

Let us prove that the functor $f_*$ is conservative.
Let $\alpha \colon F \to F'$ be a morphism in $\cX_Y$ and suppose that $f_*(\alpha)$ is an equivalence.
By the equivalence \eqref{eq:afd_in_an}, it is enough to show that $\alpha$ induces equivalences $F(Y') \to F'(Y')$ for every étale morphism $Y' \to Y$.
Using \cref{lem:qet_structure}, we can form a pullback diagram
\[
\begin{tikzcd}
Y' \arrow{r} \arrow{d} & X' \arrow{d} \\
Y \arrow{r} & X ,
\end{tikzcd}
\]
where $X' \to X$ is étale.
It follows that
\[
F(Y') = (f_* F)(X') \to (f_* F')(X') = F(Y')
\]
is an equivalence.

We are left to check that the unit of the adjunction $(f\inv, f_*)$ is an equivalence over $1$-truncated objects.
For this, it suffices to check that for every $1$-truncated sheaf $F \in \cX_X$,
the unit $u \colon F \to f_* f\inv F$ induces an equivalence on sheaves of homotopy groups.
Since both $F$ and $f_* f\inv F$ are $1$-truncated, they are hypercomplete objects.
Therefore, it suffices to check that $\eta\inv(u)$ is an equivalence for every geometric morphism $\eta\inv \colon \cX_X \to \cS\colon\eta_*$.
Such a geometric morphism corresponds to a geometric point $x$ of the adic space associated to $X$ (cf.\ \cref{rem:points_of_cX_X}).
Let $\{V_\alpha\}$ be a system of étale neighborhoods of $x$.
We have $\eta\inv(G)=\colim G(V_\alpha)$.

If $x$ does not meet $Y$, we see that $\eta\inv(G)$ is contractible whenever $G \in \cX_X / U$.
In particular $\eta\inv(u)$ is an equivalence for every $1$-truncated $F \in \cX_X / U$.

Otherwise, $x$ lifts to a geometric morphism $\eta_1\inv \colon \cX_Y \to \cS$, satisfying $\eta\inv = \eta_1\inv \circ f\inv$. 
So we have
\begin{align*}
\eta\inv(f_* f\inv F) & \simeq \colim (f_* f\inv F)(V_\alpha) \\
& \simeq \colim (f\inv F)(V_\alpha \times_{X} Y) \\
& \simeq \eta_1\inv f\inv F \simeq \eta\inv F ,
\end{align*}
completing the proof.
\end{proof}

\begin{prop} \label{prop:closed_immersion_pullback_of_topoi}
	Let
	\[ \begin{tikzcd}
	W \arrow{r} \arrow{d} & Y \arrow{d}{g} \\
	X \arrow{r}{f} & Z
	\end{tikzcd} \]
	be a pullback square in $\Ank$ and assume that $f$ is a closed immersion. The induced square of \inftopoi
	\[ \begin{tikzcd}
	\cX_W \arrow{r} \arrow{d} & \cX_Y \arrow{d}{g_*} \\
	\cX_X \arrow{r}{f_*} & \cX_Z
	\end{tikzcd} \]
	is a pullback diagram in $\RTop$.
	\end{prop}

\begin{proof}
	Let $U_X$ be the sheaf on the étale site $Z\et$ of $Z$ defined by
	\[ U_X(T) \coloneqq \begin{cases} \{*\} & \text{if } T \times_Z X = \emptyset \\ \emptyset & \text{otherwise.} \end{cases} \]
	Define $U_W$ to be the sheaf on the étale site $Y\et$ of $Y$ in a similar way.
	Using \cref{lem:closed_immersion_closed_subtopos} twice, we can rewrite the induced square of \inftopoi as
	\[ \begin{tikzcd}
	\cX_Y / U_W \arrow{r} \arrow{d} & \cX_Y \arrow{d}{g_*} \\
	\cX_Z / U_X \arrow{r} & \cX_Z .
	\end{tikzcd} \]
	In virtue of \cite[7.3.2.13]{HTT}, we only need to show that $g\inv U_X \simeq U_W$.
	First of all, let us observe that there exists a map $U_X \to g_* U_W$: indeed, if $T \to X$ is étale with $T \to Z$ a smooth morphism such that $T \times_Z X = \emptyset$, then we also have $(T \times_Z Y) \times_Y W \simeq (T\times_Z W) \times_Z Y = \emptyset$, and therefore $g_*(U_W)(T) = U_W(T \times_Z Y) = \Delta^0$.
	This allows to define the desired map, which induces by adjunction a morphism $g\inv U_X \to U_W$.
	By construction, $U_W$ is $(-1)$-truncated and \cite[5.5.6.16]{HTT} shows that $g\inv U_X$ is $(-1)$-truncated too.
	Therefore they are both hypercomplete.
	So it suffices to check that $g\inv U_X \to U_W$ is an isomorphism on the stalks of $\cX_Y$. This is true because a geometric point $\eta_* \colon \cS \to \cX_Y$ factors through $\cX_W$ if and only if $g_* \circ \eta_*$ factors through $\cX_X$ (cf.\ \cref{rem:points_of_cX_X}).
\end{proof}

\begin{cor} \label{cor:Tkan_unramified}
	The pregeometry $\cTank$ is unramified.
\end{cor}

\begin{proof}
	We check that \cref{def:unramified_pregeometry} is satisfied.
	Let $X, Y, Z \in \cTank$ and let $f \colon Y \to X$ be any morphism.
	The diagram
	\[ \begin{tikzcd}
		X \arrow{r} \arrow{d}{\mathrm{id}_X \times f} & Y \arrow{d}{\Delta} \\
		X \times Y \arrow{r} & Y \times Y
	\end{tikzcd} \]
	is a pullback diagram. Since $Y$ is separated, $Y$ is a closed immersion, and therefore the same goes for $X \to X \times Y$.
	We can therefore use \cref{prop:closed_immersion_pullback_of_topoi} to conclude that the induced square
	\[	\begin{tikzcd}
			\cX_{X \times Z} \arrow{r} \arrow{d} & \cX_X \arrow{d} \\
			\cX_{X \times Y \times Z} \arrow{r} & \cX_{X \times Y}.
		\end{tikzcd} \]
	is a pullback diagram in $\RTop$.
\end{proof}

\subsection{Algebraization}

The functor $\cTdisck \to \cTank$ induced by analytification is a transformation of pregeometries in the following sense:

\begin{defin}[{\cite[3.2.1]{DAG-IX}}]
	A \emph{transformation of pregeometries} from $\cT$ to $\cT'$ is a functor $\theta\colon\cT\to\cT'$ such that
	\begin{enumerate}[(i)]
		\item it preserves finite products;
		\item it sends admissible morphisms in $\cT$ to admissible morphisms in $\cT'$;
		\item it sends coverings in $\cT$ to coverings in $\cT'$;
		\item it sends any pullback in $\cT$ along an admissible morphism to a pullback in $\cT'$.
	\end{enumerate}
\end{defin}

In the following, we study some properties of the transformation of pregeometries $\cTdisck \to \cTank$.

\begin{lem} \label{lem:alg_conservative}
	Let $\cX$ be an \inftopos.
	The algebraization functor
	\[
	(-)^\mathrm{alg} \colon \Strloc_{\cTank}(\cX) \to \Strloc_{\cTdisck}(\cX)
	\]
	induced by composition with the transformation $\cTdisck \to \cTank$ is conservative.
\end{lem}

\begin{proof}
Let $f \colon \cO \to \cO'$ be a local morphism of $\cTank$-structures on $\cX$ such that $f\alg \colon \cO\alg \to \cO^{\prime \mathrm{alg}}$ is an equivalence.
We will show that for every $X\in\cTank$,
the induced morphism $\cO(X) \to \cO'(X)$ is an equivalence.
Since $X$ is smooth, there exists an affinoid G-covering $\{\Sp B_i\to X\}$ such that every $\Sp B_i$ admits an étale morphism to a \kanal affine space.

So we obtain a commutative square
\[
\begin{tikzcd}
\coprod \cO(\Sp B_i) \arrow{d} \arrow{r} & \coprod \cO'(\Sp B_i) \arrow{d} \\
\cO(X) \arrow{r} & \cO'(X) ,
\end{tikzcd}
\]
where the vertical morphisms are effective epimorphisms.
Moreover, since admissible open immersions are étale and $f$ is a local morphism, we see that the above square is a pullback.
We are therefore reduced to show that $\cO(\Sp B) \to \cO'(\Sp B)$ is an equivalence whenever $\Sp B$ admits an étale morphism to a \kanal affine space $\bA^n_k$.
Since $f$ is a local morphism, we have in this case a pullback square
\[
\begin{tikzcd}
\cO(\Sp B) \arrow{r} \arrow{d} & \cO(\bA^n_k) \arrow{d} \\
\cO'(\Sp B) \arrow{r} & \cO'(\bA^n_k).
\end{tikzcd}
\]
Let $\bbA^n_k$ denote the $n$-dimensional algebraic affine space over $k$.
Since $\cO(\bA^n_k) = \cO\alg(\mathbb A^n_k) \to \cO^{\prime \mathrm{alg}}(\mathbb A^n_k) = \cO'(\bA^n_k)$ is an equivalence by our assumption, we deduce that $\cO(\Sp B) \to \cO'(\Sp B)$ is an equivalence as well, completing the proof.
\end{proof}

\begin{prop} \label{prop:alg_effective_epi}
	Let $\cX$ be an \inftopos and let $f \colon \cO \to \cO'$ be a morphism in $\Strloc_{\cTank}(\cX)$.
	The following conditions are equivalent:
	\begin{enumerate}[(i)]
		\item The morphism $f$ is an effective epimorphism, i.e.\ for every $U \in \cTank$ the morphism $\cO(U) \to \cO'(U)$ is an effective epimorphism in $\cX$.
		\item The morphism $f\alg \colon \cO\alg \to \cO^{\prime \mathrm{alg}}$ is an effective epimorphism.
		\item The morphism $\cO(\bA^1_k) \to \cO'(\bA^1_k)$ is an effective epimorphism.
	\end{enumerate}
\end{prop}

\begin{proof}
	It follows directly from the definition of effective epimorphism of $\cTank$-structures that (i) implies (ii) and (ii) implies (iii).
	Let us show that (iii) implies (i).
	Let $X \in \cTank$. Choose an étale covering $\{U_i \to X \}$ such that each $U_i$ admits an étale morphism to $\bA^n_k$.
	Since $f$ is a local morphism, we have the following pullback square:
	\[ \begin{tikzcd}
	\coprod \cO(U_i) \arrow{r} \arrow{d} & \coprod \cO'(U_i) \arrow{d} \\
	\cO(X) \arrow{r} & \cO'(X) .
	\end{tikzcd} \]
	The vertical arrows are effective epimorphisms, and therefore it suffices to check that the upper horizontal map is an effective epimorphism.
	Since $f$ is a local morphism, we see that the diagram
	\[ \begin{tikzcd}
	\cO(U_i) \arrow{r} \arrow{d} & \cO'(U_i) \arrow{d} \\
	\cO(\bA^n_k) \arrow{r} & \cO'(\bA^n_k)
	\end{tikzcd} \]
	is a pullback diagram.
	So it suffices to show that $\cO(\bA^n_k) \to \cO'(\bA^n_k)$ is an effective epimorphism. This follows from the hypothesis and the fact that both $\cO$ and $\cO'$ commute with products.
\end{proof}

\begin{defin}[{\cite[10.1]{DAG-IX}}]
	Let $\theta\colon\cT'\to\cT$ be a transformation of pregeometries, and $\Theta\colon\Top(\cT)\to\Top(\cT')$ the induced functor given by composition with $\theta$.
	We say that $\theta$ is \emph{unramified} if the following conditions hold:
	\begin{enumerate}[(i)]
		\item The pregeometries $\cT$ and $\cT'$ are unramified.
		\item For every morphism $f\colon X\to Y$ in $\cT$ and every object $Z\in\cT$, the diagram
			\[ \begin{tikzcd}
			\Theta\Spec^\cT(X\times Z) \arrow{r} \arrow{d} & \Theta\Spec^\cT(X) \arrow{d} \\
			\Theta\Spec^\cT(X\times Y\times Z) \arrow{r} & \Theta\Spec^\cT(X\times Y)
			\end{tikzcd} \]
		is a pullback square in $\Top(\cT')$.
	\end{enumerate}
\end{defin}

\begin{prop} \label{prop:unramified_transformation}
The transformation of pregeometries $\cTdisck \to \cTank$ is unramified.
\end{prop}

\begin{proof}
For $X\in\cTank$, we denote the spectrum $\Spec^{\cTank}(X)$ by $(\cX_X, \cO_X)$.
For a morphism $X \to Y$ in $\cTank$,
we denote by $\cO_Y\alg | X$ the image of $\cO_Y\alg$ under the pullback functor $\Strloc_{\cTdisck}(\cX_Y) \to \Strloc_{\cTdisck}(\cX_X)$.
We have to show that for every morphism $f \colon X \to Y$ in $\cTank$ and every $Z\in\cTank$,
the commutative square
\begin{equation} \label{eq:desired_pushout}
\begin{tikzcd}
\cO_{X \times Y}\alg | (X \times Z) \arrow{r} \arrow{d} & \cO_X\alg | (X \times Z) \arrow{d} \\
\cO_{X \times Y \times Z}\alg | (X \times Z) \arrow{r} & \cO_{X \times Z}\alg
\end{tikzcd}
\end{equation}
is a pushout in $\Strloc_{\cTdisck}(\cX_{X \times Z}) \simeq \Sh_{\CAlg_k}(\cX_{X \times Z})$.

Form the pushout
\[ \begin{tikzcd}
	\cO_{X \times Y}\alg | (X \times Z) \arrow{r} \arrow{d} & \cO_X\alg | (X \times Z) \arrow{d} \\
	\cO_{X \times Y \times Z}\alg | (X \times Z) \arrow{r} & \cA
\end{tikzcd} \]
in $\Sh_{\CAlg_k}(\cX_{X \times Z})$.
Let $\cA^\wedge$ be the hypercompletion of $\cA$.
We will prove below that $\cA^\wedge$ is equivalent to $\cO_{X \times Z}\alg$.
Assuming this, we see that $\cA^\wedge$ is discrete.
It follows that $\cA$ is discrete as well, and therefore it is hypercomplete.
We thus conclude that the square \eqref{eq:desired_pushout} is a pushout.

So we are reduced to show that the map $\cA^\wedge \to \cO_{X \times Z}\alg$ is an equivalence.
Both sheaves are hypercomplete and \cref{rem:points_of_cX_X} shows that $\cX_{X \times Z}^\wedge$ has enough points.
Thus, it suffices to show that for every geometric point $(x,z)$ of the adic space associated to $X \times Z$ in the sense of Huber, the diagram
\begin{equation}\label{eq:pushout_stalks}
\begin{tikzcd}
\cO_{(x,y)}\alg \arrow{r} \arrow{d} & \cO_x\alg \arrow{d} \\
\cO_{(x,y,z)}\alg \arrow{r} & \cO_{(x,z)}\alg
\end{tikzcd}
\end{equation}
is a pushout square, where we set $y \coloneqq f(x)$.
Choose a fundamental system of étale affinoid neighborhoods $\{V_\alpha\}$ of $(x,y)$ in $X \times Y$.
Set $U_\alpha \coloneqq V_\alpha \times_{X \times Y} X$ and observe that $\{U_\alpha\}$ forms a fundamental system of étale affinoid neighborhoods of $x$ in $X$.
Choose moreover a fundamental system $\{W_\beta\}$ of étale affinoid neighborhoods of $z$ in $Z$.
We have pullback squares
\begin{equation}\label{eq:neighborhood_systems}
\begin{tikzcd}
U_\alpha \times W_\beta \arrow{r} \arrow{d} & U_\alpha \arrow{d} \\
V_\alpha \times W_\beta \arrow{r} & V_\alpha.
\end{tikzcd}
\end{equation}
Assume $U_\alpha = \Sp A_\alpha$, $V_\alpha = \Sp B_\alpha$ and $W_\beta = \Sp C_\beta$.
Since $U_\alpha \to V_\alpha$ is a closed immersion, the pullback above corresponds to a pushout in the category of $k$-algebras\begin{equation}\label{eq:algebraic_neighborhood_systems}
\begin{tikzcd}
B_\alpha \arrow{r} \arrow{d} & B_\alpha \cotimes_k C_\beta \arrow{d} \\
A_\alpha \arrow{r} & A_\alpha \cotimes_k C_\beta
\end{tikzcd}
\end{equation}
Taking limit in Diagram \ref{eq:neighborhood_systems} (or equivalently, taking colimit in Diagram \ref{eq:algebraic_neighborhood_systems}),
we observe that Diagram \ref{eq:pushout_stalks} is a pushout diagram in the category of $k$-algebras.
Since the projections $V_\alpha \times W_\beta \to V_\alpha$ are flat, we see that every morphism $B_\alpha \to B_\alpha \cotimes_k C_\beta$ is flat.
As a consequence, $\cO_{(x,y)}\alg \to \cO_{(x,y,z)}\alg$ is flat.
The pushout \eqref{eq:pushout_stalks} is therefore a derived pushout square, completing the proof.
\end{proof}

Intuitively, the pregeometry $\cTank$ enables us to consider structure sheaves with ``non-archimedean analytic structures'' in addition to the usual algebraic structures.
The unramifiedness of the transformation $\cTdisck\to\cTank$ in \cref{prop:unramified_transformation} will imply that for certain purposes, this additional analytic structure can be ignored.
Here is a simple example illustrating this phenomenon:
Consider the completed tensor product $A \cotimes_B C$ of three $k$-affinoid algebras.
When $C$ is finitely presented as a $B$-module, we have an isomorphism $A \cotimes_B C\simeq A\otimes_B C$.
That is, in this case, for the purpose of tensor product, the analytic structure on affinoid algebras can be ignored.
The proposition below elaborates on this idea:

\begin{prop} \label{prop:closed_fiber_products_Top}
	Let $f \colon (\cY, \cO_{\cY}) \to (\cX, \cO_{\cX})$ and $g \colon (\cX', \cO_{\cX'}) \to (\cX, \cO_{\cX})$ be morphisms in $\RTop(\cTank)$.
	Assume that the induced map $\theta \colon f\inv \cO_{\cX}\alg \to \cO_{\cY}\alg$ is an effective epimorphism.
	Then:
	\begin{enumerate}[(i)]
		\item \label{item:pullback_structured_Top} There exists a pullback diagram
			\[ \begin{tikzcd}
				(\cY', \cO_{\cY'}) \arrow{r}{f'} \arrow{d}{g'} & (\cX', \cO_{\cX'}) \arrow{d}{g} \\
				(\cY, \cO_{\cY}) \arrow{r}{f} & (\cX, \cO_{\cX})
			\end{tikzcd} \]
			in $\RTop(\cTank)$. If moreover $(\cX, \cO_\cX), (\cX', \cO_{\cX'}), (\cY, \cO_\cY) \in \RHTop(\cTank)$, then $\Hyp(\cY', \cO_{\cY'})$ is equivalent to the pullback computed in $\RHTop(\cTank)$.
		\item \label{item:pullback_Top} The underlying diagram of \inftopoi
			\[ \begin{tikzcd}
				\cY' \arrow{r} \arrow{d} & \cX' \arrow{d} \\
				\cY \arrow{r} & \cX
			\end{tikzcd} \]
			is a pullback square in $\RTop$.
			If moreover
			\[(\cX, \cO_\cX), (\cX', \cO_{\cX'}), (\cY, \cO_\cY) \in \RHTop(\cTank),\] then $(\cY')^\wedge$ is equivalent to the pullback computed in $\RHTop$.
		\item \label{item:pushout_analytic_algebras} The diagram
			\[ \begin{tikzcd}
				f^{\prime -1} g\inv \cO_{\cX}\alg \arrow{r} \arrow{d} & f^{\prime -1} \cO_\cY\alg \arrow{d} \\
				g^{\prime -1} \cO_{\cY}\alg \arrow{r} & \cO_{\cY'}\alg
			\end{tikzcd} \]
			is a pushout square in $\Sh_{\CAlg_k}(\cY')$. If moreover $(\cX, \cO_\cX)$, $(\cX', \cO_{\cX'})$, $(\cY, \cO_\cY) \in \RHTop(\cTank)$, the same holds after applying the hypercompletion functor $L \colon \cY' \to (\cY')^\wedge$.
		\item \label{item:effective_epi} The map $\theta' \colon f^{\prime -1} \cO_{\cX'} \to \cO_{\cY'}$ is an effective epimorphism. If moreover $(\cX, \cO_\cX)$, $(\cX', \cO_{\cX'})$, $(\cY, \cO_\cY) \in \RHTop(\cTank)$, the same holds after applying the hypercompletion functor $L \colon \cY' \to (\cY')^\wedge$
	\end{enumerate}
\end{prop}

\begin{proof}
	We first deal with the non-hypercomplete case.
	\cref{prop:alg_effective_epi} shows that the morphism $f\inv \cO_{\cX} \to \cO_{\cY}$ is an effective epimorphism.
	Moreover, $\cTank$ is unramified in virtue of \cref{cor:Tkan_unramified}.
	Therefore \cite[Theorem 1.6]{DAG-IX} implies the first two statements.
	Combining \cref{prop:unramified_transformation}, \cref{prop:alg_effective_epi} and \cite[Proposition 10.3]{DAG-IX}, we deduce the other two statements.
	
	We now assume that $(\cX, \cO_\cX), (\cX', \cO_{\cX'}), (\cY, \cO_\cY) \in \RHTop(\cTank)$.
	Then (\ref{item:pullback_structured_Top}) and (\ref{item:pullback_Top}) follow from what we already proved and the fact that $\Hyp$ commutes with limits, being a right adjoint by \cref{lem:hyp_right_adjoint}.
	On the other side, (\ref{item:pushout_analytic_algebras}) and (\ref{item:effective_epi}) follow from the fact that the hypercompletion functor $L \colon \cY' \to (\cY')^\wedge$ commutes with colimits and finite limits.
\end{proof}

\subsection{Truncations}

Now we discuss the compatibility of the pregeometry $\cTank$ with $n$-truncations.

\begin{defin}[{\cite[3.3.2]{DAG-V}}]
	Let $\cT$ be a pregeometry and let $n \ge -1$ be an integer.
	The pregeometry $\cT$ is said to be \emph{compatible with $n$-truncations} if for every \inftopos $\cX$, every $\cT$-structure $\cO \colon \cT \to \cX$ and every admissible morphism $U \to V$ in $\cT$, the induced square
	\[ \begin{tikzcd}
		\cO(U) \arrow{r} \arrow{d} & \tau_{\le n}(\cO(U)) \arrow{d} \\
		\cO(V) \arrow{r} & \tau_{\le n}(\cO(V))
	\end{tikzcd} \]
	is a pullback in $\cX$.
\end{defin}

This definition is equivalent to say that for every $\cT$-structure $\cO \colon \cT \to \cX$ the composition $\tau_{\le n} \circ \cO$ is again a $\cT$-structure and the canonical morphism $\cO \to \tau_{\le n} \circ \cO$ is a local morphism of $\cT$-structures, where $\tau_{\le n} \colon \cX \to \cX$ denotes the truncation functor of the \inftopos $\cX$.

In order to prove that $\cTank$ is compatible with $n$-truncations for every $n \ge 0$, it will be convenient to introduce a pregeometry slightly different from $\cTank$.

\begin{construction}
	We define a pregeometry $\cTan^G(k)$ as follows:
	\begin{enumerate}
		\item the underlying category of $\cTan^G(k)$ is the category of smooth $k$-analytic spaces;
		\item a morphism in $\cTan^G(k)$ is admissible if and only if it is an admissible open embedding;
		\item the topology on $\cTan^G(k)$ is the G-topology.
	\end{enumerate}
\end{construction}

\begin{lem} \label{lem:G_pregeometry_0_truncations}
	The pregeometry $\cTan^G(k)$ is compatible with $n$-truncations for every $n \ge 0$.
\end{lem}

\begin{proof}
	Since admissible open immersions are monomorphisms, the lemma is a direct consequence of \cite[3.3.5]{DAG-V}.
\end{proof}

\begin{lem} \label{lem:etale_analytic_domain}
	Let $U \to V$ be an étale morphism in $\cTank$.
	There exists a G-covering $\{V_i\to V\}_{i\in I}$, G-coverings $\{U_{ij}\to U\times_V V_i\}_{j\in J_i}$ for every $i\in I$, smooth algebraic $k$-varieties $Y_i$ and $X_{ij}$, étale morphisms $X_{ij}\to Y_i$, admissible open immersions $V_i\hookrightarrow Y_i\an$ and $U_{ij}\hookrightarrow X_{ij}\an$ such that the morphism $U_{ij}\to V_i$ equals the restriction of the morphism $X_{ij}\an\to Y_i\an$ to $V_i$ for every $i\in I$ and $j\in J_i$.
\end{lem}
\begin{proof}
	Since $V$ is smooth, there exists an affinoid G-covering $\{V_i\to V\}_{i\in I}$ such that every $V_i$ admits an étale morphism to a polydisc $\bD^{n_i}$.
	By \cite[Proposition 1.7.1]{Huber_Etale_1996}, the affinoid algebra associated to $V_i$ has a presentation of the form
	\[k\langle T_1,\dots,T_{n_i},T'_1,\dots,T'_{m_i}\rangle/(f_1,\dots,f_{m_i})\]
	such that the determinant $\det\big(\frac{\partial f_\alpha}{\partial T'_\beta}\big)_{\alpha,\beta=1,\dots,m_i}$ is invertible in $k\langle T_1,\dots,T_{n_i}\rangle$.
	By \cite[Chap.\ III Theorem 7 and Remark 2]{Elkik_Solutions_1973}, there exists a smooth affine scheme $Y_i$ and an étale morphism $Y_i\to\bbA^{n_i}_k$ such that $V_i$ is isomorphic to the fiber product $Y_i\an\times_{(\bbA^{n_i}_k)\an} \bD^{n_i}$.
	
	We now fix $i\in I$.
	Since the morphism $U\to V$ is étale, by base change, the morphism $U\times_V V_i\to V_i$ is étale.
	So the composition
	\[ U\times_V V_i\to V_i\to\bD^{n_i}\]
	is étale.
	Let $\{U_{ij}\to U\times_V V_i\}_{j\in J_i}$ be an affinoid G-covering.
	For every $j\in J_i$, by \cite[Proposition 1.7.1]{Huber_Etale_1996}, the affinoid algebra associated to $U_{ij}$ has a presentation of the form
	\[k\langle T_1,\dots,T_{n_i},T'_1,\dots,T'_{m_{ij}}\rangle/(f_1,\dots,f_{m_{ij}})\]
	such that the determinant $\det\big(\frac{\partial f_\alpha}{\partial T'_\beta}\big)_{\alpha,\beta=1,\dots,m_{ij}}$ is invertible in $k\langle T_1,\dots,T_{n_i}\rangle$.
	By \cite[Chap.\ III Theorem 7]{Elkik_Solutions_1973} again, there exists a smooth affine scheme $Z_{ij}$ and an étale morphism $Z_{ij}\to\bbA^{n_i}_k$ such that $U_{ij}$ is isomorphic to the fiber product $Z_i\an\times_{(\bbA^{n_i}_k)\an} \bD^{n_i}$.
	Let $X_{ij}\coloneqq Y_i\times_{\bbA^{n_i}_k} Z_{ij}$.
	By the universal property of the fiber product, there exists a unique map $r\colon U_{ij}\to X\an_{ij}$ making the following diagram commutative:
	\[\begin{tikzcd}
	U_{ij} \arrow{d} \arrow[swap]{rd}{r} \arrow{rrd}{t} & &\\
	U\times_V V_i \arrow{d} & X\an_{ij}\arrow{d}\arrow{r}{s} & Z\an_{ij}\arrow{d}\\
	V_i \arrow[hookrightarrow]{r} & Y\an_i \arrow{r} & (\bbA^{n_i}_k)\an
	\end{tikzcd}\]
	The map $t$ is an admissible open immersion, so it is in particular étale.
	The map $s$ is étale by base change, so it is étale.
	Since $t=s\circ r$, we deduce that the map $r$ is étale.
		Moreover, the map $t$ is a monomorphism, so the map $r$ is also a monomorphism.
	Since the map $r$ is étale, we deduce that it is an admissible open immersion.
	\end{proof}

\begin{lem} \label{lem:pullbacks_are_local}
	Let $\cX$ be an \inftopos and let
	\[ \begin{tikzcd}
		U \arrow{r} \arrow{d} & W \arrow{r} \arrow{d} & Z \arrow{d} \\
		V \arrow{r}{p} & Y \arrow{r} & X
	\end{tikzcd} \]
	be a diagram in $\cX$.
	Assume that the left and the outer squares are pullbacks and that $p$ is an effective epimorphism.
	Then the right square is a pullback as well.
\end{lem}

\begin{proof}
	Let $W' \coloneqq Y \times_X Z$.
	We obtain a commutative diagram
	\[ \begin{tikzcd}
		U \arrow{r} \arrow{d} & W' \arrow{r} \arrow{d} & Z \arrow{d} \\
		V \arrow{r}{p} & Y \arrow{r} & X .
	\end{tikzcd} \]
	Since the outer square is a pullback by our assumption, the left square is a pullback as well.
	The universal property of pullbacks induces a morphism $\alpha \colon W \to W'$.
	By hypothesis, the induced map $\alpha \times_Y V \colon W \times_Y V \to W' \times_Y V$ is an equivalence.
		Since $p$ is an effective epimorphism, the pullback functor $p\inv \colon \cX_{/Y} \to \cX_{/V}$ is conservative (cf.\ \cite[6.2.3.16]{HTT}).
	We conclude that $\alpha$ is an equivalence, completing the proof.
\end{proof}

\begin{thm} \label{thm:compatibility_truncations}
	The pregeometry $\cTank$ is compatible with $n$-truncations for every $n \ge 0$.
\end{thm}

\begin{proof}
	When $n \ge 1$, the statement is a direct consequence of \cite[3.3.5]{DAG-V}.
	We now prove the case $n = 0$.
	Let $\cX$ be an \inftopos and let $\cO \in \Strloc_{\cTank}(\cX)$.
	
	For the purpose of this proof, we will say that a morphism $f \colon U \to V$ in $\cTank$ is \emph{compatible} if the induced diagram
	\begin{equation} \label{eq:compatibility_zero_truncations}
		\begin{tikzcd}
		\cO(U) \arrow{r} \arrow{d} & \tau_{\le 0} \cO(U) \arrow{d} \\
		\cO(V) \arrow{r} & \tau_{\le 0} \cO(V)
		\end{tikzcd}
	\end{equation}
	is a pullback square.
	We need to show that every étale morphism is compatible.
	
	Let us start by observing the following properties of compatible morphisms:
	\begin{enumerate}
		\item \label{item:compatible_analytic_domain} Admissible open immersions are compatible.
		This follows from \cref{lem:G_pregeometry_0_truncations}.
		
		\item \label{item:compatible_analytification_etale} If $f \colon X \to Y$ is an \'etale morphism of smooth $k$-varieties, then the analytification $f\an \colon X\an \to Y\an$ is compatible.
		Indeed, let $\cTetk$ be the pregeometry of \cite[Definition 4.3.1]{DAG-V}.
		The analytification functor induces a morphism of pregeometries $\varphi \colon \cTetk \to \cTank$.
		We have $\cO(X\an) = (\cO \circ \varphi)(X)$ and $\cO(Y\an) = (\cO \circ \varphi)(Y)$.
		Since $\cO \circ \varphi$ is a $\cTetk$-structure on $\cX$, the statement follows from the fact that $\cTetk$ is compatible with $0$-truncations (cf.\ \cite[Proposition 4.3.28]{DAG-V}).
		
		\item \label{item:compatible_composition} Compatible morphisms are stable under composition. This follows from the composition property of pullback squares (cf.\ \cite[4.4.2.1]{HTT}).
		
		\item \label{item:compatible_pullback} Suppose given a pullback square
		\[ \begin{tikzcd}
			U \arrow{r}{g} \arrow{d}{f'} & V \arrow{d}{j} \\
			X \arrow{r}{f} & Y
		\end{tikzcd} \]
		where $f$ is compatible and $j$ is an admissible open immersion.
		Then $f'$ is compatible.
		To see this, consider the commutative diagram
		\[ \begin{tikzcd}
			\cO(U) \arrow{d} \arrow{r} & \cO(X) \arrow{d} \arrow{r} & \tau_{\le 0} \cO(X) \arrow{d} \\
			\cO(V) \arrow{r} & \cO(Y) \arrow{r} & \tau_{\le 0} \cO(Y).
		\end{tikzcd} \]
		Since admissible open immersions are in particular étale morphisms and since $\cO$ is a $\cTank$-structure, we see that the left square is a pullback diagram.
		On the other side, the right square is a pullback because $f$ is compatible by assumption.
		We conclude that the outer square in the commutative diagram
		\[ \begin{tikzcd}
			\cO(U) \arrow{r} \arrow{d} & \tau_{\le 0} \cO(U) \arrow{r} \arrow{d} & \tau_{\le 0} \cO(X) \arrow{d} \\
			\cO(V) \arrow{r} & \tau_{\le 0} \cO(V) \arrow{r} & \tau_{\le 0} \cO(Y)
		\end{tikzcd} \]
		is a pullback square.
		We remark that $\tau_{\le 0} \circ \cO$ is a $\cTan^G(k)$-structure in virtue of \cref{lem:G_pregeometry_0_truncations}.
		So by \cite[Proposition 3.3.3]{DAG-V}, the right square is a pullback as well.
		It follows that the left square is also a pullback, completing the proof of the claim.
		
		\item \label{item:compatible_G_local_source} Being compatible is G-local on the source.
		Indeed, let $f \colon X \to Y$ be a morphism in $\cTank$ and assume there exists a G-covering $\{X_i\}$ of $X$ such that each composite $f_i \colon X_i \to X \to Y$ is compatible.
		We want to prove that $f$ is compatible as well.
		Consider the commutative diagram
		\[ \begin{tikzcd}
			\coprod \cO(X_i) \arrow{r} \arrow{d} & \cO(X) \arrow{r} \arrow{d} & \cO(Y) \arrow{d} \\
			\coprod \tau_{\le 0} \cO(X_i) \arrow{r} & \tau_{\le 0} \cO(X) \arrow{r} & \tau_{\le 0} \cO(Y).
		\end{tikzcd} \]
		Since G-coverings are étale coverings, it follows from the properties of $\cTank$-structures that the total morphism $\coprod \cO(U_i) \to \cO(U)$ is an effective epimorphism. Since $\tau_{\le 0}$ commutes with coproducts (being a left adjoint) and with effective epimorphisms (cf.\ \cite[7.2.1.14]{HTT}), we conclude that the total morphism $\coprod \tau_{\le 0} \cO(U_i) \to \tau_{\le 0} \cO(U)$ is an effective epimorphism as well.
		Since each $X_i \to X$ is an admissible open immersion, Property (\ref{item:compatible_analytic_domain}) implies that the left square is a pullback.
		Moreover, the outer square is a pullback by hypothesis.
		Thus, \cref{lem:pullbacks_are_local} shows that the right square is a pullback as well, completing the proof of this property.
	\end{enumerate}
	
	Let now $f \colon U \to V$ be an étale morphism in $\cTank$.
	We will prove that $f$ is compatible.
	Using \cref{lem:etale_analytic_domain} we obtain a G-covering $\{V_i\to V\}_{i\in I}$, G-coverings $\{U_{ij}\to U\times_V V_i\}_{j\in J_i}$ for every $i\in I$, smooth algebraic $k$-varieties $Y_i$ and $X_{ij}$, étale morphisms $X_{ij}\to Y_i$, admissible open immersions $V_i\hookrightarrow Y_i\an$ and $U_{ij}\hookrightarrow X_{ij}\an$ such that the morphism $U_{ij}\to V_i$ equals to restriction of the morphism $X_{ij}\an\to Y_i\an$.
	In particular we can factor $U_{ij} \to V_i$ as the composition
	\[ \begin{tikzcd}
		U_{ij} \arrow{r} & X_{ij}\an \times_{Y_i\an} V_i \arrow{r} & V_i 
	\end{tikzcd} \]
	where the first morphism is an admissible open immersion and the second is compatible by Property (\ref{item:compatible_pullback}) of compatible morphisms.
	Therefore, Property (\ref{item:compatible_composition}) implies that $U_{ij} \to V_i$ is compatible.
	Finally, using Property (\ref{item:compatible_G_local_source}) we conclude that the morphisms $U \times_V V_i \to V_i$ are compatible.
	
	We are therefore reduced to prove the following statement: given a morphism $f \colon U \to V$, suppose that there exists a G-covering $\{v_i \colon V_i \to V\}$ such that each base change $f_i \colon U_i \coloneqq U \times_V V_i \to V_i$ is compatible, then $f$ is compatible.
	We consider the commutative diagram
	\[ \begin{tikzcd}
		\coprod \cO(U_i) \arrow{r} \arrow{d} & \cO(U) \arrow{d} \arrow{r} & \tau_{\le 0} \cO(U) \arrow{d} \\
		\coprod \cO(V_i) \arrow{r} & \cO(V) \arrow{r} & \tau_{\le 0} \cO(V).
	\end{tikzcd} \]
	Since $\cO$ is a $\cTank$-structure, the total morphism $\coprod \cO(U_i) \to \cO(U)$ is an effective epimorphism.
	Moreover, since each $V_i \to V$ is an admissible open immersion, so in particular étale, we see that the left square is a pullback.
	By hypothesis, the outer square is a pullback as well, so we conclude the proof using \cref{lem:pullbacks_are_local}.
\end{proof}

\begin{cor} \label{cor:truncation_derived_kanal_spaces}
	Let $(\cX, \cO_\cX)$ be a derived \kanal space.
	Then $(\cX, \pi_0( \cO_\cX ))$ is also a derived \kanal space.
	Moreover, we have $(\pi_0(\cO_\cX))\alg \simeq \pi_0(\cO_\cX\alg)$.
\end{cor}

\begin{proof}
	It follows from \cref{thm:compatibility_truncations} that $\pi_0(\cO_\cX)$ is a $\cTank$-structure on $\cX$.
	Let $\varphi \colon \cTdisck \to \cTank$ be the transformation of pregeometries induced by the analytification functor.
	Then we have by definition 
	\[ (\pi_0(\cO_\cX))\alg = (\pi_0^{\cX} \circ \cO_\cX ) \circ \varphi \simeq \pi_0^\cX \circ ( \cO_\cX \circ \varphi) = \pi_0( \cO_\cX\alg) , \]
	where $\pi_0^\cX$ denotes the truncation functor of the \inftopos $\cX$.
	In particular, we see that $(\cX, \pi_0(\cO_\cX)\alg)$ is a derived \kanal space.
\end{proof}

\begin{defin}
	A $\cTank$-structured $\infty$-topos $(\cX, \cO_\cX)$ is said to be \emph{discrete} if $\cO_\cX$ is a discrete object in $\Strloc_{\cTank}(\cX)$.
	We denote by $\RTop^0(\cTank)$ the full subcategory of $\RTop(\cTank)$ spanned by discrete $\cTank$-structured \inftopoi.
	
	We say that a derived \kanal space $(\cX, \cO_\cX)$ is \emph{discrete} if it is discrete as a $\cTank$-structured \inftopos.
	We denote by $\dAnk^0$ the full subcategory of $\dAnk$ spanned by discrete derived \kanal spaces.
	\end{defin}

Choose a geometric envelope $\cG_{\mathrm{an}}(k)$ for $\cTank$ and let $\cG_{\mathrm{an}}(k) \to \cG_{\mathrm{an}}^{\le 0}(k)$ be a $0$-stub for $\cG_{\mathrm{an}}(k)$ (cf.\ \cite[Definition 1.5.10]{DAG-V}).
It follows from \cite[Proposition 1.5.14]{DAG-V} that
\[ \RTop(\cG_{\mathrm{an}}^{\le 0}(k)) \simeq \RTop^0(\cTank) . \]
The relative spectrum (cf.\ \cite[§ 2.1]{DAG-V}) is a functor
\[ \trunctopoi \colon \RTop(\cG_{\mathrm{an}}(k)) \to \RTop(\cG_{\mathrm{an}}^{\le 0}(k)) \simeq \RTop^0(\cTank) , \]
which we refer to as the \emph{truncation functor}.
Using \cref{thm:compatibility_truncations}, we can identify the action of this functor on objects with the assignment
\[ (\cX, \cO_\cX) \mapsto (\cX, \pi_0(\cO_\cX)) . \]

The following proposition is an analogue of \cite[Proposition 3.13]{Porta_DCAGI} and of \cite[Proposition 2.2.4.4]{HAG-II}:

\begin{prop} \label{prop:truncation_and_finite_limits}
	Let $i \colon \dAnk^0 \to \dAnk$ denote the natural inclusion functor. Then:
	\begin{enumerate}[(i)]
		\item \label{item:truncation_derived_analytic_spaces} The functor $\trunctopoi \colon \Top(\cTank) \to \Top^0(\cTank)$ restricts to a functor $\trunc \colon \dAnk \to \dAnk^0$.
		\item \label{item:truncation_right_adjoint} The functor $i$ is left adjoint to the functor $\trunc$.
		\item \label{item:truncated_spaces_embeds_fully_faithfully} The functor $i$ is fully faithful.
	\end{enumerate}
\end{prop}

\begin{proof}
	The statement (\ref{item:truncated_spaces_embeds_fully_faithfully}) holds by definition of the functor $i$.
	It follows from \cref{cor:truncation_derived_kanal_spaces} that the functor $\trunctopoi$ respects the \infcat of derived \kanal spaces.
	Therefore the statements (\ref{item:truncation_derived_analytic_spaces}) and (\ref{item:truncation_right_adjoint}) follow immediately.
\end{proof}

\section{Fully faithful embedding of \kanal spaces} \label{sec:fullyfaithfulness}

In this section, we construct a functor $\Phi\colon\Ank\to\dAnk$ from the category of \kanal spaces to the category of derived \kanal spaces.
We will prove that $\Phi$ is a fully faithful embedding.

First we will define the functor $\Phi$ on objects, then we will define it on morphisms.

Let $\dAnk^{1, 0}$ be the full subcategory of $\dAnk$ spanned by derived \kanal spaces $(\cX, \cO_{\cX})$ such that $\cX$ is $1$-localic and $\cO_{\cX}$ is $0$-truncated.

\begin{defin}
	Let $X$ be a \kanal space and let $\cX_X$ be the étale \inftopos associated to $X$.
	We define a functor $\cO_X\colon\cTank\to\cX_X$ by the formula
	\[ \cO_X (M) (U) = \Hom_{\Ank} ( U, M ).\]
\end{defin}

\begin{lem} \label{lem:inclusion}
	Let $X$ be a \kanal space.
	Then $\cO_X$ is a 0-truncated $\cTank$-structure on the \inftopos $\cX_X$.
	Let $\cX_X^\wedge$ denote the hypercompletion of $\cX_X$ and let $\Phi(X)$ denote the pair $(\cX_X^\wedge,\cO_X)$.
	Then $\Phi(X)$ is a derived \kanal space.
\end{lem}
\begin{proof}
	In order to prove that $\cO_X$ is a $\cTank$-structure on $\cX_X$, it suffices to verify that if $\{ M_i \to M\}$ is an étale covering of $M\in\cTank$, then the induced map $\coprod_i \cO_X (M_i) \to \cO_X (M)$ is an effective epimorphism in $\cX_X$.
	Observe that for any $U$ in the étale site on $X$ and any morphism $U \to M$, there exists an étale covering $\{U_j\to U\}$ such that the composite morphisms $U_j\to M$ factor though $\coprod M_i \to M$.
	So we conclude using \cite[Corollary 2.9]{Porta_Yu_Higher_analytic_stacks_2014}.
	
	Since $\cO_X$ is $0$-truncated by construction, it is hypercomplete.
	Therefore the second statement follows from the first.
\end{proof}

In order to define the functor $\Phi$ on morphisms, our strategy is to prove that the mapping spaces $\Map_{\dAnk}(\Phi(X), \Phi(Y))$ are discrete for all $X, Y \in \Ank$ (cf.\ \cref{prop:discrete_mapping_spaces_I}).
In this way we can promote $\Phi$ to an $\infty$-functor without the need to specify higher homotopies.

We begin by introducing an auxiliary functor $\Upsilon$.
Let $\LRT$ denote the $2$-category of locally ringed 1-topoi and let $\Upsilon\colon\Ank\to\LRT$ be the functor sending every \kanal space to the associated locally ringed étale 1-topos.
For $X\in\Ank$, we denote by $\cO_X\alg$ the structure sheaf of $k$-algebras of $\Upsilon(X)$.

\begin{lem}\label{lem:affine_embedding}
	Let $X$ be a $k$-affinoid space, $Y$ a $k$-analytic space and $\alpha\colon X\to Y$ a morphism.
	Then there exists a positive integer $N$ and a monomorphism $\beta\colon X\hookrightarrow\bD^N_Y$ over $Y$, where $\bD^N_Y$ denotes the unit polydisc over $Y$.
\end{lem}
\begin{proof}
		Let $A\coloneqq\Gamma(\cO_X)$.
	Write $A=k\langle x_1,\dots,x_n\rangle /I$ as a quotient of a Tate algebra.
	Denote by $a_1,\dots,a_n$ the images of $x_1,\dots,x_n$ in $A$.
	We cover $X$ by finitely many rational domains $U_i$ such that $\alpha(U_i)$ is contained in an affinoid domain $V_i\subset Y$.
	Write \[\Gamma(\cO_{U_i})=A\Big\langle\frac{b_{i1}}{b_{i0}},\dots,\frac{b_{in_i}}{b_{i0}}\Big\rangle,\]
	where $b_{i0},\dots,b_{in_i}$ is a collection of elements in $A$ with no common zero.
	Let $c_{i0},\dots,c_{in_i}$ be elements in $k$ such that $\abs{c_{ij}}\ge\rho(b_{ij})$ for $j=0,\dots,n_i$, where $\rho(\cdot)$ denotes the spectral radius.
	Consider the morphism
	\[\Gamma(\cO_{V_i})\langle y_1,\dots,y_n, y_{i0},\dots,y_{in_i}\rangle\to\Gamma(\cO_{U_i})\]
	that sends $y_j$ to $a_j$ and $y_{ij}$ to $b_{ij}/c_{ij}$.
	It induces a monomorphism $U_i\hookrightarrow\bD^{n+n_i+1}_{V_i}$.
	
	Let $N\coloneqq n+\sum_{i=1}^m(n_i+1)$.
	Consider the unit polydisc $\bD^N_Y$ over $Y$.
	We denote by $y_i,y_{ij}$ for $i=1,\dots,m$, $j=0,\dots,n_i$ the coordinate functions on $\bD^N_Y$.
	Let $\beta\colon X\to\bD^N_Y$ be the morphism that sends $y_i$ to $a_i$ and $y_{ij}$ to $b_{ij}/c_{ij}$ for all $i=1,\dots,m$, $j=0,\dots,n_i$.
	Let $Z_i$ be the admissible open subset in $\bD^N_Y$ given by the inequalities $\abs{c_{i0}\cdot y_{ij}}\le\abs{c_{i0}\cdot y_{i0}}$ for $j=1,\dots,n_i$.
	Let $Z'_i\coloneqq Z_i\times_Y V_i$.
	We see that $\beta\inv(Z'_i)$ is $U_i$.
	By construction, $\beta|_{U_i}\colon U_i\to Z'_i$ is a monomorphism.
	We conclude that $\beta\colon X\to\bD^N_Y$ is a monomorphism.
\end{proof}

\begin{lem}\label{lem:rigidity}
	Let $f\colon X\to Y$ be a morphism of \kanal spaces.
	Let
	\[(f,f^\#)\colon\big(\Sh_\rSet (X\et),\cO\alg_X\big) \to \big(\Sh_\rSet (Y\et),\cO\alg_Y\big)\]
	denote the induced morphism of locally ringed 1-topoi.
	Let $t$ be a 2-morphism from $(f,f^\#)$ to itself.
	Then $t$ equals the identity.
\end{lem}
\begin{proof}
	Using \cref{lem:affine_embedding}, the same proof of \cite[Tag 04IJ]{stacks-project} applies.
\end{proof}

\begin{lem}\label{lem:first_fully_faithful}
	The functor
	\begin{align*}
	\Upsilon \colon \Ank &\longrightarrow \LRT\\
	X &\longmapsto (\Sh_\rSet(X\et), \cO_X\alg)
	\end{align*}
	is fully faithful.
\end{lem}

\begin{proof}
	Let $X,Y$ be two \kanal spaces.
	Let \[(g,g^\#)\colon\big(\Sh_\rSet (X\et),\cO\alg_X\big) \to \big(\Sh_\rSet (Y\et),\cO\alg_Y\big)\]
	be a morphism of locally ringed 1-topoi.
	We would like to show that there exists a unique morphism of \kanal spaces $f\colon X\to Y$ which induces $(g,g^\#)$.
	We proceed along the same lines as \cite[Tag 04JH]{stacks-project}.
	
	Let $g\inv\colon\Sh_\rSet (Y\et) \leftrightarrows \Sh_\rSet(X\et) \colon g_*$ denote the morphism of 1-topoi.
	
	First, we assume that $X=\Sp A$, $Y=\Sp B$ for some $k$-affinoid algebras $A$ and $B$.
	Since $B=\Gamma(Y\et,\cO\alg_Y)$ and $A=\Gamma(X\et,\cO\alg_X)$, we see that $g^\#$ induces a map of affinoid algebras $\varphi\colon B\to A$.
	Let $f=\Sp\varphi\colon X\to Y$.
	Let us show that $f$ induces $(g,g^\#)$.
	
	Let $V\to Y$ be an affinoid space étale over $Y$.
	Assume $V=\Sp C$.
	By \cite[Proposition 1.7.1]{Huber_Etale_1996}, we can write
	\[ C=B\langle x_1,\dots,x_n\rangle / (r_1,\dots,r_n), \]
	where $r_1,\dots,r_n\in B\langle x_1,\dots,x_n\rangle$ and the determinant $\mathrm{Jac}(r_1,\dots,r_n)$ is invertible in $C$.
	Now the sheaf $h_V$ on $Y\et$ is the equalizer of the two maps
	\[
	\xymatrix{
		\prod_{i=1}^n \cO\alg_Y \ar@<0.6ex>[r]^{a} \ar@<-0.6ex>[r]_{b} & \prod_{j=1}^n \cO\alg_Y
	}
	\]
	where $b=0$ and $a(h_1,\dots,h_n)=\big(r_1(h_1,\dots,h_n),\dots,r_n(h_1,\dots,h_n)\big)$.
	We have the following commutative diagram
	\begin{equation}\label{eq:equalizers}
	\xymatrix{
		g\inv h_V \ar@{.>}[d]^{\alpha} \ar[r] & \prod_{i=1}^n g\inv \cO\alg_Y  \ar[d]^{\prod g^\#} \ar@<0.6ex>[r]^{g\inv a} \ar@<-0.6ex>[r]_{g\inv b} & \prod_{j=1}^n g\inv\cO\alg_Y \ar[d]^{\prod g^\#}\\
		h_{X\times_Y V}\ar[r] & \prod_{i=1}^n \cO\alg_X \ar@<0.6ex>[r]^{a'} \ar@<-0.6ex>[r]_{b'} & \prod_{j=1}^n \cO\alg_X,
	}
	\end{equation}
	where $b'=0$, $a'(h_1,\dots,h_n)=\big(\varphi(r_1)(h_1,\dots,h_n),\dots,\varphi(r_n)(h_1,\dots,h_n)\big)$, the two horizontal lines are equalizer diagrams and the dotted arrow $\alpha$ is obtained by the universal property of equalizers.
	
	We claim that the map $\alpha\colon g\inv h_V\to h_{X\times_Y V}$ is an isomorphism.
	Let us check this on the stalks.
	Let $\bar x$ be a geometric point of the adic space $X^\mathrm{ad}$ associated to $X$ in the sense of Huber.
	Denote by $p$ the associated point of the 1-topos $\Sh_\rSet(X\et)$ (cf.\ \cref{rem:points_of_cX_X}).
	Applying localization at $p$ to Diagram \eqref{eq:equalizers}, we would like to show that $\alpha_p\colon (g\inv h_V)_p\to (h_{X\times_Y V})_p$ is an isomorphism.
	Set $q\coloneqq g\circ p$.
	This is a point of the 1-topos $\Sh_\rSet(Y\et)$.
	We denote by $\bar y$ the corresponding geometric point of the adic space $Y^\mathrm{ad}$ associated to $Y$.
	Then the localization of the map $g^\#$ at $p$ has the following description
	\[(g^\#)_p\colon\cO\alg_{Y,\bar y} = \cO\alg_{Y,q} = (g\inv \cO\alg_Y)_p \longrightarrow \cO\alg_{X,p} = \cO\alg_{X,\bar x}.\]
	
	It suffices to treat the two cases: either $V\to Y$ is finite étale, or $V\to Y$ is an affinoid domain embedding.
	In the first case, there exists an étale neighborhood $U$ of $\bar y$ in $Y^\mathrm{ad}$ such that the pullback morphism $V\times_Y U\to U$ splits.
	Then the equalizer of
	\begin{equation} \label{eq:O_Y_V}
	\xymatrix{
		\prod_{i=1}^n \cO\alg_Y(U)  \ar@<0.6ex>[r]^{a} \ar@<-0.6ex>[r]_{b} & \prod_{j=1}^n \cO\alg_Y(U)
	}
	\end{equation}
	is isomorphic to the equalizer of
	\begin{equation} \label{eq:ky}
	\xymatrix{
		\prod_{i=1}^n k(\bar y) \ar@<0.6ex>[r]^{a} \ar@<-0.6ex>[r]_{b} & \prod_{j=1}^n k(\bar y),
	}
	\end{equation}
	where $k(\bar y)$ denotes the residue field of $\bar y$.
	Similarly, there exists an étale neighborhood $U'$ of $\bar x$ in $X^\mathrm{ad}$ such that the pullback morphism $X\times_Y V\times_X U'\simeq V\times_Y U'\to U'$ splits.
	Then the equalizer of
	\begin{equation} \label{eq:O_X_V}
	\xymatrix{
		\prod_{i=1}^n \cO\alg_X(U') \ar@<0.6ex>[r]^{a'} \ar@<-0.6ex>[r]_{b'} & \prod_{j=1}^n \cO\alg_X(U')
	}
	\end{equation}
	is isomorphic to the equalizer of
	\begin{equation} \label{eq:kx}
	\xymatrix{
		\prod_{i=1}^n k(\bar x) \ar@<0.6ex>[r]^{a'} \ar@<-0.6ex>[r]_{b'} & \prod_{j=1}^n k(\bar x).
	}
	\end{equation}
	Since the equalizer of \cref{eq:ky} and the equalizer of \cref{eq:kx} are isomorphic by construction,
	we deduce that the equalizer of \cref{eq:O_Y_V} and the equalizer of \cref{eq:O_X_V} are isomorphic.
	Taking colimits over all such étale neighborhoods, we conclude that $\alpha_p\colon (g\inv h_V)_p\to (h_{X\times_Y V})_p$ is an isomorphism.
	Then let us consider the second case where $V\to Y$ is an affinoid domain embedding.
	If the geometric point $\bar y$ can be lifted to a geometric point in $V$, then for any étale neighborhood $U$ of $\bar y$ in $Y^\mathrm{ad}$ refining $V$, the equalizer of \cref{eq:O_Y_V} consists of a single element.
	The same goes for the equalizer of \cref{eq:O_X_V}.
	If the geometric point cannot be lifted to a geometric point in $V$, then the equalizer of \cref{eq:O_Y_V} is empty, so is the equalizer of \cref{eq:O_X_V}.
	We conclude that $\alpha_p\colon (g\inv h_V)_p\to (h_{X\times_Y V})_p$ is an isomorphism.
	
	Now the same argument in \cite[Tag 04I6]{stacks-project} shows that the isomorphisms $g\inv h_V\to h_{X\times_Y V}$ are functorial with respect to $V$ and that the map $f\colon X\to Y$ indeed induces the morphism of locally ringed 1-topoi $(g,g^\#)$ we started with.
	Finally, the argument in \cite[Tag 04I7]{stacks-project} allows us to deduce the general case from the affinoid case considered above.
\end{proof}


\begin{lem}\label{lem:alg_faithful}
	Let $\cX$ be an \inftopos.
	The induced functor
	\[
	\Strloc_{\cTank}(\tau_{\le 0} \cX) \to \Strloc_{\cTdisck}(\tau_{\le 0} \cX)
	\]
	is faithful.
\end{lem}

\begin{proof}
	We can factor the functor $\Strloc_{\cTank}(\tau_{\le 0} \cX) \to \Strloc_{\cTdisck}(\tau_{\le 0} \cX)$ as
	\[ \begin{tikzcd}
		\Strloc_{\cTank}(\tau_{\le 0} \cX) \arrow{r} & \Strloc_{\cTetk}(\tau_{\le 0} \cX) \arrow{r}{U} & \Strloc_{\cTdisck}(\tau_{\le 0} \cX) ,
	\end{tikzcd} \]
	where $\cTetk$ is the pregeometry introduced in \cite[Definition 4.3.1]{DAG-V}.
	Combining \cite[Propositions 4.3.16, 2.6.16 and Remark 2.5.13]{DAG-V} we see that the functor $U$ is faithful.
	So we are reduced to prove the same statement for
	\[ \Strloc_{\cTank}(\tau_{\le 0} \cX) \to \Strloc_{\cTetk}(\tau_{\le 0} \cX) . \]
	
	The mapping spaces of $\tau_{\le 0} \cX$ are discrete by definition.
	It follows from \cite[2.3.4.18]{HTT} that we can find a minimal $1$-category $\cD$ and a categorical equivalence $\tau_{\le 0} \cX \simeq \cD$.
	Let $F, G \in \Strloc_{\cTank}(\cD)$. We want to show that the natural morphism
	\[
	\Map_{\Strloc_{\cTank}(\cD)}(F,G) \to \Map_{\Strloc_{\cTetk}(\cD)}(F\alg, G\alg)
	\]
	is a homotopy monomorphism.
	Since $F$ and $G$ take values in the $1$-category $\cD$, both mapping spaces above are sets.
	We want to prove that the given map is a monomorphism.
	Since $\Strloc_{\cTetk}(\cD)$ is a $1$-category,  two natural transformations $\varphi$ and $\psi$ represent the same object in $\Map_{\Strloc_{\cTetk}(\cD)}(F\alg, G\alg)$ if and only if they are equal, in the sense that
	\[
	\begin{tikzcd}
	F\alg(X) \arrow{r}{\varphi\alg_{X}} \arrow{d}[swap]{\mathrm{id}_{F\alg(X)}} & G\alg(X) \arrow{d}{\mathrm{id}_{G\alg(X)}} \\
	F\alg(X) \arrow{r}{\psi\alg_{X}} & G\alg(X)
	\end{tikzcd}
	\]
	commutes for every $X \in \cTetk$.
	Fix $U \in \cTank$.
	We first assume that $U$ is isomorphic to an affinoid domain in $X\an$ for a smooth $k$-variety $X$.
	
	Since $U \to X\an$ is a monomorphism, we have a pullback square
	\[\begin{tikzcd}
	U \arrow[-, double equal sign distance]{r} \arrow[-, double equal sign distance]{d} & U \arrow[hook]{d} \\
	U \arrow[hook]{r} & X\an.
	\end{tikzcd}\]
	Since $U \to X$ is an affinoid embedding, it is étale, so it is an admissible morphism in $\cTank$.
	Applying the functor $F$, we obtain another pullback square
	\[\begin{tikzcd}
	F(U) \arrow[-, double equal sign distance]{r} \arrow[-, double equal sign distance]{d} & F(U) \arrow{d} \\
	F(U) \arrow{r} & F(X\an).
	\end{tikzcd}\]
	So $F(U) \to F(X\an)$ is a monomorphism in the category $\cD$.
	
	We have a commutative cube
	\[
	\begin{tikzcd}[column sep=small, row sep=small]
	{} & F(U) \arrow{dl} \arrow{rr} \arrow[dotted]{dd} & & F(X\an) \arrow{dl} \\
	G(U) \arrow[crossing over]{rr}  & & G(X\an)  \\
	{} & F(U) \arrow{rr} \arrow{dl} & & F(X\an) \arrow[-, double equal sign distance]{uu} \arrow{dl} \\
	G(U) \arrow{rr} \arrow[-, double equal sign distance]{uu} & & G(X\an) \arrow[-, double equal sign distance, crossing over]{uu}
	\end{tikzcd}
	\]
	where the dotted arrow exists by the universal property of the pullbacks.
	Since $F(U) \to F(X\an)$ is a monomorphism, the dotted arrow is in fact the identity of $F(U)$.
	
	Let us now consider a general $U \in \cTank$.
	Choose a G-covering of $U$ by affinoid domains $\{U_i \to U\}$ such that each $U_i$ is isomorphic to an affinoid domain in $X_i\an$ for some smooth $k$-variety $X_i$.
		Set $U^0 \coloneqq \coprod U_i$ and consider the \v{C}ech nerve $U^\bullet \to U$. Observe that both $F(U^\bullet)$ and $G(U^\bullet)$ are groupoid objects in the 1-topos $\cD$ and that their realizations are respectively $F(U)$ and $G(U)$. Since we have a commutative square of groupoids
	\[
	\begin{tikzcd}
	F(U^\bullet) \arrow{r}{\varphi_{U^\bullet}} \arrow[-, double equal sign distance]{d} & G(U^\bullet) \arrow[-, double equal sign distance]{d} \\
	F(U^\bullet) \arrow{r}{\psi_{U^\bullet}} & G(U^\bullet),
	\end{tikzcd}
	\]
	the square
	\[
	\begin{tikzcd}
	F(U) \arrow{r}{\varphi_U} \arrow[-, double equal sign distance]{d} & G(U) \arrow[-, double equal sign distance]{d} \\
	F(U) \arrow{r}{\psi_U} & G(U)
	\end{tikzcd}
	\]
	commutes as well.
	Since the identity is functorial, the proof is now complete.
\end{proof}

\begin{lem}\label{lem:1_truncated_mapping_space}
	Let $\cT$ be a pregeometry and let $(\cX, \cO_\cX)$, $(\cY, \cO_\cY)$ be $\cT$-structured \inftopoi such that $\cX$ and $\cY$ are $1$-localic and the structure sheaves $\cO_\cX$, $\cO_\cY$ are discrete.
	Then $\Map_{\RTop(\cT)}((\cX, \cO_{\cX}), (\cY, \cO_{\cY}))$ is $1$-truncated.
	Moreover, the canonical morphism
	\[ \Map_{\RTop(\cT)}((\cX, \cO_{\cX}), (\cY, \cO_{\cY})) \to \Map_{\RTop_1(\cT)}((\tau_{\le 0} \cX, \cO_{\cX}), (\tau_{\le 0} \cY, \cO_{\cY})) \]
	is an equivalence, where $\RTop_1$ denotes the $\infty$-category of 1-topoi with morphisms being right adjoint geometric morphisms.
\end{lem}

\begin{proof}
	Consider the coCartesian fibration $\LTop(\cT) \to \LTop$.
		We know from \cite[Remark 1.4.10]{DAG-V} that the fiber over an \inftopos $\cX$ is equivalent to $\Strloc_{\cT}(\cX)$.
	Let $f^{-1} \colon \cX \rightleftarrows \cY \colon f_*$
	be a geometric morphism between $\cX$ and $\cY$.
	Using \cite[2.4.4.2]{HTT} and \cite[Remark 1.4.10]{DAG-V} we obtain a fiber sequence
	\[ \begin{tikzcd}
		\Map_{\Strloc_{\cTank}(\cX)}(f^{-1} \cO_\cY, \cO_\cX) \arrow{r} & \Map_{\RTop(\cT)}((\cX, \cO_{\cX}), (\cY, \cO_{\cY})) \arrow{r} & \Fun_*(\cX, \cY) ,
	\end{tikzcd} \]
	where the fiber is taken at the geometric morphism $(f\inv, f_*)$.
	Since both $\cX$ and $\cY$ are $1$-localic, there is an equivalence
	\[ \Fun_*(\cX, \cY) \simeq \Fun_*(\tau_{\le 0} \cX, \tau_{\le 0} \cY). \]
	Therefore $\Fun_*(\cX, \cY)$ is $1$-truncated.
	On the other side, $\cO_\cX$ is $0$-truncated, so $\Map_{\Strloc_{\cT}(\cX)}(f^{-1} \cO_\cY, \cO_\cX)$ is discrete.
	The second statement follows as well.
\end{proof}

\begin{lem} \label{lem:alg_homotopy_monomorphism}
	Let $X = (\cX, \cO_{\cX})$ and $Y = (\cY, \cO_{\cY})$ be two $\cTank$-structured \inftopoi.
	Let $X\alg \coloneqq (\cX, \cO_\cX\alg)$ and $Y\alg \coloneqq (\cY, \cO_\cY\alg)$ be the underlying $\cTdisck$-structured \inftopoi.
	Assume that $\cX$ and $\cY$ are $1$-localic and that $\cO_{\cX}$ and $\cO_{\cY}$ are $0$-truncated.
	Then the canonical map
	\[ \Map_{\RTop(\cTank)}(X, Y) \to \Map_{\RTop(\cTdisck)}(X\alg, Y\alg) \]
	induces monomorphisms on $\pi_0$ and on $\pi_1$ (for every choice of base point).
\end{lem}

\begin{proof}
	Let $f_* \colon \cX \rightleftarrows \cY \colon f\inv$ be a geometric morphism in $\RTop$.
	We have a commutative diagram in $\cS$:
	\[ \begin{tikzcd}[column sep=small]
		\Map_{\Strloc_{\cTank}(\cX)}(f\inv \cO_\cY, \cO_{\cX}) \arrow{r} \arrow{d} & \Map_{\RTop(\cTank)} (X, Y) \arrow{r} \arrow{d} & \Fun_*(\cX, \cY) \arrow[-, double equal sign distance]{d} \\
		\Map_{\Strloc_{\cTdisck}(\cX)}(f\inv \cO_{\cY}\alg, \cO_{\cX}\alg) \arrow{r} & \Map_{\RTop(\cTdisck)}(X\alg, Y\alg) \arrow{r} & \Fun_*(\cX, \cY).
	\end{tikzcd} \]
	Using \cite[2.4.4.2]{HTT} and \cite[Remark 1.4.10]{DAG-V} we see that the two horizontal lines are fiber sequences.
	Moreover, since $\cO_\cX$ and $\cO_\cY$ are $0$-truncated, we can use \cref{lem:alg_faithful} to deduce that the first vertical map is a homotopy monomorphism.
	Passing to the long exact sequences of homotopy groups and applying the five lemma, we obtain monomorphisms
	\begin{gather*}
		\pi_0 \Map_{\RTop(\cTank)}(X, Y) \to \pi_0 \Map_{\RTop(\cTdisck)}(X\alg, Y\alg) , \\
		\pi_1 \Map_{\RTop(\cTank)}(X, Y) \to \pi_1 \Map_{\RTop(\cTdisck)}(X\alg, Y\alg) ,
	\end{gather*}
	completing the proof.
\end{proof}

\begin{lem} \label{lem:mapping_space_hypercompletion_localic}
	Let $\cY$ be an $n$-localic \inftopos and let $\cX$ be any \inftopos.
	Then there is a canonical equivalence in the homotopy category of spaces $\mathcal H$:
	\[ \Map_{\RTop}(\cX^\wedge, \cY^\wedge) \simeq \Map_{\RTop}(\cX, \cY) . \]
\end{lem}

\begin{proof}
	Using \cite[6.5.2.13]{HTT} we see that the canonical morphism
	\[ \Map_{\RTop}(\cX^\wedge, \cY^\wedge) \to \Map_{\RTop}(\cX^\wedge, \cY) \]
	is an equivalence.
	Since $\cY$ is $n$-localic, the restriction
	\[ \Map_{\RTop}(\cX^\wedge, \cY) \to \Map_{\RTop_n}(\tau_{\le n - 1} (\cX^\wedge), \tau_{\le n - 1} \cY) \]
	is an equivalence as well.
	On the other side, the restriction
	\[ \Map_{\RTop}(\cX, \cY) \to \Map_{\RTop_n}(\tau_{\le n - 1}\cX, \tau_{\le n - 1} \cY) \]
	is also an equivalence.
	We now conclude by observing that $\tau_{\le n - 1} \cX \simeq \tau_{\le n - 1}(\cX^\wedge)$.
\end{proof}

\begin{prop} \label{prop:discrete_mapping_spaces_I}
	Let $X, Y \in \Ank$. Then $\Map_{\RTop(\cTank)}(\Phi(X), \Phi(Y))$ is discrete.
\end{prop}

\begin{proof}
	It follows from \cref{lem:mapping_space_hypercompletion_localic} that
	\begin{equation} \label{eq:from_hypercomplete_to_non-hypercomplete}
		\Map_{\dAnk}(\Phi(X), \Phi(Y)) \simeq \Map_{\RTop(\cTank)}((\cX_X, \cO_X), (\cX_Y, \cO_Y) ) .
	\end{equation}
	On the other side, \cref{lem:1_truncated_mapping_space} shows that the right hand side is $1$-truncated and
	\begin{equation} \label{eq:from_sheaves_of_spaces_to_sheaves_of_sets}
		\Map_{\RTop(\cTank)}((\cX_X, \cO_X), (\cX_Y, \cO_Y) ) \simeq \Map_{\RTop_1(\cTank)}((\tau_{\le 0} \cX_X, \cO_X), (\tau_{\le 0} \cX_Y, \cO_Y))
	\end{equation}
	We can now apply \cref{lem:alg_homotopy_monomorphism} to conclude that the canonical map
	\begin{multline*}
	\Map_{\RTop_1(\cTank)}((\tau_{\le 0} \cX_X, \cO_X), (\tau_{\le 0} \cX_Y, \cO_Y)) \to\\
	\Map_{\RTop_1(\cTdisck)}((\tau_{\le 0} \cX_X, \cO_X\alg), (\tau_{\le 0}\cX_Y, \cO_Y\alg))
	\end{multline*}
	induces monomorphisms on $\pi_0$ and on $\pi_1$.
	
	It follows from \cref{lem:first_fully_faithful} that the canonical map
	\[ \Hom_{\Ank}(X,Y) \to \pi_0 \Map_{\RTop_1(\cTdisck)}((\tau_{\le 0} \cX_X, \cO_X\alg), (\tau_{\le 0} \cX_Y, \cO_Y\alg)) \]
	is an isomorphism.
	At this point, we can invoke \cref{lem:rigidity} to deduce that, for every choice of base point, we have
	\[ \pi_1 \Map_{\RTop_1(\cTdisck)}((\tau_{\le 0} \cX_X, \cO_X\alg), (\tau_{\le 0} \cX_Y, \cO_Y\alg)) = 0 . \]
		Thus, we conclude that
	\[ \pi_1 \Map_{\RTop_1(\cTank)}((\tau_{\le 0} \cX_X, \cO_X), (\tau_{\le 0} \cX_Y, \cO_Y)) = 0 \]
	for every choice of base point.
	It follows from the equivalences \eqref{eq:from_hypercomplete_to_non-hypercomplete} and \eqref{eq:from_sheaves_of_spaces_to_sheaves_of_sets} that $\Map_{\dAnk}(\Phi(X),\allowbreak \Phi(Y))$ is discrete, completing the proof.
\end{proof}

We can now promote $X \mapsto \Phi(X)$ to an $\infty$-functor.

\todo{Be careful with $X\et$ below.}
Let $\cC$ temporarily denote the full subcategory of $\dAnk$ spanned by the objects which are equivalent to $\Phi(X)$ for some $X \in \Ank$.
\cref{prop:discrete_mapping_spaces_I} shows that mapping spaces in $\cC$ are discrete, hence $\cC$ is equivalent to a $1$-category.
Fix a morphism $f \colon X \to Y$ in $\Ank$.
It induces a morphism of sites
\[ \varphi \colon Y\et \to X\et \]
given by base change along $f$.
Since all the morphisms in $X\et$ and $Y\et$ are étale, it follows that $X\et$ and $Y\et$ have fiber products.
Moreover, $\varphi$ is left exact.
Therefore, it follows from \cite[Lemma 2.16]{Porta_Yu_Higher_analytic_stacks_2014} that the induced adjunction
\[ \varphi^s \colon \cX_Y \rightleftarrows \colon \cX_X \colon \varphi_s \]
is a geometric morphism of \inftopoi.
In particular, we obtain an induced geometric morphism $\cX_Y^\wedge \rightleftarrows \cX_X^\wedge$, which we denote by
\[ f\inv \colon \cX^\wedge \rightleftarrows \cX_X^\wedge \colon f_* . \]
We obtain in this way a well defined morphism $(\cX_X^\wedge, \cO_X) \to (\cX_Y^\wedge, \cO_Y)$.
Since mapping spaces in $\cC$ are discrete, we see that this assignment is functorial.
We denote the resulting $\infty$-functor by
\[ \Phi \colon \Ank \to \dAnk . \]

\begin{thm} \label{thm:fully_faithfulness}
The functor $\Phi\colon \Ank \rightarrow \dAnk$ is fully faithful.
\end{thm}

\begin{proof}
	Let $X, Y \in \Ank$. We want to show that
	\[
	\Hom_{\Ank}(X,Y) \to \Map_{\dAnk}(\Phi(X), \Phi(Y))
	\]
	is an equivalence.
	\cref{lem:1_truncated_mapping_space} allows us to identify $\Map_{\dAnk}(\Phi(X), \Phi(Y))$ with
	\[
	\Map_{\RTop_1(\cTank)}\big((\Sh_\rSet(X\et), \cO_X), (\Sh_\rSet(Y\et), \cO_Y)\big) .
	\]
	
	Let us first prove the faithfulness.
	Let $f, g \colon X \to Y$ be two morphisms and assume that $\Phi(f) = \Phi(g)$.
	Since the question of $f$ being equal to $g$ is local on both $X$ and $Y$, we can assume that both $X$ and $Y$ are affinoid.
	In this case, $f$ (resp.\ $g$) can be recovered as global section of the natural transformation $\Phi(f)(\bA^1_k)$ (resp.\ $\Phi(g)(\bA^1_k)$), where $\bA^1_k$ denote the \kanal affine line.
	Therefore we have $f = g$.
	
	Let us now turn to the fullness.
	Let
	\[ (f, f^\sharp) \colon (\Sh_\rSet(X\et), \cO_X) \to (\Sh_\rSet(Y\et), \cO_Y) \]
	be a morphism in $\RTop(\cTank)$.
	After forgetting along the morphism $\cTdisck \to \cTank$, we get a morphism of locally ringed 1-topoi.
	\cref{lem:first_fully_faithful} implies that this morphism comes from a map $\varphi \colon X \to Y$.
	This means that $\Phi(\varphi)\alg$ and $(f,f^\sharp)\alg$ coincide.
	\cref{lem:alg_faithful} implies that $\Phi(\varphi)$ and $(f,f^\sharp)$ coincide as well, completing the proof.
\end{proof}

\section{Closed immersions and étale morphisms} \label{sec:closed_etale}

In this section, we study closed immersions and étale morphisms under the fully faithful embedding $\Phi\colon\Ank\to\dAnk$.

\begin{defin}[{\cite[1.1]{DAG-IX}},{\cite[2.3.1]{DAG-V}}]\label{def:closed_immersion_and_etale}
	Let $\cT$ be a pregeometry and $\cX$, $\cY$ two \inftopoi.
	A morphism $\cO\to\cO'$ in $\Strloc_\cT(\cX)$ is said to be an \emph{effective epimorphism} if for every object $X\in\cT$, the induced map $\cO(X)\to\cO'(X)$ is an effective epimorphism in $\cX$.
	A morphism $f\colon(\cX, \cO_\cX)\to(\cY, \cO_\cY)$ in $\RTop(\cT)$ is called a \emph{closed immersion} (resp.\ an \emph{étale morphism}) if the following conditions are satisfied:
	\begin{enumerate}[(i)]
		\item the underlying geometric morphism $f_*\colon\cX\to\cY$ is a closed immersion (resp.\ an étale morphism) of \inftopoi;
		\item the morphism of structure sheaves $f\inv \cO_\cY \to \cO_\cX$ is an effective epimorphism (resp.\ an equivalence) in $\Strloc_\cT(\cY)$.
	\end{enumerate}
\end{defin}

\begin{lem} \label{lem:hypercompletion_closed_immersions}
	The hypercompletion functor $\RTop \to \RHTop$ preserves closed immersions.
\end{lem}

\begin{proof}
	Let $f_* \colon \cX \rightleftarrows \cY \colon f\inv$ be a closed immersion of \inftopoi.
	By definition we can find a $(-1)$-truncated object $U \in \cY$ such that the geometric morphism $f_*$ is equivalent to the induced geometric morphism $j_* \colon \cY / U \rightleftarrows \cY \colon j\inv$.
	Since $U$ is $(-1)$-truncated, it belongs to $\cY^\wedge$.
	It is therefore enough to prove that $(\cY / U)^\wedge \simeq \cY^\wedge / U$.
	The geometric morphism $\cY / U \to \cY$ induces by passing to hypercompletions a morphism $(\cY / U)^\wedge \to \cY^\wedge$ which by construction fits in the commutative diagram
	\[ \begin{tikzcd}
		\cY / U \arrow{r}{j_*} & \cY \\
		(\cY / U)^\wedge \arrow{u}{i_{U*}} \arrow{r}{j_*^\wedge} & \cY^\wedge \arrow{u}{i_*} .
	\end{tikzcd} \]
	Since $j_*$, $i_*$ and $i_{U*}$ are fully faithful, the same goes for $j_*^\wedge$.
	Observe that by \cite[7.3.2.5]{HTT}, an object $V \in \cY^\wedge$ belongs to $\cY^\wedge / U$ if and only if $V \times U \simeq U$.
	Since both $i_*$ and $j_*$ commute with products, we conclude that $j_*^\wedge$ factors through $\cY^\wedge / U$.
	
	This provides us a fully faithful functor $(\cY / U)^\wedge \to \cY^\wedge / U$.
	In order to complete the proof, it is enough to prove that it is essentially surjective.
	The canonical map $\cY^\wedge / U \to \cY^\wedge \to \cY$ factors through $\cY / U$.
	Now it suffices to prove that this functor can be further factored through $(\cY / U)^\wedge$.
	This follows from the fact that $j_*$ respects the collection of $\infty$-connected morphisms.
		To see this, let $V \in \cY / U$. Since $U$ is $(-1)$-truncated, we see that for every $n \ge 0$ one has:
	\[ \tau_{\le n}(V) \times U \simeq \tau_{\le n}(V) \times \tau_{\le n}(U) \simeq \tau_{\le n}(V \times U) \simeq \tau_{\le n}(U) \simeq U . \]
	In particular, $\tau_{\le n}(V)$ belongs to $\cY / U$ as well.
	It follows that $j_*$ commutes with truncations, and therefore with $\infty$-connected morphisms.
\end{proof}

\begin{lem} \label{lem:different_closed_immersion}
	Let $f^{-1} \colon \cX \rightleftarrows \cY \colon f_*$ be a closed immersion of \inftopoi.
	Let $F \in \cX$, $G \in \cY$ and let $f^{-1} F \to G$ be a morphism in $\cY$.
	If the morphism $F \to f_* G$ is an effective epimorphism, then so is the morphism $f\inv F \to G$.
\end{lem}

\begin{proof}
	Since $f\inv$ is left exact, it commutes with effective epimorphisms.
	Therefore, $f\inv F \to f\inv f_* G$ is an effective epimorphism.
	Since $f_*$ is fully faithful, we see that $f\inv f_* G \simeq G$, hence completing the proof.
\end{proof}

\begin{thm} \label{thm:Phi_classes_of_morphisms}
	Let $f \colon X \to Y$ be a morphism in $\Ank$.
	Then:
	\begin{enumerate}[(i)]
		\item The morphism $f$ is an étale morphism if and only if $\Phi(f)$ is an \'etale morphism.
		\item The morphism $f$ is a closed immersion if and only if $\Phi(f)$ is a closed immersion.
	\end{enumerate}
\end{thm}

\begin{proof}
	We start by dealing with étale morphisms.
	Assume first that $f$ is an étale morphism.
	If $X$ is affinoid, it determines an object in the site $Y\et$.
	Let us denote by $U$ this object. 	It follows from \cite[5.1.6.12]{HTT} that the adjunction $f_* \colon \cX_X \rightleftarrows \cX_Y \colon f\inv$ induced by $f$ can be identified with the \'etale morphism $(\cX_Y)_{/U} \rightleftarrows \cX_Y$.
	Since $X$ is an ordinary \kanal space, $U$ is $0$-truncated and therefore it is hypercomplete.
	It follows that we can identify the adjunction
	\[ f_* \colon \cX_X^\wedge \rightleftarrows \cX_Y^\wedge \colon f\inv \]
	with the \'etale morphism $j_* \colon (\cX_Y^\wedge)_{/U} \rightleftarrows \cX_Y^\wedge \colon j\inv$.
	Moreover, since $f$ is étale, we see that $(f\inv \cO_Y)(V) = \cO_Y(V)$.
		In particular, we deduce that $f\inv \cO_Y = \cO_X$.
	In other words, $\Phi(f)$ is \'etale.
	If now $X$ is arbitrary, we choose an étale covering $\{X_i \to X\}$ such that every $X_i$ is affinoid.
	The above argument shows that the induced morphisms $\cX_{X_i} \rightleftarrows \cX_X$ and $\cX_{X_i} \rightleftarrows \cX_Y$ are \'etale.
	It follows that $f_* \colon \cX_X \rightleftarrows \cX_Y \colon f\inv$ is \'etale as well.
	
	Let us now assume that $\Phi(f)$ is \'etale.
	We will prove that $f$ is étale.
	The question being local on $X$ and $Y$, we can assume that they are affinoid, say $X = \Sp B$, $Y = \Sp A$.
	By hypothesis, $f\inv \cO_Y \to \cO_X$ is an equivalence.
	Since the morphism of \inftopoi $f_* \colon \cX_X^\wedge \rightleftarrows \cX_Y^\wedge\colon f\inv$ is \'etale, we see that, for every $U \to X$ étale, one has
	\[ f\inv(\cO_Y)(U) = \cO_Y(U) . \]
	Consider the sheaf $\mathbb L_{\cO_X / f\inv \cO_Y}$ on $\cX_X^\wedge$ defined by
	\[ C \mapsto \mathbb L\an_{\cO_X(C) / f\inv \cO_Y(C)} = \mathbb L\an_{C / f\inv \cO_Y(C)} , \]
	where the symbol $\mathbb L\an$ denotes the analytic cotangent complex (cf.\ \cite[\S 7.2]{Gabber_Almost_2003}).
	Since $f\inv \cO_Y \simeq \cO_X$, this sheaf is identically zero.
	On the other side, if $\eta\inv \colon \cX_X^\wedge \to \cS$ is a geometric point, then
	\[ \eta\inv(\mathbb L\an_{\cO_A / f\inv \cO_B}) \simeq \mathbb L\an_{\eta\inv \cO_A / \eta\inv f\inv \cO_B}. \]
	We can identify $\eta\inv f\inv \cO_B$ with a strictly henselian $B$-algebra $B'$.
	Since the map $B \to B'$ is formally \'etale, we conclude that
	\[ \mathbb L\an_{\eta\inv \cO_A / \eta\inv f\inv \cO_B} \simeq \mathbb L\an_{\eta\inv \cO_A / B}. \]
	This is also the stalk of the sheaf on $\cX_X^\wedge$ defined by
	\[ C \mapsto \mathbb L\an_{C / B}. \]
	Therefore, this sheaf vanishes as well.
	In particular, $\mathbb L\an_{A / B} \simeq 0$, completing the proof.
	\todo{A proof without cotangent complex maybe better.}
	
	We now turn to closed immersions.
	Assume first that $f$ is a closed immersion in $\Ank$.
	\cref{prop:preserve_closed_immersion} and \cref{lem:hypercompletion_closed_immersions} show that the induced geometric morphism $f_* \colon \cX_Y^\wedge \rightleftarrows \cX_X^\wedge \colon f\inv$ is a closed immersion of \inftopoi.
	We are left to show that the morphism $f\inv \cO_X \to \cO_Y$ is an effective epimorphism.
	In virtue of \cref{prop:alg_effective_epi}, it suffices to show that $(f\inv (\cO_X))(\bA^1_k) \to \cO_Y(\bA^1_k)$ is an effective epimorphism, where $\bA^1_k$ denote the \kanal affine line.
	Observe that $(f\inv (\cO_X))(\bA^1_k) \simeq f\inv (\cO_X(\bA^1_k))$.
		Since $(f\inv, f_*)$ is a closed immersion of \inftopoi, \cref{lem:different_closed_immersion} shows that it is sufficient to check that
	\begin{equation} \label{eq:closed_immersion_effective_epi}
		\cO_X(\bA^1_k) \to f_* ( \cO_Y(\bA^1_k))
	\end{equation}
	is an effective epimorphism in $\cX_X^\wedge$.
	This question is local on $\cX_X^\wedge$, so we can assume that $X$ is an affinoid space.
	Observe now that $\cO_X(\bA^1_k)$ is the underlying sheaf of (discrete) spaces associated to the structure sheaf of $X$. In the same way, $f_*(\cO_Y(\bA^1_k))$ is the underlying sheaf of spaces associated to the pushforward of the structure sheaf of $Y$.
	Both are coherent on $X$, and $f_*(\cO_Y(\bA^1_k))$ is the quotient of $\cO_X(\bA^1_k)$ by some coherent sheaf of ideals.
	In particular, the map \eqref{eq:closed_immersion_effective_epi} is an effective epimorphism.
	
	Assume now that $\Phi(f)$ is a closed immersion.
	We want to prove that $f$ is a closed immersion as well.
	The question is local both on the source and on the target, so we can assume that $X$ and $Y$ are affinoid, say $X = \Sp A$ and $Y = \Sp B$.
	In this case, it follows from the proof of \cref{thm:fully_faithfulness} that $f$ corresponds to the morphism 
	\[ A = \cO_X(\bA^1_k)(X) \to  B = \cO_Y(\bA^1_k)(Y) . \]
	Therefore, we only have to show that this morphism is surjective.
	Let $U = \Sp C \to X$ be an étale morphism.
	Then it follows again from the proof of \cref{thm:fully_faithfulness} that
	\begin{align*}
		f_* \cO_Y(\bA^1_k)(U) & = \cO_Y(\bA^1_k)(Y \times_X U) = B {\cotimes}_A C \\
		& = f_* \cO_Y(\bA^1_k)(X) {\cotimes}_{\cO_X(\bA^1_k)(X)} \cO_X(\bA^1_k)(U) .
	\end{align*}
	In particular, $f_* \cO_Y(\bA^1_k)$ is a coherent sheaf of $\cO_X(\bA^1_k)$-modules.
	We can thus apply Tate's acyclicity theorem to conclude that $A \to B$ is surjective, completing the proof.
\end{proof}

\section{Existence of fiber products} \label{sec:fiber_products}

The goal of this section is to prove the existence of fiber products of derived \kanal spaces.

First we will prove the existence of fiber products along a closed immersion (\cref{prop:closed_fiber_products_dAn}).
Then we will prove the existence of products over a point (\cref{lem:products_dAn}).
We will deduce the existence of fiber products in the general case from the two special cases above, plus \cref{lem:closed_devissage}, which shows that any derived \kanal space can locally be embedded into a non-derived smooth \kanal space.

\begin{lem} \label{lem:sheaves_coherent_modules}
	Let $f \colon (\cX, \cO_{\cX}) \to (\cY, \cO_{\cY})$ be a map of derived \kanal spaces such that $(\cX, \pi_0 \cO_{\cX}) \simeq \Phi(X)$ and $(\cY, \pi_0 \cO_{\cY}) \simeq \Phi(Y)$ for two \kanal spaces $X, Y \in \Ank$.
	Assume that $\cF$ is a connective sheaf of $\cO_{\cY}\alg$-modules on $\cY$ and that each $\pi_n \cF$ is a coherent sheaf of $\pi_0 \cO_{\cY}\alg$-modules. Then the tensor product $\cF' \coloneqq f\inv \cF \otimes_{f\inv \cO_{\cY}\alg} \cO_{\cX}\alg$ is connective, and each $\pi_n \cF'$ is a coherent sheaf of $\pi_0(\cO_{\cX}\alg)$-modules.
\end{lem}

\begin{proof}
	The connectivity of $\cF' \coloneqq f\inv \cF \otimes_{f\inv \cO_\cY} \cO_\cX$ follows from the compatibility of the tensor product with the $t$-structure (cf.\ \cite[Proposition 2.1.3(6)]{DAG-VIII}.)
	In order to prove that the homotopy groups $\pi_k \cF'$ are coherent $\pi_0 \cO_{\cX}\alg$-modules, we first remark that the question is local both on $\cX$ and on $\cY$.
	so we can assume that $X$ and $Y$ are affinoid, say $X = \Sp A$ and $Y = \Sp B$.
	We follow closely the proof of \cite[Lemma 12.11]{DAG-IX}.
	Thus, we start by proving that for every integer $m \ge - 1$ there exists a sequence of morphisms
	\[ 0 = \cF(-1) \to \cF(0) \to \cF(1) \to \cdots \to \cF(m) \to \cF \]
	of $\cO_\cY\alg$-modules with the following properties:
	\begin{enumerate}[(i)]
		\item For $0 \le i \le m$, the fiber of $\cF(i-1) \to \cF(i)$ is equivalent to a direct sum of finitely many copies of $\cO_\cY\alg[i]$.
		\item For $0 \le i \le m$, the fiber of $\cF(i) \to \cF$ is $i$-connective.
		\item For $-1 \le i \le m$, the homotopy groups $\pi_j \cF(i)$ are coherent $\pi_0(\cO_\cY\alg)$-modules, which vanish for $j < 0$.
	\end{enumerate}
	We proceed by induction on $m$.
	If $m = -1$, we simply take $\cF(-1) = 0$. The fiber of $\cF(-1) \to \cF$ is then $\cF[1]$, which is $(-1)$-connective because $\cF$ is connective.
	Assume now that we are given a sequence
	\[ 0 = \cF(-1) \to \cF(0) \to \cdots \to \cF(m) \to \cF \]
	satisfying the conditions above.
	Let $\cG$ be the fiber of the map $\cF(m) \to \cF$, so $\cG$ is $m$-connective.
	We have an exact sequence
	\[ \pi_{m+1} \cF(m) \to \pi_{m+1} \cF \to \pi_m \cF' \to \pi_m \cF(m) \to \pi_m \cF , \]
	from which we deduce that $\pi_m \cF'$ is a coherent sheaf of $\pi_0(\cO_\cX\alg)$-modules.
		In particular, there exists a positive integer $l$ and a surjection $B^l \to \cG(Y)$.
	This induces an epimorphism $(\pi_0 \cO_\cY\alg)^l \to \cG[-m]$.
	Composing with the canonical map $(\cO_\cY\alg)^l \to (\pi_0 \cO_\cY\alg)^l$, we obtain a map
	\[ (\cO_\cY\alg)^l[m] \to \cG . \]
	Let $\cF(m+1)$ be the cofiber of the composite map $(\cO_\cY\alg)^l[m] \to \cG \to \cF(m)$.
	Then the property (i) is satisfied by construction and the property (iii) follows from the long exact sequence associated to the cofiber sequence $(\cO_\cY\alg)^l[m] \to \cF(m) \to \cF(m+1)$.
	Let $\cG'$ denote the fiber of the map $\cF(m+1) \to \cF$, so we have a fiber sequence
	\[ (\cO_\cY\alg)^l[m] \to \cG \to \cG' . \]
	Passing to the long exact sequence, we deduce that $\cG'$ is $(m+1)$-connective, proving the property (ii).
	
	Let us now prove that the homotopy groups of $\cF' \coloneqq f\inv \cF \otimes_{f\inv \cO_\cY\alg} \cO_\cX\alg$ are coherent sheaves of $\pi_0(\cO_\cX\alg)$-modules.
	Fix an integer $n \ge 0$.
	Choose a sequence
	\[ 0 \to \cF(-1) \to \cF(0) \to \cdots \to \cF(n+1) \to \cF \]
	satisfying the properties (i), (ii) and (iii) above.
	In particular, the fiber of $\cF(n+1) \to \cF$ is $(n+1)$-connective and therefore the same goes for the map
	\[ f\inv \cF(n+1) \otimes_{f\inv \cO_\cY\alg} \cO_\cX\alg \to f\inv \cF \otimes_{f\inv \cO_\cY\alg} \cO_\cX\alg . \]
	So we obtain an isomorphism
	\[ \pi_n \left( f\inv \cF(n+1) \otimes_{f\inv \cO_\cY\alg} \cO_\cX\alg \right) \to \pi_n \left( f\inv \cF \otimes_{f\inv \cO_\cY\alg} \cO_\cX\alg \right) . \]
	We can therefore replace $\cF$ by $\cF(n+1)$.
	We will now prove that for $-1 \le i \le n+1$, $\pi_n \left( f\inv \cF(i) \otimes_{f\inv \cO_\cY\alg} \cO_\cX\alg \right)$ is a coherent sheaf of $\pi_0(\cO_\cX\alg)$-modules.
	We proceed by induction on $i$.
	The case $i = -1$ is trivial.
	To deal with the inductive step, we note that the property (i) implies the existence of a fiber sequence
	\[ (\cO_\cY\alg)^l[i] \to \cF(i) \to \cF(i+1) . \]
	We therefore obtain a long exact sequence
	\begin{multline*}
	\cdots \to ( \pi_{n-i} \cO_\cX\alg )^l \to \pi_n ( f\inv \cF(i) \otimes_{f\inv \cO_\cY\alg} \cO_\cX\alg ) \to\\
	 \pi_n( f\inv \cF(i+1) \otimes_{f\inv \cO_\cY\alg} \cO_\cX\alg ) \to ( \pi_{n - i - 1} \cO_\cX\alg )^l \to \cdots
	\end{multline*}
	We conclude that $\pi_n \left( f\inv \cF(i + 1) \otimes_{f\inv \cO_\cY\alg} \cO_\cX\alg \right)$ is a coherent sheaf of $\pi_0(\cO_\cX\alg)$-modules.
	\end{proof}

\begin{prop} \label{prop:closed_fiber_products_dAn}
	Assume we are given maps of derived \kanal spaces $f \colon (\cY, \cO_{\cY}) \allowbreak{}\to (\cX, \cO_{\cX})$ and $(\cX', \cO_{\cX'}) \to (\cY, \cO_{\cY})$.
	Assume moreover that $f$ is a closed immersion.
	Then we have the following statements:
	\begin{enumerate}[(i)]
		\item \label{item:dAn_fiber_products_Top} There exists a pullback diagram $\sigma$:
			\[ \begin{tikzcd}
				(\cY', \cO_{\cY'}) \arrow{r}{f'} \arrow{d} & (\cX', \cO_{\cX'}) \arrow{d} \\
				(\cY, \cO_{\cY}) \arrow{r}{f} & (\cX, \cO_{\cX})
			\end{tikzcd} \]
			in the $\infty$-category $\RHTop(\cTank)$.
		\item \label{item:dAn_fiber_products_topoi} The image of $\sigma$ in $\RHTop$ is a pullback diagram of hypercomplete \inftopoi.
		\item \label{item:dAn_closed_immersion} The map $f'$ is a closed immersion.
		\item \label{item:dAn_fiber_product_dAn} The structured \inftopos $(\cY', \cO_{\cY'})$ is a derived \kanal space.
		\item \label{item:truncation_pullback} Assume that $(\cY, \pi_0 \cO_{\cY}) = \Phi(Y)$, $(\cX, \pi_0 \cO_{\cX}) = \Phi(X)$ and $(\cX', \pi_0 \cO_{\cX'}) = \Phi(X')$.	Then $(\cY', \pi_0 \cO_{\cY'})$ can be identified with $\Phi(Y \times_X X')$.
	\end{enumerate}
\end{prop}

\begin{proof}	The statements (\ref{item:dAn_fiber_products_Top}), (\ref{item:dAn_fiber_products_topoi}) and (\ref{item:dAn_closed_immersion}) follow from \cref{prop:closed_fiber_products_Top}.
	We now prove (\ref{item:truncation_pullback}).
	Observe that the map $f$ induces a closed immersion $(\cY, \pi_0 \cO_{\cY}) \to (\cX, \pi_0 \cO_{\cX})$.
	So by \cref{thm:Phi_classes_of_morphisms}, it corresponds to a closed immersion $\varphi \colon Y \to X$ of \kanal spaces.
	On the other side, the map $\Phi(X') \to \Phi(X)$ corresponds to a map $X' \to X$ by \cref{thm:fully_faithfulness}.
	Let $Y' \coloneqq Y \times_X X'$ be the fiber product computed in $\Ank$.
	Then \cref{prop:closed_immersion_pullback_of_topoi} allows us to identify $\cX_{Y'}\coloneqq\Sh(Y'\et)^\wedge$ with $\cY'$.
	It follows from the universal property of the fiber product that there exists a map in $\RHTop(\cTank)$
	\[ (\cY', \cO_{Y'}) \to (\cY', \cO_{\cY'}) \]
	Moreover, it follows from \cref{prop:closed_fiber_products_Top}(iii) that we have an identification
	\[ \cO_{\cY'}\alg \simeq f^{\prime -1} \cO_{X'}\alg \otimes_{f^{\prime -1} g\inv \cO_{X}} g^{\prime -1} \cO_{Y} . \]
	Using \cite[7.2.1.22]{Lurie_Higher_algebra}, we obtain an equivalence
	\[ \pi_0(\cO_{\cY'}\alg) \simeq \Tor_0^{ f^{\prime -1} g\inv( \pi_0 \cO_X\alg )}(f^{\prime -1} \pi_0( \cO_{X'}\alg ), g^{\prime -1}( \pi_0 \cO_Y\alg )) . \]
	As $\pi_0( \cO_\cX) \to f_* \pi_0( \cO_\cY)$ is surjective, we see that the same formula can be used to describe $\cO_{Y'}$.
	Hence $\pi_0(\cO_{\cY'}) \simeq \cO_{Y'}$.
	This proves (\ref{item:truncation_pullback}).
	
	We are left to prove the statement (\ref{item:dAn_fiber_product_dAn}).
	The assertion is local on $\cY'$, so we can assume that $(\cX, \pi_0 \cO_{\cX}) = \Phi(X)$, $(\cY, \pi_0 \cO_{\cY}) = \Phi(Y)$ and $(\cX', \pi_0 \cO_{\cY}) = \Phi(X')$ for \kanal spaces $X, X'$ and $Y$.
	It follows from (\ref{item:truncation_pullback}) that $(\cY', \pi_0 \cO_{\cY'}\alg)$ is a \kanal space.
	Moreover, since $f$ is a closed immersion, we see that for each $n \ge 0$ the pushforward $f_* \pi_n \cO_{Y}\alg$ is a coherent sheaf of $\pi_0 \cO_X\alg$-modules on $X$.
	Using \cref{lem:sheaves_coherent_modules} and \cref{prop:closed_fiber_products_Top}, we conclude that for each $n \ge 0$, the pushforward $f'_* \pi_n \cO_{\cY'}\alg$ is a  coherent sheaf of $\pi_0 \cO_{\cX'}\alg$-modules.
	Then each $\pi_n \cO_{\cY'}\alg$ is a coherent sheaf of $\pi_0 \cO_{\cY'}\alg$-modules.
	This completes the proof.
\end{proof}

\begin{lem} \label{lem:closed_devissage}
	Let $(\cX, \cO_{\cX})$ be a derived \kanal space and let $\mathbf 1_\cX$ be the final object of $\cX$.
	Then there exists an effective epimorphism $\coprod U_i \to \mathbf 1_\cX$ and a collection of closed immersions $(\cX_{/U_i}, \cO_{\cX}|_{U_i}) \to \HSpec^{\cTank}(V_i)$, where $V_i$ is a smooth \kanal space.
\end{lem}

\begin{proof}
	We can assume without loss of generality that $(\cX, \pi_0 \cO_{\cX}) \simeq \Phi(X)$ for a $k$-affinoid space $X$.
	So we have a closed immersion into a \kanal polydisc $X \hookrightarrow \bD^n_k$. Composing with the affinoid domain embedding $\bD^n_k \hookrightarrow \bA^n_k$, we obtain an embedding $X \hookrightarrow \bA^n_k$.
	This embedding is given by $n$ global sections $f_1, \ldots, f_n \in \pi_0(\cO_{\cX}\alg)(X)$.
	Let $\{u_i \colon U_i \to X\}_{i \in I}$ be an étale covering such that each restriction $f_j \circ u_i$ is represented by some $\widetilde{f}_{ij} \in \cO_{\cX}(\bA^1_k)(U_i)$.
	Combining \cref{lem:universal_property_HSpec} and \cite[Theorem 2.2.12]{DAG-V}, we deduce that these global sections determine a morphism of derived \kanal spaces
	\[ \varphi_i \colon (\cX_{/U_i}, \cO_{\cX}|_{U_i}) \to \HSpec^{\cTank}(\bA^n_k). \]
		Choose a factorization of $U_i \to X \to \bD^n_k$ as $U_i \xrightarrow{p} V_i \xrightarrow{g} \bD^n_k$, where $p$ is a closed immersion and $g$ is étale.
		The composite map $V_i \to \bD^n_k \to \bA^n_k$ is étale and therefore by \cref{thm:Phi_classes_of_morphisms}(i) the induced morphism of derived \kanal spaces $\HSpec^{\cTank}(V_i) \to \HSpec^{\cTank}(\bA^n_k)$ is \'etale.
	Then \cite[Remark 2.3.4]{DAG-V} shows that the map $\varphi_i$ factors through $\HSpec^{\cTank}(V_i)$ if and only if the underlying morphism of \inftopoi factors through $\cX_{V_i}$.
	The latter holds by construction.
	Moreover, the truncation of $\psi_i \colon (\cX_{/U_i}, \cO_{\cX}|_{U_i}) \to \HSpec^{\cTank}(V_i)$ corresponds to the map $U_i \to V_i$, which is a closed immersion. It follows that $\psi_i$ is a closed immersion as well, completing the proof.
	\end{proof}

\begin{lem} \label{lem:products_dAn}
	Let $(\cX, \cO_{\cX})$ and $(\cY, \cO_{\cY})$ be derived \kanal spaces.
	We have the following statements:
	\begin{enumerate}[(i)]
		\item There exists a product $(\cZ, \cO_{\cZ}) \simeq (\cX, \cO_{\cX}) \times (\cY, \cO_{\cY})$ in $\RTop(\cTank)$.
		\item The structured \inftopos $(\cZ, \cO_{\cZ})$ is a derived \kanal space.
		\item Assume that $(\cX, \pi_0 \cO_{\cX}) \simeq \Phi(X)$ and $(\cY, \pi_0 \cO_{\cY}) \simeq \Phi(Y)$. Then $(\cZ, \pi_0\cO_{\cZ})$ is equivalent to $\Phi(X \times Y)$.
		\item Assume that $(\cX, \pi_0 \cO_{\cX}) \simeq \Phi(X)$ where $X$ is a separated \kanal space.
		Then the diagonal map $\delta \colon (\cX, \cO_{\cX}) \to (\cX, \cO_{\cX}) \times (\cX, \cO_{\cX})$ is a closed immersion.
	\end{enumerate}
\end{lem}

\begin{proof}
	The statements (i) and (ii) are local on $(\cX, \cO_{\cX})$ and $(\cY, \cO_{\cY})$, so we can assume in virtue of \cref{lem:closed_devissage} that there exists closed immersions $(\cX, \cO_{\cX}) \to \HSpec^{\cTank}(V)$ and $(\cY, \cO_{\cY}) \to \HSpec^{\cTank}(W)$, where $V$ and $W$ are smooth \kanal spaces.
	\cref{prop:closed_fiber_products_dAn} allows us to reduce to the case $(\cX, \cO_{\cX}) \simeq \HSpec^{\cTank}(V)$ and $(\cY, \cO_{\cY}) \simeq \HSpec^{\cTank}(W)$.
	In this case, we have
	\[(\cZ, \cO_{\cZ}) \simeq \HSpec^{\cTank}(V \times W).\]
		The statement (iii) follows from the construction of $(\cZ, \cO_{\cZ})$ we described and \cref{prop:closed_fiber_products_dAn}(\ref{item:truncation_pullback}).
	We are left to prove the statement (iv).
	The statement (ii) shows that the induced map
	\[ \pi_0(\delta) \colon (\cX, \pi_0 \cO_{\cX}\alg) \to (\cX, \pi_0 \cO_{\cX}\alg) \times (\cX, \pi_0 \cO_{\cX}\alg) \]
	corresponds to $\Phi(\Delta) \colon \Phi(X) \to \Phi(X \times X)$.
	Since $X$ is separated, $\Delta \colon X \to X \times X$ is a closed immersion and therefore \cref{thm:Phi_classes_of_morphisms} implies that $\Phi(\Delta)$ is a closed immersion.
	Now, the assertion follows from \cref{prop:alg_effective_epi}.
\end{proof}

Now we can deduce the main result of this section:

\begin{thm} \label{thm:fiber_products}
	The \infcat $\dAnk$ admits fiber products.
\end{thm}

\begin{proof}
	Let $(\cY,\cO_\cY)\to(\cX,\cO_\cX)\leftarrow(\cX',\cO_{\cX'})$ be maps of derived \kanal spaces.
	We would like to construct the fiber product.
	Working locally on $\cX$, we can assume that $(\cX,\pi_0\cO_\cX\alg)\simeq\upsilon(X)$ for a separated \kanal space $X$.
	Using \cref{lem:products_dAn}(i), we deduce the existence of two products $(\cZ,\cO_\cZ)\coloneqq(\cX',\cO_\cX')\times(\cY,\cO_\cY)$ and $(\cX,\cO_\cX)\times(\cX,\cO_\cX)$ in $\Top(\cTank)$.
	By \cref{lem:products_dAn}(iv), the diagonal map $\delta \colon (\cX, \cO_{\cX}) \to (\cX, \cO_{\cX}) \times (\cX, \cO_{\cX})$ is a closed immersion.
	We now apply \cref{prop:closed_fiber_products_dAn} to produce a fiber product
	\[ \begin{tikzcd}
	(\cY',\cO_{\cY'}) \arrow{r} \arrow{d} & (\cZ,\cO_\cZ) \arrow{d} \\
	(\cX,\cO_\cX) \arrow{r} &  (\cX, \cO_{\cX}) \times (\cX, \cO_{\cX}).
	\end{tikzcd} \]
	Note that $(\cY',\cO_{\cY'})$ is the fiber product of $(\cY,\cO_\cY)\to(\cX,\cO_\cX)\leftarrow(\cX',\cO_{\cX'})$ , completing the proof.
\end{proof}

\section{Comparison between derived spaces and non-derived stacks} \label{sec:essential_image}

In this section, we will characterize the essential image of the embedding $\Phi\colon\An_k\to\dAnk$ constructed in \cref{sec:fullyfaithfulness}.
Moreover, we will compare derived \kanal spaces with higher \kanal stacks in the sense of \cite{Porta_Yu_Higher_analytic_stacks_2014}.

\subsection{Construction of the comparison functor}

On the \infcat $\dAnk$ of derived \kanal spaces, we define the étale topology $\tauet$ to be the Grothendieck topology generated by collections of étale morphisms $\{U_i\to U\}$ such that $\coprod U_i\to U$ is an effective epimorphism (cf.\ \cref{def:closed_immersion_and_etale}).

\begin{rem}
	The restriction of $\tauet$ to the full subcategory $\Ank$ of $\dAnk$ coincides with the étale topology $\tauet$.
\end{rem}

\begin{lem}
	Every representable presheaf on $\dAnk$ is a hypercomplete sheaf for the topology $\tauet$.
\end{lem}

\begin{proof}
	Let $X \coloneqq (\cX, \cO_\cX)$ be a derived \kanal space.
	The universal property of \'etale morphisms (cf.\ \cite[Remark 2.3.4]{DAG-V}) shows that a $\tauet$-hypercovering of $X$ can be identified with a hypercovering $U^\bullet$ of $\mathbf 1_\cX$ in the \inftopos $\cX$.
	Given such a hypercovering, the associated $\tauet$-hypercovering $X^\bullet$ of $X$ is described by $X^n \coloneqq (\cX_{/U^n}, \cO_\cX |_{U^n})$.
	Therefore, we have to prove that
	\[ \colim_\Delta ( \cX_{/U^\bullet}, \cO_\cX |_{U^\bullet} ) \simeq (\cX, \cO_\cX) \]
	in the $\infty$-category $\dAnk$.
	Using the statement (3') in the proof of \cite[Proposition 2.3.5]{DAG-V}, we see that it is enough to prove that $\cX \simeq \colim \cX_{/U^\bullet}$ in $\RTop$.
	Since $\cX$ is hypercomplete, this follows from the descent theory of \inftopoi (cf.\ \cite[6.1.3.9]{HTT}) and from the fact that $|U^\bullet| \simeq \mathbf 1_\cX$ (cf.\ \cite[6.5.3.12]{HTT}).
	\end{proof}

\begin{defin} \label{def:derived_affinoid}
	A \emph{derived $k$-affinoid space} is a derived \kanal space $(\cX, \cO_\cX)$ such that $(\cX, \pi_0(\cO_\cX)) \simeq \Phi(X)$ for some $k$-affinoid space $X$.
	We denote by $\dAfdk$ the full subcategory of $\dAnk$ spanned by derived $k$-affinoid spaces.
\end{defin}

The Grothendieck topology $\tauet$ on $\dAnk$ induces by restriction a Grothendieck topology on $\dAfdk$ which we denote again by $\tauet$.
We define the functor $\tphi$ as the composition
\[ \begin{tikzcd}
	\dAnk \arrow{r} & \Fun(\dAnk^{\mathrm{op}}, \cS) \arrow{r} & \Fun( ( \dAfdk )^{\mathrm{op}}, \cS ) ,
\end{tikzcd} \]
where the first functor is the Yoneda embedding and the second one is the restriction along $\dAfdk \subset \dAnk$.
Since the Grothendieck topology $\tauet$ on $\dAnk$ is subcanonical, the functor $\tphi$ factors through $\Sh(\dAfdk, \tauet)$.
We denote by
\[ \phi \colon \dAnk \to \Sh( \dAfdk, \tauet) \]
the induced functor.
Our first goal is to show that $\phi$ is fully faithful.

\begin{lem} \label{lem:affine_site_big_site}
	Let $X = (\cX, \cO_\cX)$ be a derived \kanal space and let $p \colon U \to \mathbf 1_\cX$ be an effective epimorphism.
	Let $U^\bullet$ be the \v{C}ech nerve of $p$ and put $X^n \coloneqq (\cX_{/U^n}, \cO_\cX|_{U^n})$.
	Then in $\Sh(\dAfdk, \tauet)$ we have
	\[ \phi(X) \simeq \colim_{\Delta} \phi(X^\bullet) . \]
\end{lem}

\begin{proof}
	Let $j\colon\dAfdk\hookrightarrow\dAnk$ denote the inclusion functor.
	It is continuous and cocontinuous in the sense of \cite[\S 2.4]{Porta_Yu_Higher_analytic_stacks_2014}.
	It induces a pair of adjoint functors
	\[ j_s\colon \Sh(\dAnk,\tauet)\leftrightarrows\Sh(\dAfdk,\tauet)\colon j^s.\]
	Since the Grothendieck topology $\tauet$ on $\dAnk$ is subcanonical, we can factor $\phi$ as
	\[\dAnk \xrightarrow{\ \psi\ } \Sh( \dAnk, \tauet ) \xrightarrow{\ j_s\ } \Sh( \dAfdk, \tauet ).\]
	Moreover, we have
	\[ \psi(X) \simeq \colim_{\Delta} \psi(X^\bullet) . \]
	Since the functor $j_s$ is a left adjoint, it commutes with colimits, completing the proof.
\end{proof}

\begin{lem} \label{lem:affine_hypercover}
	Let $X = (\cX, \cO_\cX)$ be a derived \kanal space.
	Then there exists a hypercovering $X^\bullet$ of $X$ in $\dAnk$ such that each $X^n$ is a disjoint union of derived $k$-affinoid spaces.
\end{lem}
\begin{proof}
	It follows directly from Definitions \ref{def:derived_space} and \ref{def:derived_affinoid}.
\end{proof}

\begin{prop} \label{prop:phi_fully_faithful}
	The functor $\phi \colon \dAnk \to \Sh( \dAfdk, \tauet)$ is fully faithful.
\end{prop}

\begin{proof}
	Let $X, Y \in \dAnk$ and consider the natural map
	\[ \psi_{X,Y} \colon \Map_{\dAnk}(X,Y) \to \Map_{\Sh( \dAfdk, \tauet )}(\phi(X), \phi(Y)) . \]
	Keeping $Y$ fixed, let $\cC$ be the full subcategory of $\dAnk$ spanned by those $X$ such that $\psi_{X,Y}$ is an equivalence.
	Since $\cC$ is stable under colimits, combining Lemmas \ref{lem:affine_site_big_site} and \ref{lem:affine_hypercover}, we are reduced to the case where $X$ belongs to $\dAfdk$.
	In this case, the statement follows entirely from the Yoneda lemma.
\end{proof}

Our second goal is to identify the essential image of the functor $\phi$.
For this, we need to introduce some notations.

\begin{defin}
Let $\bP\et$ denote the class of étale morphisms in $\dAnk$.
The triple $(\dAfdk, \tauet, \bP\et)$ constitutes a geometric context in the sense of \cite{Porta_Yu_Higher_analytic_stacks_2014}.
We call the associated geometric stacks \emph{derived \kanal Deligne-Mumford stacks}.
We denote by $\mathrm{DM}$ the \infcat of derived \kanal Deligne-Mumford stacks.
\end{defin}

\begin{defin}
	Let $F \in \mathrm{DM}$.
	We say that $F$ is \emph{$n$-truncated} if $F(X)$ is $n$-truncated for every $X = (\cX, \cO_\cX) \in \dAfdk$ such that $\cO_\cX$ is discrete.
	We denote by $\mathrm{DM}_n$ the full subcategory of $\mathrm{DM}$ spanned by $n$-truncated \kanal Deligne-Mumford stacks.
\end{defin}

We denote by $\dAnk^{\le n}$ the full subcategory of $\dAnk$ spanned by those derived \kanal spaces $(\cX, \cO_\cX)$ such that $\cX$ is $n$-localic (cf.\ \cite[6.4.5.8]{HTT}).

With these notations we can now state our main comparison theorem, which is an analogue of \cite[Theorem 3.7]{Porta_DCAGI} and \cite[Theorem 1.7]{Porta_Comparison_2015}.

\begin{thm} \label{thm:functor_of_points_vs_dAnk}
	For every integer $n \ge 1$, the functor $\phi \colon \dAnk \to \Sh(\dAfdk, \tauet)$ restricts to an equivalence of \infcats $\dAnk^{\le n} \simeq \mathrm{DM}_n$.
\end{thm}

The proof will occupy the rest of this section.
Before plunging ourselves into the details, let us deduce from this theorem an important application.

Let $(\An_k, \tauet, \mathbf P\et)$ be the geometric context consisting of the category of $k$-analytic spaces, the étale topology and the class of étale morphisms.
The associated geometric stacks are called \emph{higher \kanal Deligne-Mumford stacks}.
They are in particular higher \kanal stacks considered in \cite{Porta_Yu_Higher_analytic_stacks_2014}.
So all the results in loc.\ cit.\ apply.

\begin{cor} \label{cor:underived_higher_kanal_stacks}
	Let $\Geom(\An_k,\tauet,\bP\et)$ denote the \infcat of higher \kanal Deligne-Mumford stacks.
	There is a fully faithful embedding $\mathrm{Geom}(\An_k, \tauet, \mathbf P\et) \to \dAnk$ whose essential image is spanned by those derived \kanal spaces $(\cX, \cO_\cX)$ such that $\cX$ is $n$-localic for some $n$ and $\cO_\cX$ is discrete.
\end{cor}

\begin{proof}
	Let $(\Afd_k, \tauet, \mathbf P\et)$ be the geometric context consisting of the category of $k$-affinoid spaces, the étale topology and the class of étale morphisms.
	Let $\Geom(\Afd_k,\tauet,\bP\et)$ denote the \infcat of geometric stacks associated to this geometric context.
	By \cite[\S 2.5]{Porta_Yu_Higher_analytic_stacks_2014}, we have an equivalence
	\[\Geom(\An_k,\tauet,\bP\et)\simeq\Geom(\Afd_k,\tauet,\bP\et).\]

	It follows from \cref{thm:fully_faithfulness} that the natural inclusion $j \colon \Afd_k \to \dAfdk$ is fully faithful.
	So the induced functor
	\[ j_s \colon \Sh(\Afd_k, \tauet) \to \Sh(\dAfdk, \tauet) \]
	is fully faithful as well.
	We know moreover that $j_s$ preserves geometric stacks.
	Therefore $j_s$ factors through the full subcategory $\mathrm{DM} = \bigcup \mathrm{DM}_n$.
	Applying \cref{thm:functor_of_points_vs_dAnk}, we obtain the desired fully faithful functor $\mathrm{Geom}(\Afd_k, \tauet, \mathbf P\et) \to \dAnk$.

	Now it suffices to observe that if a geometric stack $X \in \dAnk^{\le n}$ is discrete, then $\phi(X)$ lies in the essential image of $j_s$.
	Indeed, if $X$ is discrete, then
	\[ \Map_{\dAnk}( Y, X ) = \Map_{\dAnk}(t_0(Y), X). \]
	Therefore $\phi(X)$ coincides with the left Kan extension of its restriction along $j$, completing the proof.
\end{proof}

\subsection{The case of algebraic spaces}

Given a derived \kanal space $X$, we denote by $\dAfd_X$ the overcategory $(\dAfdk)_{/X}$.
The Grothendieck topology $\tauet$ on $\dAnk$ induces a Grothendieck topology on $\dAfd_X$, which we still denote by $\tauet$.
Let $X_\biget$ denote the Grothendieck site $(\dAfd_X, \tauet)$.

Let $(\dAfd_X)\et$ be the full subcategory of the overcategory $\dAfd_X$ spanned by \'etale morphisms $Y \to X$.
The \'etale topology $\tauet$ on $\dAfd_X$ restricts to a Grothendieck topology on $(\dAfd_X)\et$, which we still denote by $\tauet$.
Let $X\et$ denote the Grothendieck site $((\dAfd_X)\et, \tauet)$.

\begin{rem}
	Let $X$ be an ordinary \kanal space. Let $f \colon (\cY, \cO_\cY) \to \Phi(X)$ be an \'etale morphism in $\dAnk$. Since the morphism $f\inv \cO_X \to \cO_\cY$ is an equivalence, we see that $\cO_\cY$ is discrete. In particular, if $(\cY, \cO_\cY)$ is a derived $k$-affinoid space, then it belongs to the essential image of $\Phi$.
	This shows that there is a canonical equivalence $X\et \simeq \Phi(X)\et$.
\end{rem}

We have continuous functors between the sites
\[ \begin{tikzcd}
	(X\et, \tauet) \arrow{r}{u} & (X_\biget, \tauet) \arrow{r}{v} & (\dAfdk, \tauet) .
\end{tikzcd} \]
By \cite[\S 2.4]{Porta_Yu_Higher_analytic_stacks_2014}, they induce adjunctions on the \infcats of sheaves
\begin{gather*}
	u_s \colon \Sh(X\et, \tauet) \rightleftarrows \Sh(X_\biget, \tauet) \colon u^s, \\
	v_s \colon \Sh(X_\biget, \tauet) \rightleftarrows \Sh(\dAfdk, \tauet) \colon v^s .
\end{gather*}
Moreover, since $u$ is left exact, $(u_s, u^s)$ is a geometric morphism of \inftopoi. In particular, $u_s$ takes $n$-truncated objects to $n$-truncated objects.
On the other side, we can identify the adjunction $(v_s, v^s)$ with the canonical adjunction
\[ \Sh(\dAfdk, \tauet)_{/\phi(X)} \rightleftarrows \Sh(\dAfdk, \tauet), \]
where the right arrow is the forgetful functor.

\begin{defin}
	Let $X \in \dAfdk$, $Y \in \Sh(\dAfdk, \tauet)$ and $\alpha \colon Y \to \phi(X)$ a natural transformation.
	We say that \emph{$\alpha$ exhibits $Y$ as an \'etale derived algebraic space over $X$} if there exists a $0$-truncated sheaf $F \in \Sh(X\et, \tauet)$ and an equivalence $Y \simeq v_s(u_s(F))$ in $\Sh(\dAfdk, \tauet)_{/\phi(X)}$.
\end{defin}

\begin{prop} \label{prop:etale_algebraic_spaces}
	Let $X \in \dAfdk$, $Y \in \Sh(\dAfdk, \tauet)$ and $\alpha \colon Y \to \phi(X)$ a natural transformation.
	The following statements are equivalent:
	\begin{enumerate}[(i)]
		\item The natural transformation $\alpha$ exhibits $Y$ as an \'etale derived algebraic space over $X$.
		\item There exists a discrete object $U \in \cX$ such that $\phi(j)$ is equivalent to $\alpha$, where $j \colon (\cX_{/U}, \cO_\cX |_U) \to (\cX, \cO_\cX)$ is the induced \'etale morphism.
		\item The natural transformation $\alpha$ is $0$-truncated and $0$-representable by étale morphisms.
	\end{enumerate}
\end{prop}

\begin{proof}
	We first prove the equivalence between (i) and (ii).
	If $\alpha$ exhibits $Y$ as an \'etale derived algebraic space over $X$, we can find a $0$-truncated sheaf $U \in \Sh(X\et, \tauet)$ and an equivalence $Y \simeq v_s(u_s(U))$ in $\Sh(\dAfdk, \tauet)_{/\phi(X)}$.
	Consider $X_U \coloneqq (\cX_{/U}, \cO_\cX |_U)$ and let $j \colon X_U \to X$ be the induced \'etale map.
	We want to prove that $\phi(j)$ is equivalent to $\alpha$.
	For any $Z = (\cZ, \cO_\cZ) \in \dAfdk$ and any map $f \colon \phi(Z) \to \phi(X)$, we have a fiber sequence
	\[ \begin{tikzcd}
		\Map_{\Sh( \dAfdk, \tauet )_{/\phi(X)}}(\phi(Z)_f, u_s(U)) \arrow{r} \arrow{d} & \Map_{\Sh(\dAfdk, \tauet)}(\phi(Z), v_s(u_s(U))) \arrow{d} \\
		\{f\} \arrow{r} & \Map_{\Sh(\dAfdk, \tauet)}(\phi(Z), \phi(X)),
	\end{tikzcd} \]
	where $\phi(Z)_f$ denotes the object $f \colon \phi(Z) \to \phi(X)$ in $\Sh(\dAfdk, \tauet)_{/\phi(X)}$.
	Since $\phi$ is fully faithful by \cref{prop:phi_fully_faithful}, we can view $\phi(Z)_f$ as a representable object in $\Sh( X_\biget, \tauet) \simeq \Sh(\dAfdk, \tauet)_{/\phi(X)}$.
	Therefore, the Yoneda lemma combined with \cite[4.3.2.15]{HTT} implies that
	\[ \Map_{\Sh( \dAfdk, \tauet )_{/\phi(X)}}(\phi(Z)_f, u_s(U)) \simeq \Gamma(\cZ, f\inv(U)) . \]
	In particular, taking $Z$ to be an atlas for $X_U$ and choosing $f$ to be $j$, we obtain a canonical map $\phi(X_U) \to v_s(u_s(U))$.
	For any $Z \in \dAfdk$, we obtain in this way a commutative square
	\[ \begin{tikzcd}
		\Map_{\Sh(\dAfdk, \tauet)}(\phi(Z), \phi(X_U)) \arrow{r} \arrow{d} & \Map_{\Sh(\dAfdk, \tauet)}(\phi(Z), \phi(X)) \arrow[-, double equal sign distance]{d} \\
		\Map_{\Sh(\dAfdk, \tauet)}(\phi(Z), v_s(u_s(U))) \arrow{r} & \Map_{\Sh(\dAfdk, \tauet)}(\phi(Z), \phi(X)) .
	\end{tikzcd} \]
	For any morphism $f \colon \phi(Z) \to \phi(X)$, we can combine the fully faithfulness of $\phi$ and \cite[Remark 2.3.4]{DAG-V} to identify the fiber of the top horizontal morphism with $\Gamma(\cZ, f\inv(U))$.
	The same holds for the lower horizontal morphism in virtue of the above discussion.
	Therefore, there is a canonical identification of $\phi(X_U)$ with $Y = v_s(u_s(U))$ in $\Sh(\dAfdk, \tauet)$, and a canonical identification of $\phi(j)$ with $\alpha$.
	On the other side, if (ii) is satisfied, then $U$ defines an \'etale derived algebraic space $v_s(u_s(U))$ over $X$, which can be identified with $Y$ using the same argument as above.
	
	Let us now prove the equivalence between (i) and (iii).
	First, assume that (iii) is satisfied.
	In this case, we can define a sheaf $U \colon X\et \to \cS$ by sending an \'etale map $f \colon Z \to X$ to the fiber product
	\[ \begin{tikzcd}
		U(Z) \arrow{r} \arrow{d} & Y(Z) \arrow{d}{\alpha_Z} \\
		\{*\} \arrow{r}{f} & \phi(X)(Z).
	\end{tikzcd} \]
	Since $\alpha$ is $0$-truncated, we see that $U$ takes values in $\rSet$.
	Since both $\phi(X)$ and $Y$ are sheaves, the same goes for $U$.
	It follows that $U$ defines a $0$-truncated object in $\Sh(X\et, \tauet)$.
	Since $\alpha$ is $0$-representable by \'etale maps we obtain a canonical map $Y \to v_s(u_s(Y))$, and \cite[Remark 2.3.4]{DAG-V} shows that this map is an equivalence.
	
	Finally, let us prove that (i) implies (iii).
	Choose a $0$-truncated sheaf $U \in \Sh(X\et, \tauet)$ such that $Y \simeq v_s(u_s(U))$.
	We already remarked that in this case $\alpha$ is $0$-truncated.
	Choose $V_i \in X\et$ and sections $\eta_i \in U(V_i)$ which generate $U$, we obtain an effective epimorphism
	\[ \coprod \phi(V_i) = \coprod v_s(u_s( V_i )) \to v_s(u_s(U)) \]
	in $\Sh(\dAfdk, \tauet)$.
	Suppose there is a $(-1)$-truncated morphism $v_s(u_s(U)) \to \phi(Z)$ for some $Z \in X\et$.
	In this case, we see that
	\[ \phi(V_i) \times_{v_s(u_s(U))} \phi(V_j) \simeq \phi(V_i) \times_{\phi(Z)} \phi(V_j) \]
	and therefore the maps $\phi(V_i) \to v_s(u_s(U)) \simeq Y$ is $(-1)$-representable by \'etale maps.
	In the general case, the fiber product $Y_{i,j} \coloneqq \phi(V_i) \times_{v_s(u_s(U))} \phi(V_j)$ is again a derived algebraic space \'etale over $X$.
	We claim that the canonical map $Y_{i,j} \to \phi(V_i \times_X V_j)$ is $(-1)$-truncated.
	Indeed, we have a pullback diagram
	\[ \begin{tikzcd}
		Y_{i,j} \arrow{r} \arrow{d} & v_s(u_s(U)) \arrow{d} \\
		\phi(V_i) \times_{\phi(X)} \phi(V_j) \arrow{r} & v_s(u_s(U)) \times_{\phi(X)} v_s(u_s(U)).
	\end{tikzcd} \]
	Since the map $\alpha \colon Y \to \phi(X)$ is $0$-truncated, we conclude the proof of the claim by \cite[5.5.6.15]{HTT}.
	At this point, we deduce that $Y_{i,j} \to X$ is $(-1)$-representable by \'etale maps, and therefore that each $\phi(V_i) \to Y$ is $0$-representable by \'etale maps.
\end{proof}

\subsection{Proof of \cref{thm:functor_of_points_vs_dAnk}}

We begin with the following analogue of \cite[Lemma 2.7]{Porta_Comparison_2015}

\begin{lem} \label{lem:decreasing_truncated_level}
	Let $n \ge 0$ be an integer.
	Fix $X = (\cX, \cO_\cX) \in \dAnk^{\le n+1}$ and let $V \in \cX$ be an object such that $X_V \coloneqq (\cX_{/V}, \cO_\cX |_V)$ is a derived $k$-affinoid space.
	Then $V$ is $n$-truncated.
\end{lem}

\begin{proof}
	We have to prove that for every object $U \in \cX$, the space
	\[ \Map_\cX(U, V) \simeq \Map_{\cX_{/U}}(U, U \times V) \]
	is $n$-truncated.
	This property is local on $U$, so we can restrict ourselves to the situation where $X_U \coloneqq (\cX_{/U}, \cO_\cX |_U)$ is a derived $k$-affinoid space.
	Using \cite[Remark 2.3.4]{DAG-V}, we see that this space fits into a fiber sequence
	\[ \Map_{\cX}(U, V) \to \Map_{\dAnk}( X_U, X_V ) \to \Map_{\dAnk}( X_U, X ) . \]
	
	Since a derived \kanal space $(\cY, \cO_\cY)$ belongs to $\dAfdk$ if and only if its truncation $(\cY, \pi_0(\cO_\cY))$ does, we can replace $X$ with its truncation.
	Let us denote by $F_X \colon \Afd_k \to \cS$ the functor of points associated to $X$ and by $F_V \colon \Afd_k \to \cS$ the functor of points associated to $(\cX_{/V}, \cO_\cX |_V)$.
	The arguments above show that it is enough to prove that for every ordinary $k$-affinoid space $Z$, the fibers of $F_X(Z) \to F_V(Z)$ are $n$-truncated.
	By hypothesis, $F_V$ is the functor of points associated to some $k$-affinoid space, so it takes values in $\rSet$.
	Since $F_V(Z)$ is discrete, it suffices to show that $F_X(Z)$ is $(n+1)$-truncated. This follows directly from \cite[Lemma 2.6.19]{DAG-V}.
	\end{proof}

\begin{prop} \label{prop:geometricity}
	Let $n \ge 1$ and let $X = (\cX, \cO_\cX) \in \dAnk$ be a derived \kanal space such that $\cX$ is $n$-localic.
	Then $\phi(X)$ belongs to $\mathrm{DM}_n$.
\end{prop}

\begin{proof}
	Let $Y = (\cY, \cO_\cY) \in \dAfdk$ be a derived $k$-affinoid space such that $\cO_\cY$ is discrete.
	For every geometric morphism $f\inv \colon \cX \leftrightarrows \cY \colon f_*$ we can use \cite[2.4.4.2]{HTT} to obtain a fiber sequence
	\[ \Map_{\Strloc_{\cTank}(\cY)}(f\inv \cO_\cX, \cO_\cY ) \to \Map_{\dAnk}( Y, X ) \to \Map_{\RTop}( \cY, \cX ) , \]
	where the fiber is taken at $(f\inv, f_*)$.
	
	Since $\cY$ is $1$-localic and $n \ge 1$, it is also $n$-localic.
	Therefore, \cite[Lemma 2.2]{Porta_Comparison_2015} shows that $\Map_{\RTop}( \cY, \cX )$ is $n$-truncated.
		Since $\cO_\cY$ is discrete, we see that $\Map_{\Strloc_{\cTank}(\cY)}( f\inv \cO_\cX, \cO_\cY )$ is $0$-truncated, hence $n$-truncated.
	So
	\[\phi(X)(Y) = \Map_{\dAnk}(Y,X)\]
	is $n$-truncated as well.
	
	Let us now prove that $\phi(X)$ is geometric. Combining \cref{cor:representability_diagonal} and \cref{cor:dAfdk_closed_under_tau}, we see that it is enough to prove that $\phi(X)$ admits an atlas.
	Choose objects $U_i \in \cX$ such that $(\cX_{/U_i}, \cO_\cX |_{U_i})$ is a derived $k$-affinoid space and that the joint morphism $\coprod U_i \to \mathbf 1_\cX$ is an effective epimorphism.
	Put $X_i \coloneqq (\cX_{/U_i}, \cO_\cX |_{U_i})$.
	By functoriality we obtain maps $\phi(X_i) \to \phi(X)$.
	It follows from \cref{lem:affine_site_big_site} that the total morphism $\coprod \phi(X_i) \to \phi(X)$ is an effective epimorphism.
	We are therefore left to prove that $\phi(X_i) \to \phi(X)$ are $(n-1)$-representable by \'etale maps.
	
	First of all, we remark that if $Z \in \dAfdk$, then for any map $\phi(Z) \to \phi(X)$, using full faithfulness of $\phi$, we obtain
	\[ \phi(Z) \times_{\phi(X)} \phi(X_i) \simeq \phi( Z \times_X X_i ) , \]
	and $Z \times_X X_i$ is \'etale over $Z$. Therefore we are reduced to prove that the stacks $\phi(X) \times_{\phi(Z)} \phi(X_i)$ are $(n-1)$-geometric.
	
	We prove this by induction on $n$.
	If $n = 1$, \cref{lem:decreasing_truncated_level} shows that the objects $U_i$ are discrete.
	It follows from \cref{prop:etale_algebraic_spaces} that $\phi(Z) \times_{\phi(X)} \phi(X_i)$ is $0$-geometric.
	Now suppose that $\cX$ is $n$-localic and $n > 1$.
	\cref{lem:decreasing_truncated_level} shows again that the objects $U_i$ are $(n-1)$-truncated.
	Therefore \cite[Lemma 2.3.16]{DAG-V} shows that the underlying \inftopos of $Z \times_X X_i$ is $(n-1)$-localic.
	We conclude by the inductive hypothesis.
\end{proof}

As a consequence of \cref{prop:geometricity}, the functor $\phi \colon \dAnk \to \Sh(\dAfdk, \tauet)$ induces a well defined functor
\[ \phi_n \colon \dAnk^{\le n} \to \mathrm{DM}_n . \]
In order to achieve the proof of \cref{thm:functor_of_points_vs_dAnk}, we are left to show that $\phi_n$ is essentially surjective.

We will need the following elementary observation:

\begin{lem} \label{lem:equivalence_etale_sites}
	Let $X$ be a geometric stack for the geometric context $(\dAfdk, \tauet, \mathbf P\et)$.
	The functor $\trunc \colon X\et \to (\trunc(X))\et$ is an equivalence of sites.
\end{lem}

\begin{proof}
	We prove this by induction on the geometric level of $X$.
	If $X$ is $(-1)$-geometric we can find $Y = (\cY, \cO_\cY) \in \dAfdk$ such that $X \simeq \phi(Y)$.
	Consider the chain of equivalences
	\[ (\RTop(\cTank)_{/Y})\et \simeq (\RTop_{/Y})_{\et} \simeq (\RTop(\cTank)_{/\trunc(Y)})\et  . \]
	We now remark that, if $X \to Y$ is an \'etale map in $\RTop(\cTank)$, then $X$ is a derived \kanal space.
	Moreover, a derived \kanal space belongs to $\dAfdk$ if and only if its truncation does.
	These observations imply that the above equivalence restricts to an equivalence
	\[ Y\et \simeq (\trunc(Y))\et, \]
	thus achieving the proof of the base step of the induction.
	
	Suppose now that $X$ is $n$-geometric and that the statement holds for $(n-1)$-geometric stacks.
	Choose an \'etale $n$-groupoid presentation $U^\bullet$ for $X$.
	This means that $U^\bullet$ is a groupoid object in the $\infty$-category $\Sh(\dAfdk, \tauet)$ such that each $U^m$ is $(n-1)$-geometric and that the map $U^0 \to X$ is $(n-1)$-representable by \'etale maps.
	Since $\trunc$ commutes with products in virtue of \cref{prop:truncation_and_finite_limits} and it takes effective epimorphisms to effective epimorphisms by \cite[7.2.1.14]{HTT}, we see that $V^\bullet \coloneqq \trunc(U^\bullet)$ is a groupoid presentation for $\trunc(X)$.
	
	Now, let $Y \to \trunc(X)$ be an \'etale map. We see that $Y \times_{\trunc(X)} V^\bullet \to V^\bullet$ is an \'etale map (i.e.\ it is a map of groupoids which is \'etale in each degree). By the inductive hypothesis, we obtain a map of simplicial objects $Z^\bullet \to U^\bullet$, such that
	\[ \trunc(Z^\bullet) = Y \times_{\trunc(X)} V^\bullet . \]
	Since $Y \times_{\trunc(X)} V^\bullet$ is a groupoid, so is $Z^\bullet$ (here we use again the equivalence guaranteed by the inductive hypothesis).
	The geometric realization of $Z^\bullet$ provides us with an \'etale map $Z \to X$. Since $\trunc$ preserves effective epimorphisms, we conclude that $\trunc(Z) = Y$.
	This construction is functorial in $Y$, and it provides the inverse to the functor $\trunc$.
\end{proof}

\begin{prop} \label{prop:phi_essentially_surjective}
	The functor $\phi_n \colon \dAnk^{\le n} \to \mathrm{DM}_n$ is essentially surjective.
\end{prop}

\begin{proof}
	Let $X\in\mathrm{DM}_n$.
	By \cref{lem:equivalence_etale_sites}, $X\et$ is equivalent to $(\trunc(X))\et$.
	By hypothesis, $\trunc(X)$ is $n$-truncated.
	Therefore, \cref{prop:over_n_category} shows that the mapping spaces in $(\trunc(X))\et$ are $(n-1)$-truncated.
	In other words, $(\trunc(X))\et$ is equivalent to an $n$-category (cf.\ \cite[2.3.4.18]{HTT}).
	As a consequence, $\Sh(X\et, \tauet)$ is $n$-localic.
	
	Put $\cX \coloneqq \Sh(X\et, \tauet)^\wedge$.
	Consider the composition
	\[ \cTank \times X\et^{\mathrm{op}} \to \dAfdk \times \dAfdk^{\mathrm{op}} \xrightarrow{y} \cS , \]
	where the last functor classifies the Yoneda embedding (cf.\ \cite[\S 5.2.1]{Lurie_Higher_algebra}).
	This induces a well defined functor
	\[ \overline{\cO_\cX} \colon \cTank \to \PSh(X\et) , \]
	which factors through $\Sh(X\et, \tauet)$.
	Let $\cO_\cX$ be its hypercompletion.
	Since the functor $\cTank \to \dAfdk$ preserves products and admissible pullbacks, the same holds for $\cO_\cX$.
	Moreover, \cref{lem:affine_site_big_site} implies that $\cO_\cX$ takes $\tauet$-coverings to effective epimorphisms.
	In other words, $\cO_\cX$ defines a $\cTank$-structure on $\cX$.
	
	If $\{U_i \to X\}$ is an \'etale $n$-atlas of $X$, each $U_i$ defines an object $V_i$ in $\cX$. 
	Unraveling the definitions, we see that the $\cTank$-structured \inftopos $(\cX_{/V_i}, \cO_X |_{V_i})$ is canonically isomorphic to $U_i \in \dAnk$ itself.
	Therefore $X' \coloneqq (\cX, \cO_X)$ is a derived \kanal space.
	
	We are left to prove that $\phi(X') \simeq X$.
	We can proceed by induction on the geometric level $n$ of $X$.
	If $n = -1$, $\phi(X')$ is the functor represented by $X'$, and the same holds for $X$.
	Let now $n \ge 0$.
	Choose an \'etale $n$-atlas $\{U_i \to X\}$ for $X$.
	Set $U \coloneqq \coprod U_i$ and let $U^\bullet$ denote the \v{C}ech nerve of $U \to X$.
	Every map $U^n \to X$ is \'etale.
	In particular, the functor $X\et \to \cS$ sending $Y$ to $\Map_{X\et}(Y, U^n)$ defines an element $V^n \in \Sh(X\et, \tauet)$.
	Using \cref{lem:equivalence_etale_sites}, we see that
	\[ \Map_{X\et}(Y, U^n) \simeq \Map_{\trunc(X)\et}(\trunc(Y), \trunc(U^n)). \]
	Since $t_0(U^n)$ is a geometric stack, we conclude that the above space is truncated.
	In particular, the object $V^n$ is a truncated object in $\Sh(X\et, \tauet)$, so it is hypercomplete.
	In other words, $V^n$ belongs to $\cX$.
	We can therefore identify $\Sh(U^n\et, \tauet)^\wedge$ with $\cX_{/V^n}$.
	The universal property of \'etale morphisms (cf.\ \cite[Remark 2.3.4]{DAG-V}) shows that we can arrange the $V^n$s into a simplicial object $V^\bullet$ in $\cX$, whose geometrical realization coincides with $\mathbf 1_\cX$.
	The inductive hypothesis shows that $\phi( \cX_{/V^\bullet}, \cO_X |_V^\bullet ) \simeq U^\bullet$ as simplicial objects in $\Sh(\dAfdk, \tauet)$.
	Since $\phi$ commutes with \v{C}ech nerves of \'etale maps and their realizations (in virtue of \cref{lem:affine_site_big_site}), we conclude that $\phi(X')$ is equivalent to $X$ itself.
\end{proof}

The proof of \cref{thm:functor_of_points_vs_dAnk} is now achieved.

\section{Appendices}

\subsection{Complements on overcategories}

The goal of this subsection is to provide a proof of the following basic result, for which we do not know a reference: if $(\cC, \tau)$ is a Grothendieck site and $\cC$ is a $1$-category, then for every $n$-truncated sheaf $X \in \PSh(\cC)$, the overcategory $\cC_{/X}$ is an $(n-1)$-category.
The proof relies on the following lemma:

\begin{lem} \label{lem:fiber_sequence_for_over}
	Let $\cC$ be an $\infty$-category.
	Let $X \in \cC$ be an object and let $f \colon U \to X$, $g \colon V \to X$ be two $1$-morphisms of $\cC$ viewed as objects of $\cC_{/X}$.
	For every morphism $h \colon U \to V$ in $\cC$, choose a $2$-simplex $\sigma \colon \Delta^2 \to \cC$ extending the morphism $\Lambda^2_1 \to \cC$ classified by $h$ and $g$.
	Put $f' \coloneqq d_1(\sigma)$.
	Then we have a fiber sequence
	\[ \mathrm{Path}_{\Map_\cC(U,X)}(f,f') \to \Map_{\cC_{/X}}(f,g) \to \Map_{\cC}(U,V). \]
\end{lem}

\begin{proof}
	It follows from \cite[Proposition 2.1.2.1]{HTT} that the canonical map $p \colon \cC_{/X} \to \cC$ is a right fibration. In particular, it is a Cartesian fibration where every edge of $\cC_{/X}$ is $p$-Cartesian.
	The $2$-simplex $\sigma \colon \Delta^2 \to \cC$ can be viewed as an edge of $\cC_{/X}$.
	Reviewing the Kan complex $\Map_{\cC}(U,X)$ as an $\infty$-category, we have a canonical equivalence $\mathrm{Path}_{\Map_\cC(U,X)}(f,f') \simeq \Map_{\Map_{\cC}(U,X)}(f,f')$.
	The conclusion follows at this point from \cite[Proposition 2.4.4.2]{HTT}.
\end{proof}

\begin{prop} \label{prop:over_n_category}
	Let $\cC$ be an $\infty$-category.
	Let $X \in \cC$ be an $n$-truncated object.
	Let $f \colon U \to X$ and $g \colon V \to X$ be two morphisms viewed as objects in $\cC_{/X}$.
	If $V$ is $m$-truncated with $m < n$, then $\Map_{\cC_{/X}}(U, V)$ is $(n-1)$-truncated.
\end{prop}

\begin{proof}
	Choosing $f'$ as in \cref{lem:fiber_sequence_for_over}, we obtain a fiber sequence
	\[ \mathrm{Path}_{\Map_\cC(U,X)}(f,f') \to \Map_{\cC_{/X}}(f,g) \to \Map_{\cC}(U,V). \]
	Now, $\Map_{\cC}(U,V)$ is $m$-truncated by hypothesis.
	On the other hand, we have a pullback diagram
	\[ \begin{tikzcd}
		\mathrm{Path}_{\Map_{\cC}(U,X)}(f,f') \arrow{r} \arrow{d} & \{*\} \arrow{d}{f'} \\
		\{*\} \arrow{r}{f} & \Map_{\cC}(U,X).
	\end{tikzcd} \]
	Therefore $\mathrm{Path}_{\Map_\cC(U,X)}(f,f')$ fits in the pullback diagram
	\[ \begin{tikzcd}
		\mathrm{Path}_{\Map_\cC(U,X)}(f,f') \arrow{d} \arrow{r} & \Map_{\cC}(U,X) \arrow{d}{\Delta} \\
		\{*\} \arrow{r}{(f,f')} & \Map_{\cC}(U,X) \times \Map_{\cC}(U,X) .
	\end{tikzcd} \]
	Since $X$ is $n$-truncated, it follows that $\Map_{\cC}(U,X)$ is $n$-truncated.
	Therefore, \cite[5.5.6.15]{HTT} shows that $\Delta$ is $(n-1)$-truncated.
	We deduce that $\mathrm{Path}_{\Map_\cC(U,X)}(f,f')$ is $(n-1)$-truncated.
	Thus the fiber sequence of \cref{lem:fiber_sequence_for_over} implies that $\Map_{\cC_{/X}}(f,g)$ is $(n-1)$-truncated as well, completing the proof.
\end{proof}

\subsection{Complements on geometric stacks}

\begin{defin}\label{def:C_closed_under_tau}
	Let $(\cC,\tau)$ be an \infsite.
	The \infcat $\cC$ is said to be \emph{closed under $\tau$-descent} if for any morphism from a sheaf $X$ to a representable sheaf $Y$, any $\tau$-covering $\{Y_i\to Y\}$, the representability of $X\times_Y Y_i$ for every $i$ implies the representability of $X$.
\end{defin}

We need the following converse to \cite[Corollary 2.12]{Porta_Yu_Higher_analytic_stacks_2014}:

\begin{lem} \label{lem:factorizing_property_effective_epimorphisms}
	Let $(\cC, \tau)$ be a subcanonical $\infty$-site.
	Let $F \to G$ be an effective epimorphism in $\Sh(\cC, \tau)$.
	For any object $X \in \cC$ and any morphism $h_X \to G$, there exists a $\tau$-covering $\{U_i \to X\}$ such that the composite morphisms $h_{U_i} \to h_X \to G$ factor as
	\[ \begin{tikzcd}
		{} & {} & F \arrow{d} \\
		h_{U_i} \arrow[dashed]{urr} \arrow{r} & h_X \arrow{r} & G .
	\end{tikzcd} \]
\end{lem}

\begin{proof}
	Using \cite[Proposition 2.11]{Porta_Yu_Higher_analytic_stacks_2014} we see that the morphism $\pi_0(F) \to \pi_0(G)$ is an effective epimorphism of sheaves of sets.
	In particular, there exists a covering $\{V_j \to X\}$ such that the composite morphisms $\pi_0( h_{V_j}) \to \pi_0(h_X) \to \pi_0(G)$ factor through $\pi_0(F)$.
	Since $\pi_0(F)$ is by definition the sheafification of the presheaf $Y \mapsto \pi_0(F(Y))$ and since
	\[ \Map_{\Sh(\cC, \tau)}( \pi_0(h_{V_j}), \pi_0(F) ) \simeq \Map_{\Sh(\cC, \tau)}(h_{V_j}, \pi_0(F)) \simeq \pi_0(F) (V_j) , \]
	we can find a $\tau$-covering $\{U_{ij} \to V_j\}$ such that every composite morphism $h_{U_{ij}} \to h_{V_j} \to \pi_0(F)$ factors through $F \to \pi_0(F)$.
	Finally, again since $\pi_0(G)$ is the sheafification of the presheaf $Y \mapsto \pi_0(G(Y))$, we can further refine the covering such that the morphisms $h_{U_{ij}} \to F$ are homotopic to the compositions $h_{U_{ij}} \to h_X \to G$. This completes the proof.
\end{proof}

\begin{prop} \label{prop:geometric_stacks_closed_under_tau}
	Let $(\cC, \tau, \bP)$ be a geometric context in the sense of \cite{Porta_Yu_Higher_analytic_stacks_2014}.
	Assume that $\cC$ is closed under $\tau$-descent.
	Then the class of $n$-representable morphisms is closed under $\tau$-descent, in the sense that for any morphism $f \colon X \to Y$ with $Y$ a $n$-geometric stack, if there exists an $n$-atlas $\{U_i\}$ of $Y$ such that $X \times_Y U_i$ is $n$-geometric for every $i$, then $F$ is $n$-geometric as well.
\end{prop}

\begin{proof}
	The proof goes by induction on the geometric level $n$.
	When $n = -1$, this holds because $\cC$ is closed under $\tau$-descent.
	Let now $n \ge 0$. Let $\{U_i\}$ be an $n$-atlas of $Y$ such that $X_i \coloneqq X \times_Y U_i$ is $n$-geometric for every $i$.
	Choose an $n$-atlas $\{V_{ij}\}$ of $X \times_Y U_i$. The compositions $V_{ij} \to X_i \to X$ provide an $n$-atlas of $X$.
	We are therefore left to prove that the diagonal of $X$ is $(n-1)$-representable.
	Let $V \coloneqq \coprod V_{ij}$ be the $n$-atlas of $X$ introduced above. By construction, the map $V \to X$ is $(n-1)$-representable.
	It follows that the induced map $V \times_X V \to V$ is $(n-1)$-representable as well.
	Since $V$ is a disjoint union of $(-1)$-representable stacks, it follows that $V \times_X V$ is $(n-1)$-geometric.
	Observe now that $V \times V \to X \times X$ is an effective epimorphism.
	Therefore for every morphism $S \to X \times X$ from a $(-1)$-representable stack $S$, by \cref{lem:factorizing_property_effective_epimorphisms}, we can choose a $\tau$-covering $S_i \to S$ such that the composite map $S_i \to S \to X \times X$ factors as
	\[ \begin{tikzcd}
		{} & {} & V \times V \arrow{d} \\
		S_i \arrow[dashed]{urr} \arrow{r} & S \arrow{r} & X \times X .
	\end{tikzcd} \]
	In order to prove that the diagonal $\Delta_X \colon X \to X \times X$ is $(n-1)$-representable, we have to show that $S \times_{X \times X} X$ is $(n-1)$-geometric.
	Using the induction hypothesis, it suffices to show that each stack $S_i \times_{X \times X} X$ is $(n-1)$-geometric.
	Note that this stack fits in the following diagram of cartesian squares:
	\[ \begin{tikzcd}
		S_i \times_{X \times X} X \arrow{r} \arrow{d} & V \times_X V \arrow{r} \arrow{d} & X \arrow{d} \\
		S_i \arrow{r} & V \times V \arrow{r} & X \times X .
	\end{tikzcd} \]
	Since $V \times V$, $V \times_X V$ and $S_i$ are $(n-1)$-geometric, it follows that the same goes for $S_i \times_{X \times X} X$, thus completing the proof.
\end{proof}

\begin{cor} \label{cor:representability_diagonal}
	Let $(\cC, \tau, \bP)$ be a geometric context and assume that $\cC$ is closed under $\tau$-descent.
	If $X \in \Sh(\cC, \tau)$ admits an $n$-atlas, then it is $n$-geometric.
\end{cor}

\begin{proof}
	We have to prove that the diagonal of $X$ is $(n-1)$-representable.
	Let $V \to X$ be an $n$-atlas.
	Then $V \times V \to X \times X$ is an $n$-atlas for $X\times X$.
	By \cref{lem:factorizing_property_effective_epimorphisms}, for any map $S \to X \times X$, with $S$ being representable, we can find a $\tau$-covering $\{ S_i \to S\}$ such that the composite maps $S_i \to S \to X \times X$ factor through $V \times V$.
	Using \cref{prop:geometric_stacks_closed_under_tau}, we are reduced to prove that each $S_i \times_{X \times X} X$ is $(n-1)$-geometric.
	Consider the diagram
	\[ \begin{tikzcd}
		S_i \times_{X \times X} X \arrow{r} \arrow{d} & V \times_X V \arrow{r} \arrow{d} & X \arrow{d} \\
		S_i \arrow{r} & V \times V \arrow{r} & X \times X .
	\end{tikzcd} \]
	The right and the outer squares are pullback diagrams by construction.
	Therefore, so is the left square.
	Now we conclude from the fact that $V \times_X V$ is $(n-1)$-geometric.
\end{proof}

\begin{prop} \label{prop:Afdk_closed_under_tau}
	The category $\Afd_k$ of $k$-affinoid spaces is closed under $\tauet$-descent.
\end{prop}
\begin{proof}
	Let $Y$ be a $k$-affinoid space and let $f \colon F \to h_Y$ be a morphism in $\Sh(\Afd_k, \tauet)$.
	Let $\{Y_i \to Y\}_{i \in I}$ be a finite étale covering in the category $\Afd_k$.
	Assume that for every index $i$, the fiber product $h_{Y_i} \times_{h_Y} F$ is representable by $X_i \in \Afd_k$.
	Put $Y^0 \coloneqq \coprod_{i \in I} Y_i$ and let $Y^\bullet$ be the \v{C}ech nerve of $Y^0 \to Y$.
	By assumption, we see that for every integer $n$, $h_{Y^n} \times_{h_Y} F$ is representable.
	Choose $X^n \in \Afd_k$ such that $h_{Y^n} \times_{h_Y} F \simeq h_{X^n}$.
	Fully faithfulness of the Yoneda embedding implies that we can arrange the objects $X^n$ into a simplicial object $X^\bullet$ in $\Afd_k$.
		Let $\bDelta_s$ be the semisimplicial category.
	It follows from \cite[6.5.3.7]{HTT} that the inclusion $\bDelta_s\op \subset \bDelta\op$ is cofinal.
	Let $\bDelta_{s, \le 2}$ be the full subcategory of $\bDelta_s$ spanned by the objects $[0]$, $[1]$ and $[2]$.
	The inclusion $\bDelta_{s, \le 2}\op \subset \bDelta_s\op$ is $1$-cofinal, in the sense that for every $[n] \in \bDelta_s$, the undercategory $(\bDelta_{s,\le 2}\op)_{[n]/}$ is nonempty and connected.
	Let $j \colon \bDelta_{s, \le 2}\op \hookrightarrow \bDelta$ be the composite functor.
	Since $\Afd_k$ is a $1$-category, we see that $X^\bullet$ admits a colimit if and only if $X^\bullet_{s, \le 2} \coloneqq X^\bullet \circ j$ does.
	The latter statement is true because $\Afd_k$ admits finite colimits.
	
	Let $X$ be the colimit of $X^\bullet$ and let $g \colon X \to Y$ be the canonical map.
	We claim that $X^n \simeq Y^n \times_Y X$.
	To prove this, it is enough to show that $X^0 \simeq Y^0 \times_Y X$.
	We first remark that if the map $X^0 \to Y^0$ is a closed immersion, then the statement follows directly from the fpqc descent of coherent sheaves (cf.\ \cite{Conrad_Descent_for_coherent_2003}).
	In the general case, we factor $X^0 \to Y^0$ as $X^0 \hookrightarrow \bD^N_{Y^0} \to Y^0$, where $\bD^N_{Y^0}$ denotes the $N$-dimensional unit polydisc over $Y^0$ and the first arrow is a closed immersion.
	Observe that the colimit of $\bD^N_{Y^\bullet}$ is $\bD^N_Y$, and that $\bD^N_{Y^0} \simeq Y^0 \times_Y \bD^N_Y$.
	Consider the following diagram:
	\[ \begin{tikzcd}
		X^0 \arrow{r} \arrow{d} & X \arrow{d} \\
		\bD^N_{Y^0} \arrow{r} \arrow{d} & \bD^N_Y \arrow{d} \\
		Y^0 \arrow{r} & Y .
	\end{tikzcd} \]
	Since $X_0 \hookrightarrow \bD^N_{Y^0}$ is a closed immersion, we see that the top square is a pullback.
	Moreover, we remarked that the bottom square is also a pullback.
	Hence so is the outer square, completing the proof of the claim.
	
	As a consequence, we see that $X^\bullet$ is the \v{C}ech nerve of the étale covering $X^0 \to X$.
	In particular, in $\Sh(\Afd_k, \tauet)$ we have
	\[ h_X \simeq | h_{X^\bullet} | ,\]
	where $\abs{\cdot}$ denotes the geometric realization.
	Finally, since $\Sh(\Afd_k, \tauet)$ is an \inftopos, we obtain:
	\[ 	h_X \simeq | h_{X^\bullet} | \simeq |h_{Y^\bullet} \times_{h_Y} F| \simeq |h_{Y^\bullet}| \times_{h_Y} F \simeq F. \]
	This shows that $F$ is representable, thus completing the proof.
	\end{proof}

\begin{cor} \label{cor:dAfdk_closed_under_tau}
	The category $\dAfdk$ of derived $k$-affinoid spaces is closed under $\tauet$-descent.
\end{cor}

\begin{proof}
	Let $Y = (\cY, \cO_\cY)$ be a derived $k$-affinoid space.
	Let $F \to h_Y$ be a morphism in $\Sh(\dAfdk, \tauet)$.
	Assume there exists an \'etale covering $Y_i \to Y$ such that each base change $h_{Y_i} \times_{h_Y} F$ is representable by a derived $k$-affinoid space $X_i$.
	In particular, $\trunc(h_{Y_i} \times_{h_Y} F) \simeq \trunc(h_{Y_i}) \times_{\trunc(h_Y)} \trunc(F)$ is representable by an ordinary $k$-affinoid space $\trunc(X_i)$.
	It follows from \cref{prop:Afdk_closed_under_tau} that $\trunc(F)$ is representable by an ordinary $k$-affinoid space $Z$.
	
	Form the \v{C}ech nerve $G^\bullet$ of $\coprod h_{Y_i} \times_{h_Y} F \to F$.
	By hypothesis, each $G^n$ is a disjoint union of derived $k$-affinoid spaces.
	Since $\phi$ is fully faithful, we obtain in this way a simplicial object $X^\bullet$ in $\dAnk$, such that all the face maps are \'etale morphisms.
	It follows from \cite[Proposition 2.3.5]{DAG-V} that this simplicial object admits a colimit $Y$ in $\RTop(\cTank)$ and that the canonical maps $X^n \to X$ are \'etale.
	This shows that we can cover $X$ with derived $k$-affinoid spaces.
	In particular, $X$ is a derived \kanal space.
	
	We are left to prove that $X$ is a derived $k$-affinoid space.
	Observe that the maps $\trunc(X^n) \to \trunc(X)$ are \'etale.
	Since $X$ (resp.\ $X^n$) and $\trunc(X)$ (resp.\ $\trunc(X^n)$) share the same underlying \inftopos, we can use the statement (3') in the proof of \cite[Proposition 2.3.5]{DAG-V} to conclude that the colimit of $\trunc(X^\bullet)$ in $\RTop(\cTank)$ is $\trunc(X)$.
	On the other hand, since $\trunc$ commutes with limits, we can further identify $\phi(\trunc(X^\bullet))$ with the \v{C}ech nerve of the map $\coprod \trunc(h_{Y_i} \times_{h_Y} F) \to \trunc(F) \simeq \phi(Z)$.
	It follows that $\trunc(X) \simeq Z$ in $\dAnk$. This shows that $X$ is a derived $k$-affinoid space, and $\phi(X) \simeq F$.
	The proof is thus complete.
\end{proof}

\bibliographystyle{plain}
\bibliography{dahema}

\end{document}